\documentclass[11pt,a4paper,twoside]{article}

\usepackage[utf8]{inputenc}
\usepackage{amssymb,amsthm}
\usepackage[english]{babel}
\usepackage[leqno]{amsmath} 
\usepackage{mathtools}
\usepackage{bbm}
\usepackage[twoside, bindingoffset=5mm, left=15mm, right=25mm, top=28mm, bottom=28mm,a4paper]{geometry} 
\usepackage{float}
\usepackage{fancyhdr}
\usepackage[breaklinks,plainpages=false,pdfpagelabels=true]{hyperref}
\usepackage{bookmark}
\usepackage{tikz}
\usepackage{enumerate}
\usepackage{amsmath}
\usepackage{newpxtext}
\usepackage{keyval}
\usepackage{constants}
\usepackage{sectsty}

\newcommand{\R}{\mathbbm {R}}
\newcommand{\Rr}{\mathcal {R}}
\newcommand{\Cc}{{C}}
\newcommand{\M}{\mathcal {M}}
\newcommand{\N}{\mathbbm {N}}

\newcommand{\W}{\mathcal {W}}
\newcommand{\Kk}{\mathcal {K}}
\newcommand{\diff}{\operatorname{d}\!}
\newcommand{\Hd}{\mathcal{H}}

\renewcommand{\div}{\operatorname{div}}

\newcommand {\cb}[1] {\tikz[baseline=(char.base)]{
            \node[shape=circle,draw,inner sep=2pt] (char) {\itshape{\textbf{#1}}};}}
\newcommand {\cbm}[1] {\tikz[baseline=(char.base)]{
            \node[shape=circle,draw,inner sep=2pt] (char) {\scriptsize\itshape{\textbf{#1}}};}}

\newtheoremstyle{plain}	
	{1em}{1em}	
	{\itshape}	
	{}	
	{\bfseries}	
	{}{\newline}	
	{\thmnumber{#2 }\thmname{#1}\thmnote{ (#3)}}

\newtheorem{thm}{Theorem}
\newtheorem{corollary}[thm]{Corollary}
\newtheorem{lemma}[thm]{Lemma}

\title{The Willmore Flow of Graphs with Boundary Data: Low-Regularity Initial Data and Global Convergence}
\author{Boris Gulyak\\
    \small  Otto-von-Guericke-Universität,\\
   \small  Fakultät für Mathematik,\\
   \small   Postfach 4120,\\
   \small  39016 Magdeburg,\\
   \small Germany\\
    \small \texttt{bgulyak@ovgu.de}}

\begin{document}
\maketitle

\begin{abstract}
We study the Willmore flow for graphs over a bounded domain in $\R^2$ with Dirichlet (clamped) boundary conditions, a still little-studied setting that also serves as a prototype for higher-order flows with fixed boundary data.
We develop a low-regularity theory that avoids the classical fourth-order compatibility condition at $t=0$.
Combining a reformulation of the graphical equation, which isolates the quasilinear fourth-order principal part from the lower-order terms, with time-weighted parabolic Hölder spaces, we prove short-time existence for initial data in $C^{1+\alpha}(\overline\Omega)$ and, under a smallness assumption, also for Lipschitz data in $C^{0,1}(\overline\Omega)$, even when the initial Willmore energy is not defined. In the Hölder regime, uniqueness is obtained.

In the small-data Lipschitz regime, we also prove global existence, uniform gradient bounds, and exponential convergence to a stationary solution of the elliptic Willmore equation with the prescribed boundary data.
A key ingredient is an $L^2$-smallness criterion for graphical surfaces with small Willmore energy and small boundary values.
The approach is mainly analytic and extends naturally to related higher-order geometric flows and other boundary conditions.
\end{abstract}



\maketitle

\section{Introduction}

For a smooth immersed surface $S\subset\R^3$, the Willmore energy is defined by
\begin{align}
\W(S):=\frac14\int_S H^2\,dS ,
\label{eq:meinewilldef}
\end{align}
where $H=\kappa_1+\kappa_2$ denotes the mean curvature, i.e., the sum of the principal curvatures.
Historically, quadratic bending energies of this type already appear in the nineteenth-century work of Germain and later entered conformal surface geometry through Thomsen. The functional was subsequently popularized by Willmore and also plays a central role in the Helfrich model for elastic lipid bilayers. For modern surveys on the variational and geometric aspects of the Willmore functional, we refer to
\cite{germain1821recherches,thomsen1924grundlagen,willmore1996riemannian,helfrich1973elastic,kuwert2012willmore}.

Critical points of $\W$ are called \emph{Willmore surfaces}.
Minimal surfaces are particular examples, since $H\equiv0$ implies $\W(S)=0$.
A fundamental feature of the theory is the conformal invariance of the integrand
\(
\tfrac14 H^2-K,
\)
where $K=\kappa_1\kappa_2$ denotes the Gaussian curvature.
For closed surfaces, the Gauss-Bonnet theorem implies that the Willmore energy is invariant under Möbius transformations up to the topological constant $2\pi\chi(S)$.
This conformal structure underlies many deep results in the theory, including the resolution of the Willmore conjecture by Marques and Neves~\cite{marques2014min}.
For existence results on Willmore minimizers, we refer to Simon~\cite{simon1993existence}.

A natural way to search for critical points is to study the
$L^2$-gradient flow of the Willmore functional, the \emph{Willmore flow}.
Given a sufficiently smooth one-parameter family of immersions
\[
f:\Sigma\times[0,T)\to\R^3,
\]
with pull-back $g = f^\ast\langle\,\cdot\,, \cdot\,\rangle$ of the Euclidean
metric in $\R^3$, the Willmore flow equation reads
\begin{align}
\langle \partial_t f,N\rangle
&=
-\big(
\Delta_{g}H
+
2H
(
\tfrac14 H^2-K
)
\big),
\label{eq:willmoreflowgen}
\end{align}
where where $\Delta_g$ is the Laplace–Beltrami operator on $(\Sigma, g)$ and $N(t,\,\cdot\,)$ is a choice of unit normal along $f(\Sigma,t)$.
This is a fourth-order quasilinear geometric evolution equation, and its
analysis is substantially subtler than that of second-order flows such as
mean curvature flow.

For closed surfaces, the flow was studied in particular by Kuwert and
Schätzle \cite{kuwert2001willmore}, while in the graphical setting Dziuk and Deckelnick \cite{dziuk2006error} derived and
analyzed the corresponding graph equation.
Rough-data well-posedness for entire graphs was obtained by Koch and Lamm,
\cite{koch2012geometric}. Further,
Simonett \cite{simonett2001willmore} proved global existence and
exponential convergence to a round sphere for initial data sufficiently
close to spheres in the $C^{2+\alpha}$-topology.
Kuwert and Schätzle developed a comprehensive theory in codimension one:
in \cite{kuwert2001willmore,kuwert2004removability} they showed that for
initial energy $\W\le 8\pi$, the flow of immersed spheres exists globally
and converges to a round sphere, while in \cite{kuwert2002gradient} they
established lower bounds on the lifespan of smooth solutions in terms of
curvature concentration.
More generally, Chill, Fasangová and Schätzle \cite{chill2009willmore}
proved global existence and convergence to a local minimizer for initial
surfaces that are $W^{2,2}\cap C^1$-close to a $C^2$ Willmore minimizer.

Numerical experiments by Mayer and Simonett \cite{mayer2002numerical}
suggest the formation of singularities once smallness conditions are
violated.
Moreover, Blatt \cite{blatt2009singular} constructed examples with
Willmore energy arbitrarily close to $8\pi$ that do not converge under
the flow, exhibiting either unbounded diameter or concentration of
curvature.

To understand the possible phenomena, one can restrict attention to surfaces with symmetry or projectability properties, 
hoping to extract geometric and analytic information on the Willmore energy, its minimizers, and flow. The weaker these constraints, the harder the analysis becomes.

In this paper we consider the \emph{graphical} Willmore flow over a fixed
sufficiently smooth bounded planar domain $\Omega\subset\R^2$ with $\nu$ denoting the outward unit normal field on $\partial\Omega$.
Thus surfaces are represented as graphs for each fixed time $t\ge 0$
\[
\Gamma(u)=\{(x,u(x,t)):x\in\overline\Omega\},
\]
and the geometric evolution becomes a scalar fourth-order parabolic equation
for the sufficiently smooth height function $u\colon \overline\Omega\times[0,T]\to \R$ for some $T>0$.
We impose for all times fixed boundary position and fixed tangent half-planes along
$\partial\Omega$, which in graph coordinates correspond to the clamped
boundary conditions
\[
u=g_0,
\qquad
\partial_\nu u=g_1
\qquad\text{on }\partial\Omega 
\]
for all times and some sufficiently smooth fixed (given) functions $g_0,g_1\colon \partial\Omega\to\R$.

This is the Dirichlet, or clamped, boundary value problem for the Willmore flow.
By contrast, in the Navier problem one prescribes the height together with the mean curvature on the boundary. 
Boundary value problems for Willmore surfaces and the Willmore flow evolution become more
involved and much less is known when compared with closed surfaces. One of the reasons is that
we cannot directly apply scaling arguments. The other is that in general, no a-priori bounds are
known neither for the solution of the Willmore equation nor for Willmore energy-minimizing
sequences or minimizers.

In \cite{nitsche1993boundary}, various boundary conditions together with
corresponding existence results for small data in strong topologies are
discussed; see also \cite{barrett2017} for further boundary conditions.
For the stationary problem we refer to Schätzle~\cite{schatzle2010willmore}.
Beyond these general results, the Dirichlet problem has been studied
extensively in the class of surfaces of revolution:
Dall'Acqua, Deckelnick, and Grunau proved existence of classical
solutions for symmetric Dirichlet data
\cite{dall2008classical},
Dall'Acqua, Fröhlich, Grunau, and Schieweck treated arbitrary symmetric
Dirichlet boundary data
\cite{dall2011symmetric},
and Eichmann and Grunau established existence for non-symmetric
Dirichlet data under an explicit energy condition
\cite{eichmann2019existence}.
For rigidity and uniqueness in the homogeneous Dirichlet case see
Dall'Acqua~\cite{dall2012uniqueness}.
A broader overview of boundary value problems for the Willmore functional
is given in Grunau's survey \cite{grunau2018boundary}.
Graphical minimizers under Dirichlet or Navier conditions were studied by
Deckelnick, Grunau, and Röger~\cite{deckelnick2017minimising}.
Related results for the elliptic problem with low-regularity boundary and
Dirichlet data can be found in our recent work~\cite{gulyak2026}.

On the parabolic side, much less is known. We refer to our previous work \cite{gulyak2017willmore}. 
Recent work of Schlierf \cite{schlierf2024convergence}
studies the Willmore flow with Dirichlet boundary conditions in a
rotationally symmetric setting and proves global existence and
convergence below an explicit energy threshold.
See also Schlierf~\cite{schlierf2025global} for a complementary global
existence result based on a Simon-Li-Yau type inequality.

\subsection{The Initial-Boundary Value Problem}

In the Dirichlet case, for a given initial datum $u_0:\overline\Omega\to\R$, we consider the following 
initial-boundary value problem
\begin{align}
\left\{
\begin{aligned}
\partial_t u
&=
-\,Q\left\{
\Delta_{\Gamma(u)}H
+2H\left(\tfrac14 H^2-\Kk\right)
\right\}
&& \text{in } \Omega\times(0,T],\\
u(x,t)&=g_0(x),
&& (x,t)\in\partial\Omega\times[0,T],\\
\partial_\nu u(x,t)&=g_1(x),
&& (x,t)\in\partial\Omega\times[0,T],\\
u(x,0)&=u_0(x),
&& x\in\overline\Omega.
\end{aligned}
\right.
\tag{WF}\label{Willmore}
\end{align}
where $Q=\sqrt{1+|\nabla u|^2}$ is the induced surface area element.

Since we have prescribed boundary values, to get continuity for $t \searrow 0$ we have to make sure that the initial and boundary values are consistent with each other. 
At time $t=0$ this requires the compatibility conditions
\begin{align}
 g_0=&\	u_0, \hspace{1em} 
g_1	= \partial_\nu u_{0},  \hspace{1em} \text { on }  \partial \Omega \label{eq:wkb} \tag{CC}
\end{align}
for the solution $u$ to be at least bounded in $C^{0,1}(\overline{\Omega}).$

There are \emph{two main analytical difficulties}.
First, the equation is a fourth-order quasilinear equation, so a direct Schauder approach would naturally require, see \cite[Remark 2.3]{deckelnick2015C1}
\[
u_0\in C^{4+\alpha}(\overline\Omega),
\qquad
g\in C^{4+\alpha}(\overline\Omega).
\]
Moreover, classical $C^{4+\alpha}$ theory implies that the solution remains $C^{4+\alpha}$-bounded as $t \searrow 0$ only if the initial datum satisfies the fourth-order compatibility condition
\begin{align}
    0
    = \Delta_{\Gamma(u_{0})} H(u_0)
    + 2H\left(\tfrac{1}{4}H^2-\Kk\right)(u_0)
    \quad \text{on } \partial \Omega,
    \label{eq:comp4} \tag{uCC}
\end{align}
see \cite{lunardi1992semigroup}.
This condition reflects the fact that the height $u$ is fixed on $\partial\Omega$, and hence $\partial_t u = 0$ on the boundary.
It is automatically satisfied for all $t>0$ provided $u(\,\cdot\,,t)\in C^4(\overline\Omega)$.

In fact, it is desirable to avoid \eqref{eq:comp4}, since higher-order
compatibility conditions are a standard but restrictive feature of the
classical theory for parabolic boundary value problems and are difficult
to enforce in numerical applications. From the parabolic point of view, however, one expects an instantaneous
smoothing effect away from the initial time.
In geometric fourth-order flows with boundary, such compatibility
conditions appear explicitly, for instance, in Menzel's treatment of
Willmore type evolutions \cite{menzel2021boundary}, in Diana's recent
work on elastic flow with partial free boundary \cite{Diana2024}, in the
Willmore flow of planar networks \cite{garcke2019willmore}, and in the
surface diffusion setting \cite{garcke2021nonlinear}.
A broader discussion can also be found in the survey
\cite[Remark 2.8]{mantegazza2021survey}.

Second problem is that, the leading operator is uniformly elliptic only as long as
the gradient remains controlled. If $\|\nabla u\|_{L^\infty(\Omega)}$
is unbounded, the graph structure may be lost.
For global existence, some smallness condition on the data is therefore
unavoidable, and the fixed boundary prevents a direct scaling reduction.
Also, scaling is not directly useful, since boundary is fixed.

On the other hand, the graphical setting provides strong a priori
geometric control. In particular, Deckelnick, Grunau and Röger
\cite{deckelnick2017minimising} proved area and diameter bounds for graphs
in terms of the Willmore energy and the boundary data, a feature that
fails for general immersed surfaces with fixed boundary, see also \cite{gulyak2014willmore, gulyak2024boundary}
\begin{align*}
     \sup_{x\in\Omega}|u| + \| Q \|_{L^1(\Omega)} 
     &\  \le 
C\left(\mathcal{H}^2(\Omega),\mathcal{H}^1(\partial\Omega) ,\|g\|_{W^{1,1}(\partial\Omega)} , \W(u), |\Omega|\right).
\end{align*}
These estimates will later be used to derive $L^2$-smallness
and global existence in the small-data regime. Moreover, rewriting the Willmore flow equation in terms of the height
function transforms the geometric evolution into a fourth-order
quasilinear parabolic equation, making the tools of parabolic PDE theory
available.

The associated numerical $C^1$-finite element method for this problem was provided by Deckelnick, Katz, und Schieweck in \cite{deckelnick2015C1} with quasioptimal error bounds in Sobolev norms for the solution and its time derivative.


\subsection{Goals and Methods}

Our main goals are threefold.
First, we prove short-time existence in weighted parabolic Hölder spaces for low-regularity initial data, covering both the Hölder regime $u_0\in C^{1+\alpha}(\overline\Omega)$ and the Lipschitz regime $u_0\in C^{0,1}(\overline\Omega)$.
Second, under a smallness assumption on the Lipschitz norm of the data, we establish global existence and convergence of the flow to a stationary graph solving the elliptic Willmore equation with the prescribed boundary data.
Third, we derive geometric and analytic a priori estimates specific to the graphical setting, including an $L^2$-smallness criterion that is crucial for the global theory.

\emph{In particular, we achieve:}
\begin{itemize}
    \item[\cbm 1] the removal of the compatibility condition \eqref{eq:comp4},
    \item[\cbm 2] a reduction of the regularity assumptions for local existence from $C^{4+\alpha}$ to $C^{1+\alpha}$,
    \item[\cbm 3] global existence and convergence to a critical point for Lipschitz initial data with sufficiently small $\|\nabla u_0\|_{L^\infty(\Omega)}$.
\end{itemize}

We emphasize that the methods are primarily analytic rather than geometric in nature.
In particular, they can be applied to related flows, such as the graphical \emph{Helfrich flow} or the \emph{surface diffusion flow}, to reduce the regularity assumptions on the initial data to the Lipschitz level.
Moreover, the approach may be extended to non-graphical settings, where $u$ represents a height function with respect to a reference surface.

By introducing weighted parabolic Hölder spaces, we provide a general framework to avoid higher-order compatibility conditions for higher-order geometric flows, provided that the complementarity conditions for the boundary operators on $\partial\Omega\times[0,T]$ are satisfied, see \cite{solonnikov1965boundary}.
These conditions follow, in particular, from the validity of the Lopatinskii-Shapiro condition, see \cite{Eidelman1998parabolic, menzel2021boundary}.
Consequently, our results can be extended to other boundary conditions, such as the Navier case or problems where $(u,\Delta u)$ are prescribed on the boundary.   

From the perspective of the Sobolev theory developed by Solonnikov, \cite{solonnikov1965boundary} 
the natural regularity of the initial data is considerably higher.
Indeed, for the biharmonic parabolic operator ($2m=4$), the classical
$L^p$-theory yields that in the two-dimensional case  every solutions in the space
$W^{1,4}_p(Q_T)$ 
for $p>3$ that the initial data must be at least of class
$C^{2+\alpha}$, by embedding theorems, see \cite[Thm. B.20, Prop. B.35]{menzel2021boundary}.
This comparatively high regularity requirement is a major drawback of the
Sobolev approach, as it excludes, for instance, Lipschitz or merely
$C^{1+\alpha}$ initial data.
In contrast, our method based on weighted Hölder spaces allows for the
treatment of significantly rougher initial data.



The \emph{main strategy and methods} are as follows.
\begin{itemize}
\item[\cbm 1]
First, we formulate the Willmore flow equation in the graphical setting in a way that makes its divergence structure accessible.
For the local theory, it is convenient to separate the highest-order part of the equation.
Expanding $\Delta_{\Gamma(u)}H$ as in \cite{koch2012geometric}, we write
\begin{align*}
\partial_t u 
&= -Q\Delta_{\Gamma(u)} H - 2QH\left(\tfrac{1}{4}H^2-\Kk\right)\\
&= -\,L(\nabla u)D^4u - \mathcal R(\nabla u,D^2u,D^3u),
\end{align*}
where $L(\nabla u)D^4u$ is the fourth-order principal part and
$\mathcal R$ contains only lower-order terms.
In particular, $\mathcal R(\nabla u,D^2u,D^3u)$ is polynomial in derivatives of $u$
of order at most three, with coefficients depending smoothly on $\nabla u$
through powers of $Q^{-1}$.
More precisely, each monomial containing a third derivative is linear in $D^3u$
and involves at most one factor of $D^2u$, whereas terms without third derivatives are cubic in $D^2u$.
Consequently, for every $\Lambda>0$ there exists a constant $C_\Lambda>0$ such that
\begin{align}
|\mathcal R(\nabla u,D^2u,D^3u)|
\le
C_\Lambda\big(|D^3u|\,|D^2u|+|D^2u|^3\big)
\qquad
\text{whenever } |\nabla u|\le \Lambda .
\tag{\ref{eq:R-structure}}
\end{align}
This estimate will be used in the weighted Schauder theory and in the fixed-point argument.

\item[\cbm 2]
For short-time existence, we work in time-weighted parabolic Hölder spaces introduced by Belonosov~\cite{belonosov1979estimates}, which allow us to lower the regularity of the initial data and to avoid compatibility conditions at $t=0$.
More precisely, the spaces
\[
C^{4+\alpha,1+\alpha/4}_{1+\alpha}(\overline\Omega\times(0,T])
\quad\text{and}\quad
C^{4+\alpha,1+\alpha/4}_{1}(\overline\Omega\times(0,T])
\]
permit initial data in $C^{1+\alpha}(\overline\Omega)$ and $C^{0,1}(\overline\Omega)$, respectively.
We linearize the problem by freezing the coefficients, allowing derivatives $D^\beta u$ with $|\beta|\ge 2$ to blow up as $t\searrow 0$.
This enables the use of weighted Schauder estimates for linear parabolic equations, after which a fixed-point argument yields local existence.
The case of $C^{3+\alpha}$ initial data was treated in \cite{gulyak2017willmore}; see also \cite{DeGennaroDianaKubinKubin2024} for related uses of weighted Hölder spaces in geometric contexts.

\item[\cbm 3]
For global existence, we combine the guaranteed local existence time with the structure of the weighted norms to obtain bounds on higher-order derivatives.
We then apply spatial Hölder interpolation together with graphical a priori $L^2$-smallness estimates.
The limiting object is characterized via the associated stationary elliptic problem from \cite{gulyak2026}.
For the convergence rate, we use the divergence formulation of the Willmore flow equation due to Koch and Lamm~\cite{koch2012geometric}.
\end{itemize}

We also mention an alternative approach to reduce the regularity assumptions on the initial and boundary data, namely from $u_0\in C^{4+\alpha}(\overline\Omega)$ to $C^{2+\alpha}(\overline\Omega)$ and from $C^{4+\alpha}(\partial\Omega)$ to $C^{2+\alpha}(\partial\Omega)$, while avoiding the compatibility condition \eqref{eq:comp4}.
This approach relies on the biharmonic and divergence structure of the graphical Willmore flow, as observed by Koch and Lamm~\cite[Lemma~3.2, p.~215]{koch2012geometric} (see also \eqref{eq:willflowkoch} below), combined with the use of \emph{time-unweighted parabolic Hölder spaces}
\begin{align*}
C^{2+\alpha,(2+\alpha)/4}_{x,t}(\overline{\Omega}\times [0,T]).
\end{align*}
This method was developed by the author in his PhD thesis~\cite[Section~6.5]{gulyak2024boundary}.
The corresponding regularity framework is closer in spirit to \cite{simonett2001willmore}, where global existence and exponential convergence to a sphere are established for initial data close to spheres in the $C^{2+\alpha}$-topology.

\subsection{Main Results}

In this section we collect the main results of the paper.
They are formulated in three principal theorems.
The first theorem establishes short-time existence for both Hölder
and Lipschitz initial data. 
We present both cases since the less regular Lipschitz setting requires
a smallness assumption on the initial data, whereas the more regular
Hölder case does not. Moreover, the proofs rely on different techniques, and uniqueness is obtained only in the Hölder setting.

\begin{thm}[Local Existence]\label{thm:main1}
Let $\Omega \subset \R^2$ be a bounded domain with $C^{4+\alpha}$ boundary,
$\alpha\in(0,1)$, and let $g_0\in C^{4+\alpha}(\partial\Omega),g_1\in C^{3+\alpha}(\partial\Omega)$.

\begin{enumerate}
\item[(a)] \textbf{(Local Hölder case).}
Let $K>0$ and suppose
\[
u_0\in C^{1+\alpha}(\overline{\Omega}),
\qquad
\|\nabla u_0\|_{C^{\alpha}(\overline{\Omega})}
+\|g_0\|_{C^{4+\alpha}(\partial\Omega)}+\|g_1\|_{C^{3+\alpha}(\partial\Omega)} <K.
\]
Then there exists a time $T=T(\alpha,K,\Omega)>0$ and a unique solution
\[
u\in C^{4+\alpha,1+\alpha/4}_{1+\alpha}(\overline{\Omega}\times(0,T])
\]
to the Willmore flow problem \eqref{Willmore}.

\item[(b)] \textbf{(Local Lipschitz case).}
There exists a constant $M=M(\alpha,\Omega)>0$ such that if
\[
u_0\in C^{0,1}(\overline{\Omega}),
\qquad
\|\nabla u_0\|_{L^\infty(\Omega)}\le M,
\]
then there exists a time
\[
T=T(\alpha,\Omega,\|g\|_{C^{4+\alpha}(\partial\Omega)})>0
\]
for which the Willmore flow problem \eqref{Willmore} admits a solution
\[
u\in C^{4+\alpha,1+\alpha/4}_{1}(\overline{\Omega}\times(0,T]).
\]
\end{enumerate}
\end{thm}
\begin{proof}
The claim follows from Theorem~\ref{gewichtetexistkurz} and
Corollary~\ref{shorttimeexistenceCfixed}. Moreover, observe that any
vertical translation of a solution to the Willmore flow is again a
solution.
\end{proof}

The second main theorem provides criteria for global existence.
In contrast to the local Hölder theory, here it suffices to assume
Lipschitz regularity of the initial data.
Under a smallness condition on the Lipschitz constant and the
Dirichlet boundary data, we obtain a global solution with uniformly
bounded gradient which converges to a critical point determined by
the associated elliptic problem.

\begin{thm}[Global Existence]\label{thm:main2}
Let $\Omega \subset \R^2$ be a bounded domain with $C^{4+\alpha}$ boundary,
$\alpha\in(0,1)$, and let $g_0\in C^{4+\alpha}(\partial\Omega),g_1\in C^{3+\alpha}(\partial\Omega)$.
Assume
\[
u_0\in C^{0,1}(\overline{\Omega}).
\]

\begin{enumerate}
\item[(a)] \textbf{(Global existence)}
There exist constants $C=C(\alpha,\Omega)>0$ and
$\overline C=\overline C(\alpha,\Omega)>0$ such that if
\[
\|\nabla u_0\|_{L^\infty(\Omega)}
+
\|g_0\|_{C^{4+\alpha}(\partial\Omega)}+\|g_1\|_{C^{3+\alpha}(\partial\Omega)}
<C,
\]
then the Willmore flow problem \eqref{Willmore} admits a global solution.
More precisely, for every $T>0$ there exists
\[
u \in \Cc^{4+\alpha,1+\alpha/4}_{1}(Q_T)
\]
satisfying
\[
\|\nabla u(\cdot,t)\|_{L^\infty(\Omega)}
\le \overline C
\qquad\text{for all }t\ge0.
\]

\item[(b)] \textbf{(Convergence)}
There exists a critical point of the Willmore energy
\[
u_\infty\in C^{4+\alpha}(\overline{\Omega})
\]
such that for the global solution $u$ from part (a) and every
$\beta\in(0,\alpha)$ we have
\[
\|u(\cdot,t)-u_\infty\|_{C^{4+\beta}(\overline\Omega)}
\le \widetilde C e^{-\widetilde C (t-1)}
\qquad\text{for all } t\ge1,
\]
where $\widetilde C=\widetilde C(\alpha,\beta,\Omega)>0$.

\item[(c)] \textbf{(A priori estimate)}
Furthermore or all $t\ge0$,
\[
\sup_{x\in\Omega}|u(x,t)|
+
\int_\Omega Q(u(\cdot,t))\,dx
\le
C\big(\Omega,\|g\|_{W^{1,1}(\partial\Omega)},\W(u_0)\big).
\]
\end{enumerate}
\end{thm}
\begin{proof}
    See Theorems~\ref{thm:globalc1ac4a},  \ref{thm:globconv} and \ref{l2bound} (a)
\end{proof}

In the third theorem we present a result describing $L^2$-smallness. This estimate is of independent interest for two
reasons. First, it applies to all graphical surfaces with finite
Willmore energy, not only to solutions of the Willmore flow.
Second, for surfaces with fixed boundary, previously known estimates
typically provided only upper bounds depending on the Willmore energy.
Here we show that if both the Willmore energy and the boundary data are
sufficiently small, then the graph itself must remain small.

\begin{thm}[$L^2$-Smallness] \label{thm:l2small}
Let $\Omega\subset\R^2$ be a smooth bounded domain and suppose that
$\varphi\in C^2_c(\R^2)$ and $u\in W^{2,2}(\Omega)$ satisfy
\[
u-\varphi\in \mathring W^{2,2}(\Omega).
\]
Let $\gamma:I\to\partial\Omega$ be a parametrization of the boundary
$\partial\Omega$ by arclength.

Let $K>0$ and assume that
\[
\|(\varphi\circ\gamma)'\|_{L^1(I)}<K.
\]
Then for every $\varepsilon>0$ there exists
$\delta=\delta(\varepsilon,K,\Omega)>0$
such that
\[
\W(u)+\|\varphi\|_{L^1(\partial\Omega)}\le\delta
\quad\Rightarrow\quad
\|u\|_{L^2(\Omega)}<\varepsilon .
\]
\end{thm}

\begin{proof}
See Theorem~\ref{l2bound}\,\cbm{b}.
\end{proof}

\subsection{Outline}

We briefly outline the structure of the paper.

\begin{enumerate}
\item [] 
In Sec.~\ref{sec:willmoreflow}, we analyze the structure of the Willmore flow
in the graphical setting and recall the expressions of the basic geometric
quantities for graphs. Using results of Deckelnick and Dziuk
\cite{dziuk2006error} and Koch and Lamm \cite{koch2012geometric},
we rewrite the Willmore flow as a fourth-order nonuniformly elliptic
operator with a lower-order right-hand side.
The corresponding formulation is established in Lemma~\ref{ellipL}
and summarized in Lemma~\ref{lemm:parabwillmorestructure}.

\item  [] 
In Sec.~\ref{sec:sup}, we prove Theorem~\ref{thm:l2small}, which establishes
a smallness result for surfaces with small Willmore energy.
This estimate applies to all graphical surfaces in $\R^3$ and will be used
exclusively in the proof of global existence in Sec.~\ref{sec:globexist}.

\item  [] 
In Sec.~\ref{sec:paraspace}, we introduce parabolic Hölder spaces from \cite{belonosov1979estimates},
which form the functional framework for the fixed-point arguments.
The main linear results are the weighted Schauder estimate
(Theorem~\ref{gewichteteschauder}) and the corresponding existence theorem
(Theorem~\ref{gewichteteexistenz}), which will be used in the global analysis.

\item  [] 
In Sec.~\ref{sec:local_existence}, we prove short-time existence results.
For initial data in $C^{1+\alpha}$, we treat the problem in
Sec.~\ref{subsec:hoelderdata}, where we first establish a time-interpolation
result (Lemma~\ref{gammaalpha}) and several auxiliary estimates
(Lemmas~\ref{productrule}, \ref{gewichtethoedlerII}, \ref{gewichtethoedlerIII}),
before proving the main result in Theorem~\ref{gewichtetexistkurz}.
For Lipschitz initial data, we proceed analogously (see
Sec.~\ref{subsec:hoelderdata}), deriving preparatory estimates in
Lemmas~\ref{productrule22}, \ref{1gewichtethoedlerII}, and
\ref{1gewichtethoedlerIII}, and then proving the short-time existence
result in Theorem~\ref{shorttimeexistenceC}.

\item  [] 
In Sec.~\ref{sec:globexist}, we establish global existence and convergence.
In Sec.~\ref{subsec:globalexist}, we prove in
Theorem~\ref{thm:globalc1ac4a} that solutions exist for all time provided
the Lipschitz norm of the initial data and the Dirichlet boundary values
are sufficiently small.
In Sec.~\ref{subsec:subconv}, we first show subconvergence as $T\to\infty$
(Theorem~\ref{thm:subc}), and then convergence to a critical point
(Theorem~\ref{thm:globconv}).
\end{enumerate}

\section{Graphical Willmore flow}
\label{sec:willmoreflow}

We collect here the structural formulas for the graphical Willmore flow
that will be used throughout the sequel.
Let $\Omega\subset\R^2$ be a bounded domain with $\nu$ denoting the outward unit normal field on $\partial\Omega$ and let
$u:\Omega\to\R$ be sufficiently smooth.
We write $x=(x_1,x_2)$ and use the notation
\[
u_{x_i}=\partial_{x_i}u,
\qquad
u_{x_i x_j}=\partial_{x_i x_j}u,
\qquad
D^m u=\{\partial^\beta u:|\beta|=m\}.
\]
The graph of $u$ is denoted by
\[
\Gamma(u):=\{(x,u(x)):x\in\Omega\}\subset\R^3.
\]
Geometric quantities, like
induced surface element, mean curvature and the Gauss curvature of the graph take the form
\[
Q=(1+|\nabla u|^2)^{1/2},
\quad
H=\div\!\left(\frac{\nabla u}{Q}\right)
=
\frac{\Delta u}{Q}
-
\frac{\nabla u\cdot(D^2u\,\nabla u)}{Q^3},
\quad
\Kk
=
\frac{\det D^2u}{Q^4}
\]
so that the Willmore energy is given by
\[
\W(u)=\frac14\int_\Omega H^2Q\,dx.
\]
The mean curvature and the Gauss curvature of the graph takes the form

As already defined in \eqref{eq:willmoreflowgen} in the introduction, a family $(u(\cdot,t))_{t\in[0,T]}$ of graphs evolves by the \emph{Willmore flow} if
\begin{align*}
\partial_t u
=
-\,Q\left\{
\Delta_{\Gamma(u)}H
+
2H\left(\tfrac14H^2-\Kk\right)
\right\}
\qquad
\text{in }\Omega\times(0,T].
\end{align*}
It is a quasilinear parabolic fourth-order differential equation, whereby the $\Delta_{\Gamma(u)} H$-terms contains a fourth order nonuniformly  elliptic operator $L(\nabla u)D^4u$, described below in \eqref{eq:A}, which depends on $\nabla u$. In detail, 
\begin{align*}
	 \Delta_{\Gamma(u)} H&\ =\frac{1}{Q}\frac{\partial \ }{\partial x}\bigg\{\frac{1}{Q}\Big((1+u_{y}^2)\frac{\partial\ }{\partial x} H-u_{x}u_{y} \frac{\partial\ }{\partial y} H\Big) \bigg\} \\
	&\qquad+\frac{1}{Q}\frac{\partial \ }{\partial y}\bigg\{\frac{1}{Q}\Big(-u_{x}u_{y}\frac{\partial\ }{\partial x}H+(1+u^2_{x}) \frac{\partial\ }{\partial y} H\Big) \bigg\}.
	\end{align*}    

The Willmore flow is the $L^2$-gradient flow of the Willmore functional. Therefore, for sufficiently smooth solutions with time-independent boundary data one has the dissipation identity
\[
\W(u(t_2))\le \W(u(t_1))
\qquad\text{for }0\le t_1\le t_2.
\]
We will use this monotonicity repeatedly in the global theory.

Let us continue with discussing the structure of the graphical Willmore flow equation. Especially, we want to separate of the fourth-order part.
Following Dziuk and Deckelnick \cite[(1.5)-(1.9)]{dziuk2006error}
and Koch and Lamm \cite{koch2012geometric},
it may also be written in \emph{divergence form} as
\begin{align}
{u_t}
= - Q
\div\!\left\{
\frac{1}{Q}
\Big(I-\frac{\nabla u\otimes\nabla u}{Q^2}\Big)\nabla(QH)
-\frac{H^2}{2Q}\nabla u
\right\}.
\label{eq:willmorefloweq}
\end{align}
This formulation is particularly useful for energy estimates and for
rewriting the equation in divergence form relative to the biharmonic
operator. 

For the local existence theory it is convenient to isolate the highest-order
terms. A direct expansion of the right-hand side of \eqref{eq:willmoreflowgen}
shows that the Willmore flow equation can be written as
\begin{align}
Q\Delta_{\Gamma(u)}H
+
2QH\bigg(\frac14H^2-\Kk\bigg)
=
L(\nabla u)D^4u+\Rr(\nabla u,D^2u,D^3u).
\tag{A}\label{eq:A}
\end{align}
Here the the \emph{principal fourth-order part} is
\[
L(\nabla u)D^4u
=
\sum_{k+\ell=4}
L_{k\ell}(\nabla u)\,
\partial_{x_1}^k\partial_{x_2}^{\ell}u,
\]
while $\Rr$ contains only lower-order
terms in the sense that it depends on derivatives of $u$ of order at most
three.

The \emph{coefficients of the principal part} are given by
\begin{align}
\begin{pmatrix}
L_{40}(\nabla u)\\[0.2em]
L_{31}(\nabla u)\\[0.2em]
L_{22}(\nabla u)\\[0.2em]
L_{13}(\nabla u)\\[0.2em]
L_{04}(\nabla u)
\end{pmatrix}
=
\frac1{Q^4}
\begin{pmatrix}
(1+u_{x_2}^2)^2\\[0.2em]
-4(1+u_{x_2}^2)u_{x_1}u_{x_2}\\[0.2em]
2(1+u_{x_2}^2)(1+u_{x_1}^2)+4u_{x_1}^2u_{x_2}^2\\[0.2em]
-4(1+u_{x_1}^2)u_{x_1}u_{x_2}\\[0.2em]
(1+u_{x_1}^2)^2
\end{pmatrix}.
\tag{L}\label{eq:L}
\end{align}
The important point is that $L(\nabla u)$ depends only on the gradient.
Hence a bound on $\|\nabla u\|_{L^\infty}$ yields uniform ellipticity.

\begin{lemma}[Ellipticity]\label{ellipL}
Let $u\in C^1(\overline\Omega)$.
Then for every $\xi=(\xi_1,\xi_2)\in\R^2$,
\[
\frac{1}{\bigl(1+\|\nabla u\|_{C^0(\overline\Omega)}^2\bigr)^2}|\xi|^4
\le
\sum_{k+\ell=4}L_{k\ell}(\nabla u)\,\xi_1^k\xi_2^\ell
\le
4|\xi|^4.
\]
\end{lemma}

\begin{proof}
See \cite[Lemma~6.6]{gulyak2017willmore}.
\end{proof}

Next, we want to discuss the \emph{structure of the remainder term.}
For $L^2$-based arguments it is convenient to rewrite the flow relative to
the biharmonic operator.
To this end we use the $\star$-notation of
Kuwert and Schätzle \cite{kuwert2001willmore}
and Koch and Lamm \cite{koch2012geometric},
where $\star$ denotes a linear combination of tensor contractions and
\[
P_i(\nabla u)
:=
\underbrace{\nabla u\star\cdots\star\nabla u}_{i\text{ times}}.
\]
Then, by \cite[Lemma~3.2 p. 215]{koch2012geometric},
the graphical Willmore flow may be written as
\begin{align}
\partial_t u+\Delta^2u
=
f_0[u]+\partial_i f_1^i[u]+\partial_{ij}^2 f_2^{ij}[u]
=:\,f[u],
\label{eq:willflowkoch} \tag{D}
\end{align}
where
\begin{align}
\begin{aligned}
f_0[u]
&=
D^2u\star D^2u\star D^2u
\star
\sum_{k=1}^4 Q^{-2k}P_{2k-2}(\nabla u),\\
f_1[u]
&=
D^2u\star D^2u
\star
\sum_{k=1}^4 Q^{-2k}P_{2k-1}(\nabla u),\\
f_2[u]
&=
D^2u
\star
\sum_{k=1}^2 Q^{-2k}P_{2k}(\nabla u).
\end{aligned}
\label{eq:f}
\end{align}

Expanding the divergences in \eqref{eq:willflowkoch} reproduces the
decomposition \eqref{eq:A}. In particular, the two representations
\eqref{eq:A} and \eqref{eq:willflowkoch} are complementary:
the former is adapted to Schauder estimates and ellipticity,
whereas the latter is more convenient for energy and interpolation
arguments. It holds
\begin{align}
\begin{aligned}
\Rr(\nabla u,D^2u,D^3u)
=&\ 
D^3u\star D^2u \star \sum_{k=1}^4 Q^{-2k} P_{2k-1}(\nabla u)
\notag\\
&\quad
+ D^2u\star D^2u\star D^2u \star \sum_{k=0}^4 Q^{-2(k+1)} P_{2k}(\nabla u).
\end{aligned} \tag{$\mathcal R$}
\end{align}

We summarize the structural decomposition in the following statement.

\begin{lemma}\label{lemm:parabwillmorestructure}
The graphical Willmore flow equation \eqref{eq:willmorefloweq} can be
rewritten as
\[
-u_t
=
L(\nabla u)D^4u+\Rr(\nabla u,D^2u,D^3u)
\qquad\text{in }\Omega\times(0,T],
\]
where $L$ and $\Rr$ are given by \eqref{eq:L} and
\eqref{eq:R-structure}, respectively.
\end{lemma}

\begin{proof}
This follows from the expansion of \eqref{eq:willmorefloweq};
see also \cite{koch2012geometric} and
Lemma~\ref{thm:structure-remainder-term}.
\end{proof}


%

\section{Supremum and \texorpdfstring{$L^2$}{L2}-Smallness Estimates for Graphs} \label{sec:sup}

Due to the graphical setting, stronger estimates are available.
Let $\Omega\subset\R^2$ be a smooth bounded domain and denote by $\nu$
the outward unit normal vector field on $\partial\Omega$.
Let $\varphi:\overline{\Omega}\to\R$ be a $C^2$ boundary datum and consider
graphs $u:\overline{\Omega}\to\R$ belonging to the class
\[
M:=\big\{u\in W^{2,2}(\Omega)\; \big|\; (u-\varphi)\in \mathring{W}^{2,2}(\Omega)\big\},
\]
which represent Dirichlet boundary conditions.

In the graphical setting, Deckelnick, Grunau, and Röger
\cite[Theorem~2]{deckelnick2017minimising}
proved area and diameter bounds in terms of the Willmore energy,
$\|\varphi\|_{W^{2,2}(\partial\Omega)}$, and the geometry of the domain.
In particular, their result does not require any a priori bound on the
Willmore energy, such as $\W(\Sigma)<4\pi$, which is typically imposed in
the immersed setting.

As observed in \cite{grunau2018boundary}, such an estimate fails for
general non-projectable surfaces due to the scaling invariance of the
Willmore functional. 
To obtain their area and diameter bounds, Deckelnick, Grunau, and Röger
used the diameter estimate involving
$\|A\|_{L^1(\mathcal M)}$, proved by Simon in
\cite[Lemma~1.2, p.~283]{simon1993existence}.
A simpler graphical version of this lemma with explicit constants
can be found in \cite[Satz~4.2]{gulyak2014willmore}.

In the present subsection we first slightly refine this result by
employing Theorem~\ref{thm:imdiamest} instead.
This allows us to weaken the assumption
$\|\varphi\|_{W^{2,1}(\partial\Omega)}$
to
$\|\varphi\|_{W^{1,1}(\partial\Omega)}$
while also providing explicit constants.

In the second part we establish a new smallness estimate for
$\|u\|_{L^2(\Omega)}$.
This should be compared with the bound on
$\|u\|_{L^\infty(\Omega)}$
obtained in \cite[Theorem~2]{deckelnick2017minimising},
which does not include a corresponding smallness statement.

\begin{thm}\label{l2bound}
Let $\Omega\subset\R^2$ be a smooth bounded domain and suppose that
$\varphi\in C^2_c(\R^2)$ and $u\in W^{2,2}(\Omega)$ satisfy
\[
u-\varphi\in \mathring W^{2,2}(\Omega).
\]
Let $\gamma:I\to\partial\Omega$ be a parametrization of the boundary
$\partial\Omega$ by arclength.

\begin{enumerate}
\item[\cbm a]
Then, with
$\|\varphi\|_{W^{1,1}(\partial\Omega)}
=\|\varphi\circ\gamma\|_{W^{1,1}(I)}$,
we obtain
\begin{align*}
\sup_{x\in\Omega}|u(x)|+\int_\Omega Q\,dx
\le
64\Big(
\Hd^2(\Omega)
+\Hd^1(\partial\Omega)
+\|\varphi\|_{W^{1,1}(\partial\Omega)}
+\frac{16^2}{\pi^2}\W(u)
\Big)
\bigl(1+|\Omega|\W(u)\bigr).
\end{align*}

\item[\cbm b]
Let $K>0$ and assume that
$\|(\varphi\circ\gamma)'\|_{L^1(I)}<K$.
Then for every $\varepsilon>0$ there exists
$\delta=\delta(\varepsilon,K,\Omega)>0$
such that
\[
\W(u)+\|\varphi\|_{L^1(\partial\Omega)}\le\delta
\quad\Rightarrow\quad
\|u\|_{L^2(\Omega)}<\varepsilon .
\]
\end{enumerate}
\end{thm}

\begin{proof}
We begin with the same crucial integral $\int_\Omega uH\,dx$
as used in \cite[Theorem~2]{deckelnick2017minimising}.
Applying the divergence
theorem in $\Omega$, together with $Q^2=1+|\nabla u|^2$, we obtain
\begin{align}
\int_\Omega uH\,dx
&=
\int_\Omega u\,\div\!\left(\frac{\nabla u}{Q}\right) dx
\notag\\
&=
-\int_\Omega \frac{|\nabla u|^2}{Q}\,dx
+\int_{\partial\Omega}\frac{u\,\partial_\nu u}{Q}\,ds .
\label{eq:angewdivsatz}
\end{align}

This identity will be used in both parts of the proof.
The main task in parts \cbm a and \cbm b is to estimate the term
\[
\int_\Omega \frac{|\nabla u|^2}{Q}\,dx .
\]

\medskip\noindent\cb{a}
We follow the argument in \cite[Theorem~2]{deckelnick2017minimising},
replacing the diameter estimate by Theorem~\ref{thm:imdiamest}.
Using Hölder's inequality we obtain
\begin{align}
\frac{\pi}{16}\sup_{x\in\Omega}|u|
&\le
\int_{\M}|H|\,d\Hd^2
+\frac{\pi}{2}\Hd^1(\partial\M)
\notag\\
&\le
\sqrt{\int_\Omega |H|^2Q\,dx}\,
\sqrt{\int_\Omega Q\,dx}
+\frac{\pi}{2}\int_{\partial\Omega}\sqrt{1+|\varphi'|^2}\,ds
\notag\\
&\le
\sqrt{4\W(u)}\,
\sqrt{\int_\Omega Q\,dx}
+\frac{\pi}{2}\Big(\Hd^1(\partial\Omega)
+\|\varphi'\|_{L^1(\partial\Omega)}\Big).
\label{eq:supu}
\end{align}

Next we combine \eqref{eq:supu} with \eqref{eq:angewdivsatz}.
Using Hölder's inequality and the notation
$|\Omega|=\Hd^2(\Omega)$ and $|\partial\Omega|=\Hd^1(\partial\Omega)$,
we obtain
\begin{align*}
\int_\Omega Q\,dx
&=
\int_\Omega \frac{1}{Q}\,dx
+
\int_\Omega \frac{|\nabla u|^2}{Q}\,dx
\\
&\le
|\Omega|
+
\int_\Omega |uH|\,dx
+
\int_{\partial\Omega}
\varphi
\frac{\partial_\nu u}{\sqrt{1+|\nabla u|^2}}
\,ds .
\end{align*}
Since $Q\ge1$, we estimate the right-hand side by
\begin{align*}
\int_\Omega Q\,dx
&\le
|\Omega|
+
\frac{16}{\pi}
\Big(
\sqrt{4\W(u)}\sqrt{\int_\Omega Q\,dx}
+\frac{\pi}{2}\big(|\partial\Omega|
+\|\varphi'\|_{L^1(\partial\Omega)}\big)
\Big)
\sqrt{4|\Omega|\W(u)}
\\
&\quad
+\|\varphi\|_{L^1(\partial\Omega)} .
\end{align*}
Applying Young's inequality yields
\begin{align*}
    \int_\Omega Q \diff x 
     &\ \le |\Omega| 
    +  \frac{16^2}{2\pi^2}|\Omega|(4 \W (u))^2
     +\frac{1}{2}\int_\Omega Q\diff x + 8 \big(|\partial\Omega| + \|\varphi'\|_{ L^1(\partial\Omega)}\big) \sqrt{4|\Omega|\W(u)} \\
     &\ \qquad + \|\varphi\|_{ L^1(\partial\Omega)}
\end{align*}
Collecting the terms, we get
\begin{align*}
    \int_\Omega Q \diff x 
     &\ \le 2|\Omega| + 2\|\varphi\|_{ L^1(\partial\Omega)}
    +  \frac{16^3}{\pi^2}|\Omega|(\W (u))^2
      + 32 \big(|\partial\Omega| + \|\varphi'\|_{ L^1(\partial\Omega)}\big) \sqrt{|\Omega|\W(u)} .
\end{align*}
Then we use again  \eqref{eq:supu} and Hölder's inequality to show
\begin{align*}
     \sup_{x\in\Omega}|u| + \int_\Omega Q \diff x 
     &\ \le  \sqrt{2 \frac {16^2}{\pi^2}\W (u)}
     \sqrt{2\int_\Omega Q\diff x}+ 8 \big(\Hd^1(\partial\Omega) + \|\varphi'\|_{ L^1(\partial\Omega)}\big) + \int_\Omega Q \diff x  \\
     &\ \le \frac{16^2}{\pi^2}  \W (u)  + {\int_\Omega Q\diff x}+
     {\int_\Omega Q\diff x}
     + 8\big(|\partial\Omega| + \|\varphi'\|_{ L^1(\partial\Omega)}\big)\\
     &\ \le 
     4\big( |\Omega| + \|\varphi\|_{ L^1(\partial\Omega)}\big)
    +  2\frac{16^2}{\pi^2}\big(16|\Omega|\W (u) +1\big)
    \W (u)\\
&\ \quad      + 8\big(|\partial\Omega| + \|\varphi'\|_{ L^1(\partial\Omega)}\big)\left(1+2\cdot 4 \sqrt{|\Omega|\W(u)}\right) \\
&\ \le 
64 \left(|\Omega|+|\partial\Omega| + \|\varphi\|_{W^{1,1}(\partial\Omega)} + \frac{16^2}{\pi^2} \W(u)\right)\big(1+|\Omega|\W(u)\big).
\end{align*}

\medskip\noindent\cb{b}
To obtain an $L^2$-smallness estimate, we combine the boundedness of
$\|u\|_{L^\infty(\Omega)}$ and $\|\nabla u\|_{L^1(\Omega)}$
with Hölder's inequality as an interpolation tool.
Thus it suffices to control $\|u\|_{W^{1,1}(\Omega)}$.
For this purpose we estimate the term
$\int_\Omega |\nabla u|^2/Q\,dx$
instead of $\int_\Omega Q\,dx$.
This avoids the appearance of the factor $\Hd^2(\Omega)$,
which cannot be made arbitrarily small since the domain $\Omega$ is fixed.

We assume that $\|\varphi\|_{L^1(\partial\Omega)}<K$.
By part \cbm{a} there exists a constant $\Cl{uQbound}$ depending only on
$|\partial\Omega|$, $K$, and $\W(u)$ such that
\[
\sup_{x\in\Omega}|u(x)|+\int_\Omega Q\,dx \le \Cr{uQbound}.
\]
Using the identity \eqref{eq:angewdivsatz} we obtain
\begin{align*}
\int_\Omega \frac{|\nabla u|^2}{Q}\,dx
&=
\int_{\partial\Omega}\frac{\partial_\nu u}{Q}\,u\,ds
-\int_\Omega uH\,dx \\
&\le
\|\varphi\|_{L^1(\partial\Omega)}
+\sup_{x\in\Omega}|u(x)|\int_\Omega |H|\,dx \\
&\le
\|\varphi\|_{L^1(\partial\Omega)}
+\Cr{uQbound}\sqrt{|\Omega|\W(u)} .
\end{align*}
Consequently,
\begin{align*}
\int_\Omega |\nabla u|\,dx
&\le
\left(\int_\Omega Q\,dx\right)^{1/2}
\left(\int_\Omega \frac{|\nabla u|^2}{Q}\,dx\right)^{1/2} \\
&\le
\Cr{uQbound}^{1/2}
\left(
\|\varphi\|_{L^1(\partial\Omega)}
+\Cr{uQbound}\sqrt{|\Omega|\W(u)}
\right)^{1/2}.
\end{align*}
By Hölder's inequality with $\frac{1}{p}+\frac{1}{q}=1$
(in particular the Cauchy–Schwarz inequality), we obtain
\begin{align*}
\left(\int_\Omega u^2\,dx\right)^2
&=
\left(\int_\Omega |u|^{3/2}|u|^{1/2}\,dx\right)^2
\le
\left(\int_\Omega |u|^3\,dx\right)
\left(\int_\Omega |u|\,dx\right).
\end{align*}
Since
\[
\int_\Omega |u|^3\,dx
\le
|\Omega|\,\sup_{x\in\Omega}|u(x)|^3,
\]
and by the Poincaré–Friedrichs inequality
\[
\int_\Omega |u|\,dx
\le
\Cl{poin}(\Omega)\!\left(
\int_\Omega |\nabla u|\,dx
+
\int_{\partial\Omega}|u|\,ds
\right),
\]
we obtain
\begin{align*}
\left(\int_\Omega u^2\,dx\right)^2
&\le
|\Omega|\,\sup_{x\in\Omega}|u(x)|^3
\,\Cr{poin}
\left(
\int_\Omega |\nabla u|\,dx
+
\int_{\partial\Omega}|u|\,ds
\right).
\end{align*}
Using the bounds obtained above,
we arrive at
\begin{align*}
\left(\int_\Omega u^2\,dx\right)^2
&\le
\Cr{uQbound}^3|\Omega|\Cr{poin}
\left(
\Cr{uQbound}^{1/2}
\left(\|\varphi\|_{L^1(\partial\Omega)}
+\Cr{uQbound}\sqrt{|\Omega|\W(u)}\right)^{1/2}
+
\|\varphi\|_{L^1(\partial\Omega)}
\right).
\end{align*}
The desired estimate now follows by choosing
$\W(u)$ and $\|\varphi\|_{L^1(\partial\Omega)}$
sufficiently small.
\end{proof}

For the interpolation arguments used later, smallness in the $L^2$-norm is sufficient to prove global existence of the Willmore flow.
A natural question in this context is whether a corresponding smallness estimate also holds for the $L^\infty$-norm.
To the best of our knowledge, this question remains open.

\section{Parabolic Hölder spaces}
\label{sec:paraspace}

In this section we recall the anisotropic Hölder spaces and the
corresponding linear Schauder theory for fourth-order parabolic
equations.
These spaces are the natural analogue of the elliptic Hölder spaces
used earlier, but now the time variable has to be treated differently
from the spatial variables.
For fourth-order equations the parabolic scaling is
\[
\partial_t \sim D_x^4,
\]
so that one time derivative carries the same weight as four spatial
derivatives.
This anisotropy is reflected in the Hölder norms below.

The general theory of anisotropic Hölder spaces can be found, for
instance, in Solonnikov \cite{solonnikov1965boundary},
Lunardi \cite{lunardi1992semigroup}, and Belonosov
\cite{belonosov1979estimates}.
Since our later applications concern the Willmore flow, which is a
fourth-order parabolic equation, we restrict ourselves from the outset
to the fourth-order scaling.

Throughout this section let $\Omega\subset\R^n$ be a bounded domain with
$\Cc^{4+\alpha}$-boundary, $0<\alpha<1$, and let $T>0$.
We denote the two kinds of parabolic cylinders by
\[
\overline Q_T:=\overline\Omega\times[0,T],
\qquad
Q_T:=\overline\Omega\times(0,T].
\]
Let $\ell>0$.
For $\ell\notin\N$ we define \emph{unweighted anisotropic Hölder spaces}
\begin{align*}
\|u\|_{\Cc^{\ell,\ell/4}_{x,t}(\overline Q_T)}
:={}&
\sum_{4k+|\beta|\le \lfloor \ell\rfloor}
\sup_{(x,t)\in\overline Q_T}|D_t^kD_x^\beta u(x,t)|
\\
&\quad
+\sum_{4k+|\beta|=\lfloor \ell\rfloor}
\sup_{t\in[0,T]}
\big[D_t^kD_x^\beta u(\cdot,t)\big]_{\Cc^{\ell-\lfloor\ell\rfloor}(\overline\Omega)}
\\
&\quad
+\sum_{\ell-4<4k+|\beta|<\ell}
\sup_{x\in\overline\Omega}
\big[D_t^kD_x^\beta u(x,\cdot)\big]_
{\Cc^{\frac{\ell-4k-|\beta|}{4}}([0,T])}.
\end{align*}
If $\ell\in\N$, we set
\begin{align*}
\|u\|_{\Cc^{\ell,\ell/4}_{x,t}(\overline Q_T)}
:={}&
\sum_{4k+|\beta|\le \ell}
\sup_{(x,t)\in\overline Q_T}|D_t^kD_x^\beta u(x,t)|
\\
&\quad
+\sum_{\ell-4<4k+|\beta|<\ell}
\sup_{x\in\overline\Omega}
\big[D_t^kD_x^\beta u(x,\cdot)\big]_
{\Cc^{\frac{\ell-4k-|\beta|}{4}}([0,T])}.
\end{align*}

For later applications we also need \emph{weighted anisotropic Hölder spaces} that allow a
controlled loss of regularity at the initial time \(t=0\).
Let \(s\le \ell\).
For \(0<t\le T\) we set
\[
Q_t':=\overline\Omega\times[t/2,t].
\]
Furthermore, for \(\ell\notin\N\) we define
\begin{align*}
[u]^\ell_{Q_t'}
:={}&
\sum_{4k+|\beta|=\lfloor\ell\rfloor}
\sup_{t'\in[t/2,t]}
\big[D_t^kD_x^\beta u(\cdot,t')\big]_{\Cc^{\ell-\lfloor\ell\rfloor}(\overline\Omega)}
\\
&\quad
+\sum_{\ell-4<4k+|\beta|<\ell}
\sup_{x\in\overline\Omega}
\big[D_t^kD_x^\beta u(x,\cdot)\big]_
{\Cc^{\frac{\ell-4k-|\beta|}{4}}([t/2,t])},
\end{align*}
and use the same formula without the first sum when \(\ell\in\N\).

The weighted norm is then given by
\begin{align*}
\|u\|_{\Cc^{\ell,\ell/4}_{s}(Q_T)}
:={}&
\sup_{0<t\le T} t^{\frac{\ell-s}{4}}[u]^\ell_{Q_t'}
+
\sum_{s<4k+|\beta|<\ell}
\sup_{(x,t)\in Q_T}
t^{\frac{4k+|\beta|-s}{4}}|D_t^kD_x^\beta u(x,t)|
\\
&\quad
+
\begin{cases}
\|u\|_{\Cc^{s,s/4}_{x,t}(\overline Q_T)}, & s\ge0,\\
0, & s<0.
\end{cases}
\end{align*}

For \(\N\not\ni s>0\), the space \(\Cc^{\ell,\ell/4}_{s}(Q_T)\) consists of all
functions \(u\in \Cc^0(\overline Q_T)\) such that

\begin{itemize}
\item \(D_t^kD_x^\beta u\) extends continuously to \(\overline Q_T\)
whenever \(4k+|\beta|\le\lfloor s\rfloor\),
\item \(D_t^kD_x^\beta u\) exists in \(Q_T\) whenever
\(\lfloor s\rfloor<4k+|\beta|\le\lfloor\ell\rfloor\),
\item \(\|u\|_{\Cc^{\ell,\ell/4}_{s}(Q_T)}<\infty\).
\end{itemize}

For $s<0$, no continuity at $t=0$ is imposed. In this case one only requires
the derivatives to exist in $Q_T$ and the weighted norm to be finite.

For $s\in\mathbb{N}$, the spaces $\Cc^{\ell,\ell/4}_{s}(Q_T)$ are defined in
a uniform way, in contrast to \cite[below (1.6), p.~154]{belonosov1979estimates},
namely in the same manner as for non-integral $s$.
In particular, for integral $s$ the derivatives $D_t^kD_x^\beta u$ with
$4k+|\beta|\le s$ are, in general, no longer continuous up to $t=0$;
see \cite[below Theorem~2, p.~166]{belonosov1979estimates}.

\emph{Remark.}
The parameter \(s\) measures how much regularity is retained at
\(t=0\). 
If \(s<\ell\), then derivatives of parabolic order strictly larger than
\(s\) are allowed to blow up as \(t\searrow0\), but only at the rate
prescribed by the weights.
This is precisely the mechanism needed later for rough initial data.

\emph{Boundary spaces:}
The spaces
\(\Cc^{\ell,\ell/4}_{s}(\partial\Omega\times(0,T])\)
are defined in the same way by replacing \(\overline\Omega\) by
\(\partial\Omega\) and taking tangential spatial derivatives.
For brevity we omit the explicit formulas.

\emph{Basic properties:}
We record the properties of weighted parabolic Hölder spaces that will be
used later. They are taken from Belonosov
\cite[p.~154]{belonosov1979estimates}; see also
 \cite[Appendix, Lemma 91]{gulyak2024boundary} for the equivalence of our norm
with Belonosov's original definition.

\begin{itemize}
\item[(i)] \(\Cc^{\ell,\ell/4}_{s}(Q_T)\) is a Banach space.
\item[(ii)] \(\Cc^{\ell,\ell/4}_{x,t}(\overline Q_T)\subset
\Cc^{\ell,\ell/4}_{s}(Q_T)\) whenever \(s\le \ell\), and 
\[
\Cc^{\ell,\ell/4}_{\ell}(Q_T)=\Cc^{\ell,\ell/4}_{x,t}(\overline Q_T).
\]
\item[(iii)] If \(4k+|\beta|\le \ell\), then 
\begin{align}
\|D_t^kD_x^\beta u\|_{
\Cc^{\ell-(4k+|\beta|),\,\frac{\ell-(4k+|\beta|)}4}_{\,r-(4k+|\beta|)}(Q_T)}
\le
\|u\|_{\Cc^{\ell,\ell/4}_{r}(Q_T)} .
\label{eq:einbett}
\end{align}
\item[(iv)] There exists a function
\(\Cl{gewichtetprodukt}(T)\), bounded as \(T\searrow0\), such that
\begin{align}
\|uw\|_{\Cc^{\ell,\ell/4}_{r}(Q_T)}
\le
\Cr{gewichtetprodukt}(T)\,
\|u\|_{\Cc^{\ell,\ell/4}_{r_1}(Q_T)}
\|w\|_{\Cc^{\ell,\ell/4}_{r_2}(Q_T)},
\label{eq:simpleprod}
\end{align}
where
\[
r=\min\{r_1,r_2,r_1+r_2\}.
\]
\end{itemize}

The following two results are the basic linear tools for the nonlinear
fixed-point arguments in the next sections, namely the \emph{linear weighted Schauder theory.}
Since later we work with fourth-order equations and clamped boundary
conditions, we formulate the theory directly in that setting.

\begin{thm}[Weighted existence]\label{gewichteteexistenz}
Let \(\Omega\subset\R^n\) be bounded with \(\Cc^{4+\alpha}\)-boundary,
let \(0<\alpha<1\), \(T>0\), and \(s\in[0,4+\alpha]\).
Assume
\[
f\in \Cc^{\alpha,\alpha/4}_{s-4}(Q_T),
\qquad
a_\beta\in \Cc^{\alpha,\alpha/4}_{\max\{0,s-4\}}(Q_T)
\quad\text{for all }|\beta|\le4.
\]
Moreover, suppose that the principal coefficients satisfy the uniform
ellipticity condition
\[
\lambda|\xi|^4
\le
\sum_{|\beta|=4} a_\beta(x,t)\,\xi^\beta
\le
\Lambda|\xi|^4
\qquad
\text{for all }\xi\in\R^n,\ (x,t)\in \overline Q_T,
\]
with constants \(0<\lambda\le\Lambda\).

Consider the linear fourth-order parabolic problem
\begin{align*}
\partial_t u+\sum_{|\beta|\le4}a_\beta(x,t)D_x^\beta u &= f(x,t),
&& (x,t)\in Q_T,\\
u(x,0)&=u_0(x),
&& x\in\overline\Omega,\\
u(x,t)&=\varphi(x,t),
&& (x,t)\in\partial\Omega\times[0,T],\\
\partial_\nu u(x,t)&=h(x,t),
&& (x,t)\in\partial\Omega\times[0,T],
\end{align*}
with data
\[
u_0\in \Cc^s(\overline\Omega),\qquad
\varphi\in \Cc^{4+\alpha,1+\alpha/4}_{s}(\partial\Omega\times(0,T]),\qquad
h\in \Cc^{3+\alpha,(3+\alpha)/4}_{s-1}(\partial\Omega\times(0,T]).
\]
Assume the compatibility conditions
\begin{align*}
\varphi(x,0)&=u_0(x),
\qquad
h(x,0)=\partial_\nu u_0(x),
\qquad x\in\partial\Omega,
\end{align*}
and, only in the case \(s\ge4\), the additional fourth-order
compatibility condition
\begin{align*}
\partial_t\varphi(x,0)
=
-\sum_{|\beta|\le4}a_\beta(x,0)D_x^\beta u_0(x)+f(x,0),
\qquad x\in\partial\Omega.
\end{align*}
Then there exists a unique solution
\[
u\in \Cc^{4+\alpha,1+\alpha/4}_{s}(Q_T).
\]
\end{thm}

\begin{proof}
This follows from the results of Belonosov
\cite[Theorem~4, p.~185]{belonosov1979estimates}
for fourth-order parabolic problems with clamped boundary conditions,
combined with the standard anisotropic Hölder framework in
\cite{lunardi1992semigroup}.
\end{proof}

\begin{thm}[Weighted Schauder estimate]\label{gewichteteschauder}
Under the assumptions of Theorem~\ref{gewichteteexistenz}, the solution
satisfies the estimate
\begin{align*}
\|u\|_{\Cc^{4+\alpha,1+\alpha/4}_{s}(Q_T)}
\le
\Cl{ws}(T)
\Big(
&\ \|f\|_{\Cc^{\alpha,\alpha/4}_{s-4}(Q_T)}
+\|\varphi\|_{\Cc^{4+\alpha,1+\alpha/4}_{s}(\partial\Omega\times(0,T])}\\
&\ +\|h\|_{\Cc^{3+\alpha,(3+\alpha)/4}_{s-1}(\partial\Omega\times(0,T])}
+\|u_0\|_{\Cc^{s}(\overline\Omega)}
\Big),
\end{align*}
where \(\Cr{ws}:\R_+\to\R_+\) is a monotone function depending only on
\(\lambda,\Lambda,\Omega\), and an upper bound $C$
\[
\|a_\beta\|_{\Cc^{\alpha,\alpha/4}_{\max\{0,s-4\}}(Q_T)}\le C
\qquad
(|\beta|\le4).
\]
\end{thm}

\begin{proof}
This is Belonosov's weighted Schauder estimate,
see \cite[p.~184, Corollary to (4.10)]{belonosov1979estimates},
specialized to fourth-order equations and clamped boundary data.
\end{proof}

\emph{Remark.}
In later applications we will always set \(\ell=4+\alpha\) and take
\[
s=m+\alpha,
\qquad m\in\{1,2,3,4\},
\]
or, in the Lipschitz case, \(s=1\).
The weighted theory above is precisely what allows us to lower the
regularity of the initial data while retaining full \(4+\alpha\) spatial
regularity for positive times.



\section{Local Existence}
\label{sec:local_existence}

\subsection{Hölder Initial Data}
\label{subsec:hoelderdata}

In this section we develop the short-time theory for the graphical
Willmore flow with Hölder regular initial data.
The main application is the case
\(
u_0\in \Cc^{1+\alpha}(\overline\Omega),
\)
but the weighted framework naturally treats the more general classes
\[
u_0\in \Cc^{m+\alpha}(\overline\Omega),
\qquad m\in\{1,2,3,4\},
\]
and we formulate the auxiliary estimates at this level of generality.

The main difficulty is caused by the fourth-order quasilinear structure
of the equation and the loss of spatial regularity at the initial time.
If one worked in the classical unweighted parabolic Hölder spaces,
then for $m<4$ one would not expect the solution to remain in
$\Cc^{4+\alpha,1+\alpha/4}_{x,t}$ up to $t=0$.
To overcome this, we use the time-weighted spaces
\[
\Cc^{4+\alpha,1+\alpha/4}_{m+\alpha}(Q_T),
\qquad m=1,2,3,4,
\]
which allow derivatives of parabolic order strictly larger than $m$
to blow up as $t\searrow0$ at a controlled rate, while derivatives of
parabolic order at most $m$ remain Hölder continuous up to $t=0$.
In particular, this avoids the fourth-order compatibility condition at
the initial time and treats the cases $m=1,2,3,4$ in a unified way.

For the weighted linear theory developed in the appendix, we set
\[
s=m+\alpha,
\qquad
f=\Rr(\nabla u,D^2u,D^3u),
\qquad
(a_\beta)\cong L(\nabla u),
\]
with time-independent Dirichlet data
\(
\varphi=g_0,
\
h=g_1.
\)

We first recall the unweighted part of the norm:
\begin{align*}
\|u\|_{\Cc^{m+\alpha,\frac{m+\alpha}{4}}_{x,t}(\overline Q_T)}
={}&
\sum_{4k+|\beta|\le m}
\sup_{(x,t)\in \overline Q_T}|D_t^kD_x^\beta u(x,t)| \\
&\quad
+\sum_{4k+|\beta|=m}
\sup_{t\in[0,T]}
\big[D_t^kD_x^\beta u(\cdot,t)\big]_{\Cc^\alpha(\overline\Omega)} \\
&\quad
+\sum_{m-4+\alpha<4k+|\beta|\le m}
\sup_{x\in\overline\Omega}
\big[D_t^kD_x^\beta u(x,\cdot)\big]_
{\Cc^{\frac{m+\alpha-4k-|\beta|}{4}}([0,T])}.
\end{align*}
The full weighted norm is then defined by
\begin{align*}
\|u\|_{\Cc^{4+\alpha,1+\alpha/4}_{m+\alpha}(Q_T)}
={}&
\sup_{0<t\le T}
t^{\frac{4-m}{4}}[u]^{4+\alpha}_{Q_t'}
+
\sum_{m<4k+|\beta|\le4}
\sup_{(x,t)\in Q_T}
t^{\frac{4k+|\beta|-m-\alpha}{4}}
|D_t^kD_x^\beta u(x,t)| \\
&\quad
+\|u\|_{\Cc^{m+\alpha,\frac{m+\alpha}{4}}_{x,t}(\overline Q_T)},
\end{align*}
where
\(
Q_t':=\overline\Omega\times[t/2,t]
\) for \(0<t\le T,
\)
and
\begin{align*}
[u]^{4+\alpha}_{Q_t'}
={}&
\sum_{4k+|\beta|=4}
\sup_{t'\in[t/2,t]}
\big[D_t^kD_x^\beta u(\cdot,t')\big]_{\Cc^\alpha(\overline\Omega)} \\
&\quad
+\sum_{1\le4k+|\beta|\le4}
\sup_{x\in\overline\Omega}
\big[D_t^kD_x^\beta u(x,\cdot)\big]_
{\Cc^{\frac{4+\alpha-4k-|\beta|}{4}}([t/2,t])}.
\end{align*}

To keep the fixed-point argument readable, we move the technical
estimates to the appendix and record here only the ingredients needed
below. One of the key steps is to lower the Hölder exponent from
$\alpha$ to some $\gamma\in(0,\alpha)$; for functions with vanishing
initial traces, this produces the small factor
$T^{\frac{\alpha-\gamma}{4}}$ and allows us to choose the time interval
short enough for a contraction argument.

\begin{lemma}\label{gammaalpha}
Let $m\in\{1,2,3,4\}$ and let
$u\in \Cc^{4+\alpha,1+\alpha/4}_{m+\alpha}(Q_T)$ satisfy
\[
D_t^kD_x^\beta u(x,0)=0
\qquad
\text{for all }x\in\overline\Omega
\text{ and }4k+|\beta|\le m.
\]
Then there exists a constant
\[
\Cl{constgammaalpha}=\Cr{constgammaalpha}(\alpha,\gamma)>0
\]
such that for every $0<\gamma<\alpha$ and every $T\le1$,
\[
\|u\|_{\Cc^{4+\gamma,1+\gamma/4}_{m+\gamma}(Q_T)}
\le
\Cr{constgammaalpha}\,
T^{\frac{\alpha-\gamma}{4}}
\|u\|_{\Cc^{4+\alpha,1+\alpha/4}_{m+\alpha}(Q_T)}.
\]
\end{lemma}

\begin{proof}
See Appendix, Lemma~\ref{gammaalphaappendix}.
\end{proof}

The remainder term $\Rr$ consists of products of derivatives of $u$.
Therefore, in order to estimate it in
$\Cc^{\alpha,\alpha/4}_{m+\alpha-4}(Q_T)$,
we need product estimates in weighted Hölder spaces.

\begin{lemma}[Basic product estimates]\label{productrule}
Let $m\in\{1,2,3,4\}$, let $0<\gamma\le\alpha<1$, assume $\gamma\ge\alpha/2$,
and let $T\le1$.
Then there exists a constant
\[
\Cl{prodrule}=\Cr{prodrule}(\alpha,\gamma,\Omega)>0
\]
depending only on $\alpha$, $\gamma$, $\Omega$, and the fixed algebraic
structure of $L$ and $\Rr$, such that for all
$u,v,w\in \Cc^{4+\gamma,1+\gamma/4}_{m+\gamma}(Q_T)$,
\begin{align}
          &\ \|\nabla u\| _{C^{\alpha, \frac{\alpha}{4}}_{\max\{0,m+\alpha-4\}}(Q_T)} 
          \le     \Cr{prodrule}\|\nabla u\|_{C^{3+\gamma, \frac{3+\gamma}{4}}_{m+\gamma-1}(Q_T)}    \label{eq:1}, \\
     &\ \|D^3wD^2 u\|_{C^{\alpha, \frac{\alpha}{4}}_{m+\alpha-4}(Q_T)}  
     \le 
     \Cr{prodrule} \|D^3w\|_{C^{1+\gamma, \frac{1+\gamma}{4}}_{m+\gamma-3}(Q_T)} 
     \cdot \|D^2u\|_{C^{2+\gamma, \frac{2+\gamma}{4}}_{m+\gamma-2}(Q_T)}  \label{eq:32} ,\\ 
      &\  \begin{aligned}
            \|D^2 u D^2 & w D^2v\|_{C^{\alpha, \frac{\alpha}{4}}_{m+\alpha-4}(Q_T)} \\
           &\ \le 
          \Cr{prodrule} \|D^2u\|_{C^{2+\gamma, \frac{2+\gamma}{4}}_{m+\gamma-2}(Q_T)} \cdot \|D^2w\|_{C^{2+\gamma, \frac{2+\gamma}{4}}_{m+\gamma-2}(Q_T)} 
          \cdot \|D^2v\|_{C^{2+\gamma, \frac{2+\gamma}{4}}_{m+\gamma-2}(Q_T)}  .
        \end{aligned}    \label{eq:222}
\end{align}
\end{lemma}

\begin{proof}
See Appendix, Lemma~\ref{productruleappendix}.
\end{proof}

The preceding product bounds imply the following estimates for the
coefficients of the principal part and for the lower-order remainder.

\begin{lemma}[Hölder estimates I]\label{gewichtethoedlerII}
Let $m\in\{1,2,3,4\}$, let $0<\gamma,\alpha<1$, and let $T\le1$.
Then there exist constants
\[
\Cl{gewhoelderI}=\Cr{gewhoelderI}(\Omega,\alpha,\gamma)>0,
\qquad
k_H\in\N,
\]
depending only on $\Omega$, $\alpha$, $\gamma$, and the fixed algebraic
structure of $L$ and $\Rr$, such that for every
$u\in \Cc^{4+\gamma,1+\gamma/4}_{m+\gamma}(Q_T)$,
\begin{align*}
\|\Rr(\nabla u,D^2u,D^3u)\|_{\Cc^{\alpha,\alpha/4}_{m+\alpha-4}(Q_T)}
&\le
\Cr{gewhoelderI}
\Big(
1+\|u\|_{\Cc^{4+\gamma,1+\gamma/4}_{m+\gamma}(Q_T)}
\Big)^{k_H}
\|u\|_{\Cc^{4+\gamma,1+\gamma/4}_{m+\gamma}(Q_T)}^3,\\
\sum_{k+\ell=4}
\|L_{k\ell}(\nabla u)\|_{\Cc^{\alpha,\alpha/4}_{\max\{0,m+\alpha-4\}}(Q_T)}
&\le
\Cr{gewhoelderI}
\Big(
1+\|u\|_{\Cc^{4+\gamma,1+\gamma/4}_{m+\gamma}(Q_T)}^4
\Big).
\end{align*}
\end{lemma}

\begin{proof}
See Appendix, Lemma~\ref{gewichtethoedlerIIappendix}.
\end{proof}

For the contraction step we also need difference estimates.

\begin{lemma}[Hölder estimates II]\label{gewichtethoedlerIII}
Let $m\in\{1,2,3,4\}$, let $0<\gamma,\alpha<1$, and let $T\le1$.
Then there exist constants
\[
\Cl{hoelderII}=\Cr{hoelderII}(\alpha,\gamma,\Omega)>0,
\qquad
k_H'\in\N,
\]
depending only on $\alpha$, $\gamma$, $\Omega$, and the fixed algebraic
structure of $L$ and $\Rr$, such that for all
$u,w\in\Cc^{4+\gamma,1+\gamma/4}_{m+\gamma}(Q_T)$,
\begin{align*}
	&\	\left\|\Rr( \nabla u,D^2 u, D^3 u)-\Rr( \nabla w,D^2 w, D^3 w)\right \|_{\Cc^{\alpha,\alpha/4}_{m+\alpha-4}(Q_T)} \\
&\ \quad +
\sum_{k+\ell=4}\left\|L_{k\ell}(\nabla u)-L_{k\ell}(\nabla w)\right\|_{\Cc^{\alpha,\alpha/4}_{\max\{0,m+\alpha-4\}}( Q_T)}\\
&\ \qquad \qquad \qquad   
\le 
\Cr{hoelderII} \Big(1+ \max\Big\{\|u\|_{\Cc^{4+\gamma,1 + \gamma/4}_{m+\gamma} (Q_T)},\|w\|_{\Cc^{4+\gamma, 1 + \gamma/4}_{m+\gamma} (Q_T)}\Big\}\Big)^{k'_H} \\
& \ \qquad \qquad \qquad \qquad 
\cdot
\max\Big\{\|u\|_{\Cc^{4+\gamma,1 + \gamma/4}_{m+\gamma} (Q_T)},\|w\|_{\Cc^{4+\gamma, 1 + \gamma/4}_{m+\gamma} (Q_T)}\Big\}^2
 \cdot
 \|\nabla(u-w)\|_{\Cc^{3+\gamma, \frac{3+\gamma} 4}_{m+\gamma-1} (Q_T)}.
\end{align*}
\end{lemma}

\begin{proof}
The proof is similar to that of
Lemma~\ref{gewichtethoedlerIIappendix}; in addition one uses the
representation of differences of polynomials from \cite[Appendix, Lemma 96]{gulyak2024boundary}.
\end{proof}

We are now in a position to state the short-time existence theorem for
Hölder initial data.



\begin{thm}[Short-time existence and uniqueness in weighted Hölder spaces]
\label{gewichtetexistkurz}
Let $\Omega \subset \R^2$ be a bounded domain with $\Cc^{4+\alpha}$-smooth boundary and outer unit normal $\nu$.
Let $m \in \{1,2,3,4\}$.
Then there exists a time $T \in (0,1)$ such that the Willmore flow problem \eqref{Willmore}
admits a unique solution
\[
u \in \Cc^{4+\alpha,\,1+\alpha/4}_{m+\alpha}(Q_T),
\]
provided the initial datum
\[
u_0 \in \Cc^{m+\alpha}(\overline{\Omega})
\]
and the boundary data
\[
g_0 \in \Cc^{4+\alpha}(\partial\Omega),
\qquad
g_1 \in \Cc^{3+\alpha}(\partial\Omega)
\]
satisfy the compatibility conditions
\begin{align*}
g_0(x) = u_0(x),
\qquad
g_1(x) = \partial_\nu u_0(x),
\qquad x \in \partial\Omega.
\end{align*}
In the case $m=4$, an additional compatibility condition is required, namely
\begin{align*}
0
= \Delta_{\Gamma(u_0)} H(u_0)
+ 2 H(u_0)\Bigl(\tfrac14 H(u_0)^2 - \Kk(u_0)\Bigr),
\qquad x \in \partial\Omega.
\end{align*}
\end{thm}

\begin{proof}
The main idea of the proof is to reduce the quasilinear Willmore flow equation
to a linear parabolic problem by freezing all derivatives of order strictly
less than four.
More precisely, we formulate a fixed point problem based on an iterative
solution of a linear parabolic equation, whose fixed point coincides with
the desired solution of the Willmore flow.

The iteration set is constructed as a neighborhood of a reference solution,
given by a linearized  Willmore flow with the right-hand side frozen derivatives coming from the biharmonic heat flow with the same initial and boundary data.
In particular, the fixed point argument is carried out for small deviations
from this reference flow.

To justify this approach, we verify that the assumptions of the Banach
fixed point theorem are satisfied.
The proof is organized as follows:
\begin{itemize}
	\item[\cbm1] \emph{Biharmonic heat flow extension of the parabolic boundary data},
	used to construct the reference solution.
	
	\item[\cbm2] \emph{Definition of the iteration mapping $G$ and the 	iteration set $\M$} for the fixed point problem.
	
	\item[\cbm3] \emph{Invariance of the iteration set}: the mapping $G$ is a self-map, that is,
	\(
	G \colon \M \to \M.
	\)
	
	\item[\cbm4] \emph{Contraction property of the iteration mapping}: 	for all $u,w \in \M$,
	\[
	\|G(u)-G(w)\| \le q \|u-w\|,
	\qquad q \in (0,1).
	\]
	
	\item[\cbm5] \emph{Application of the Banach fixed point theorem}
	to obtain the existence of a fixed point of $G$.
	
	\item[\cbm6] \emph{Uniqueness of the solution} in the parabolic Hölder
	space $\Cc^{4+\alpha,\,1+\alpha/4}_{m+\alpha}(Q_T)$.
\end{itemize}

In the first two steps, the time horizon $T \in (0,1)$ is not fixed.
It is chosen sufficiently small in Steps~\cbm3 and~\cbm4 in order to
ensure the invariance and contraction properties required for the
application of the Banach fixed point theorem. The smallness properties in time arise from the interpolation
between the $\alpha$- and $\gamma$-Hölder scales,
which produces the factor $T^{\frac{\alpha-\gamma}{4}}$. The contraction argument is then carried out in a suitably chosen
closed subset of a lower Hölder space.

\medskip\noindent\cb{1} \textbf{ Boundary Values Discussion}\\
Since the boundary data
\[
u|_{\partial\Omega}(\,\cdot\,,t)=g_0(\,\cdot\,), 
\qquad 
\partial_\nu u|_{\partial\Omega}(\,\cdot\,,t)=g_1(\,\cdot\,)
\]
are time–independent for all $t\in[0,1]$, we extend them trivially to
$\partial\Omega\times(0,1]$ by setting
\[
\overline g_0(x,t):=g_0(x), 
\qquad 
\overline g_1(x,t):=g_1(x),
\quad x\in\partial\Omega.
\]
All time derivatives of $\overline g_0$ and $\overline g_1$ therefore vanish.
In particular, all temporal Hölder seminorms are zero and we obtain
\begin{align*}
\|\overline g_0\|_{\Cc^{4+\alpha,1+\alpha/4}_{m+\alpha}(\partial\Omega\times(0,1])}
&= \|g_0\|_{\Cc^{4+\alpha}(\partial\Omega)},\\
\|\overline g_1\|_{\Cc^{3+\alpha,(3+\alpha)/4}_{m-1+\alpha}(\partial\Omega\times(0,1])}
&= \|g_1\|_{\Cc^{3+\alpha}(\partial\Omega)}.
\end{align*}
Next, we construct a space–time extension of the initial datum
$u_0\in \Cc^{m+\alpha}(\overline\Omega)$.
A trivial constant extension would not yield sufficient spatial regularity
for $t>0$ in the case $m<4$.
To obtain the required smoothing, we define $\overline u_0$ as the solution of
the biharmonic heat equation
\begin{align}\left\{
\begin{aligned}
\partial_t v &= -\Delta^2 v 
&&\text{in } \Omega\times(0,1],\\
v(\cdot,0) &= u_0 
&&\text{in } \Omega,\\
v &= \overline g_0 
&&\text{on } \partial\Omega\times[0,1],\\
\partial_\nu v &= \overline g_1 
&&\text{on } \partial\Omega\times[0,1].
\end{aligned}
\right. \label{eq:u0} \tag{A} 
\end{align}
It follows directly from Theorem~\ref{gewichteteexistenz} that
there exists a unique solution $\overline u_0\in \Cc^{4+\alpha,1+\alpha/4}_{m+\alpha}(Q_1)$ that solves \eqref{eq:u0}.
Moreover, the weighted Schauder estimate of Theorem~\ref{gewichteteschauder} yields 
\begin{align}
	\begin{aligned}
		\|\overline u_0\|_{\Cc^{4+\alpha, 1+\alpha/4}_{m+\alpha} (Q_1)}
	\le&\ \Cr{ws}(1) \Big( \|\overline g_0\|_{\Cc^{4+\alpha,1+\alpha/4}_{m+\alpha}(\partial \Omega\times(0,1])} \! +\|\overline g_1\|_{\Cc^{3+\alpha,\frac{3+\alpha}{4}}_{m-1+\alpha}(\partial \Omega\times(0,1])}\!\! + \! \|u_0\|_{\Cc^{m+\alpha}(\overline\Omega
	)}\Big)\\
		\le&\ \Cl{A(1)}\big(\Omega, \|g_0\|_{\Cc^{4+\alpha}(\partial \Omega)},\|g_1\|_{\Cc^{3+\alpha}(\partial \Omega)},\|u_0\|_{\Cc^{m+\alpha}(\overline\Omega)}\big)
\end{aligned}\label{gewichtetschauderv11}
\end{align}
where $\Cr{ws}(1)$ depends only on $\Omega$, $m$, $\alpha$ and the operator
$\Delta^2$.
The compatibility conditions are also satisfied since $v(\,\cdot\,,0)=u_0$ and
\begin{align*}
	 \overline g_0(x,0)=&\ g_0(x)=	u_0(x), \hspace{1em} 
  \overline g_1(x,0)	= g_1(x)= \frac{\partial u_{0}}{\partial \nu}(x)  \hspace{1em} x\in \partial \Omega.
\end{align*}
In the case $m=4$, we simply define the constant extension
$\overline u_0(x,t):=u_0(x)$ for $(x,t)\in\Omega\times[0,1]$.
Then $\overline u_0\in \Cc^{4+\alpha,1+\alpha/4}_{m+\alpha}(Q_1)$
and the estimate \eqref{gewichtetschauderv11} remains valid.

\medskip\noindent\cb{2} \textbf{ Definition of the Iteration Map and Set}\\
For $T\in(0,1]$ we define the iteration map
\[
G_T\colon \Cc^{4+\alpha,1+\alpha/4}_{m+\alpha}(Q_T)
\longrightarrow
\Cc^{4+\alpha,1+\alpha/4}_{m+\alpha}(Q_T)
\]
by freezing the coefficients of the quasilinear operator as well as all
derivatives occurring in the lower–order term $\mathcal R$.
For $w\in \Cc^{4+\alpha,1+\alpha/4}_{m+\alpha}(Q_T)$ we define
$v:=G_Tw$ as the solution of
\begin{align}
\left\{
\begin{aligned}
\partial_t v &=
- L(\nabla w)\, D^4 v
- \mathcal R(\nabla w, D^2 w, D^3 w)
&&\text{in } \Omega\times(0,T],\\
v(\cdot,0) &= u_0
&&\text{in } \Omega,\\
v &= \overline g_0
&&\text{on } \partial\Omega\times[0,T],\\
\partial_\nu v &= \overline g_1
&&\text{on } \partial\Omega\times[0,T].
\end{aligned}
\right.
\label{eq:gewichtetlinear}
\tag{G}
\end{align}

For $w\in \Cc^{4+\alpha,1+\alpha/4}_{m+\alpha}(Q_T)$,
Lemma~\ref{gewichtethoedlerII} implies
\[
L(\nabla w)\in \Cc^{\alpha,\alpha/4}_{\max\{0,m+\alpha-4\}}(Q_T),
\qquad
\mathcal R(\nabla w,D^2 w,D^3 w)\in \Cc^{\alpha,\alpha/4}_{m+\alpha-4}(Q_T).
\]
Moreover, uniform ellipticity follows from Lemma~\ref{ellipL}.
Hence, by Theorem~\ref{gewichteteexistenz},
there exists a unique solution
\[
v=G_Tw\in \Cc^{4+\alpha,1+\alpha/4}_{m+\alpha}(Q_T),
\]
and the mapping $G_T$ is well defined.

In the case $m=4$, we additionally require the existence of
$G_T\overline u_0$ together with uniform estimates, in order to control
time derivatives at $t=0$ and to get later the same derivatives in $t=0$ as for the fixed point solution; see \eqref{eq:grunaustar}.
By Lemma~\ref{gewichtethoedlerII}, for $|\beta|=4$,
\[
\left\|L_{\beta_1,\beta_2}(\nabla \overline u_ 0)\right\|_{\Cc^{\alpha,\alpha/4}_{\max\{0,m+\alpha-4\}}(Q_T)}
\le 	\Cr{gewhoelderI}\Big(1+  \|\overline u_0\|_{\Cc^{4+\alpha, 1+\alpha/4}_{m+\alpha} (Q_1)}^4\Big).
\]
Uniform ellipticity follows from \eqref{gewichtetschauderv11}, i.e.
\[
\frac{|\xi|^4}{(1+\|\nabla\overline u_0\|_{\Cc^0(\overline Q_T)}^2)^2}
\le
\sum_{k+\ell=4} L_{k\ell}(\nabla\overline u_0)\,\xi_1^k\xi_2^\ell
\le C|\xi|^4 .
\]
Applying the Schauder estimate in Theorem~\ref{gewichteteschauder} to \eqref{eq:gewichtetlinear} with constant $\Cr{ws}$ yields
	\begin{align}
	\begin{aligned}
	\|G_T\overline u_0\|_{\Cc^{4+\alpha, 1+\alpha/4}_{m+\alpha}(Q_T)}\overset{\mathclap{\text{Thm. } \ref{gewichteteschauder}}}\le&\ \Cr{ws}(1) \left( \begin{aligned}
	\|&\mathcal R( \nabla \overline u_0,D^2 \overline u_0, D^3 \overline u_0)\|_{\Cc^{\alpha,\alpha/4}_{m+\alpha-4}(Q_T)}+ \| g_0\|_{\Cc^{4+\alpha}(\partial \Omega)}\\
	&+\| g_1\|_{\Cc^{3+\alpha}(\partial \Omega)}+ \|u_0\|_{\Cc^{m+\alpha}(\overline\Omega)}
\end{aligned}\right), \\
	  \underset{\mathclap{\eqref{eq:gewichteteinsabsch}}}{\overset{\mathclap{\text{Lem. }\ref{gewichtethoedlerII}}}\le} &\
	  \Cl{ast}\big(\Omega, \|g_0\|_{\Cc^{4+\alpha}(\partial \Omega)},\|g_1\|_{\Cc^{3+\alpha}(\partial \Omega)},\|u_0\|_{\Cc^{m+\alpha}(\overline\Omega)}\big).
\end{aligned}\label{gewichtetschauderv1} 
\end{align}
Let $\gamma\in(0,\alpha)$.
We define the following iteration set as  a closed subset of $\Cc^{4+\gamma,1+\gamma/4}_{m+\gamma}(Q_T)$
\begin{align}
	\M_T:=\!\left\{
w	 \in \Cc^{4+\gamma,1+\gamma/4}_{m+\gamma}(Q_T) 
\left|  \begin{aligned}  & \|w-G_T\overline u_0\|_{\Cc^{4+\gamma, 1+\gamma/4}_{m+\gamma} (Q_T)}\le 1,\\
& w(\cdot,0)=u_0, 
\ w|_{\partial\Omega}=g_0,
\ \partial_\nu w|_{\partial\Omega}=g_1
	\end{aligned} \right. \right\}. \label{eq:mmw} \tag{M}
\end{align}
The set $\M_T$ is non-empty since $G_T\overline u_0\in\M_T$.
For all $w\in \M_T$ we also obtain the following estimate, which will be used later:
\begin{align}
	\|w\|_{\Cc^{4+\gamma, 1+\gamma/4}_{m+\gamma} (Q_T)}\le\|w-G_T\overline u_0\|_{\Cc^{4+\gamma, 1+\gamma/4}_{m+\gamma} (Q_T)} +\|G_T\overline u_0\|_{\Cc^{4+\gamma, 1+\gamma/4}_{m+\gamma} (Q_T)}\overset{\eqref{gewichtetschauderv1}}\le 1+\Cr{ast}. \label{eq:gewichteteinsabsch}
\end{align}


\medskip\noindent\cb{3} \textbf{ $G$ is a self-map}\\
We show that $G_T\colon \M_T\to\M_T$ using the weighted Schauder estimates
from Theorem~\ref{gewichteteschauder}.
Let $w\in\M_T$.
We begin by proving the \emph{uniform control of the coefficients}.
By Theorem~\ref{gewichtethoedlerII}, for all multi-indices $|\beta|=4$,
\begin{align*}
\|L_{\beta_1,\beta_2}(\nabla w)\|_{\Cc^{\alpha,\alpha/4}_{\max\{0,m+\alpha-4\}}(Q_T)}
&\le
\Cr{gewhoelderI}\Bigl(1+\|w\|^4_{\Cc^{4+\gamma,1+\gamma/4}_{m+\gamma}(Q_1)}\Bigr) \\
&\overset{\eqref{eq:gewichteteinsabsch}}{\le}
\Cr{gewhoelderI}\Bigl(1+(1+\Cr{ast})^4\Bigr),
\end{align*}
which is bounded by a constant independent of $T$.

Moreover, by \eqref{eq:gewichteteinsabsch} and Lemma~\ref{ellipL}, the operator
is \textit{uniformly elliptic} on $\M_T$: there exist constants depending only on $\Omega,\|g_0\|_{\Cc^{4+\alpha}(\partial \Omega)},\|g_1\|_{\Cc^{3+\alpha}(\partial \Omega)}$ and $\|u_0\|_{\Cc^{m+\alpha}(\overline\Omega)}$
\[
\lambda:=\frac{1}{\bigl(1+(1+\Cr{ast})^2\bigr)^2},
\qquad
\Lambda:=4,
\]
such that for all $w\in\M_T$, $T\in(0,1)$,
\begin{align}
\lambda|\xi|^4
\le
\sum_{k+\ell=4} L_{k\ell}(\nabla w)\,\xi_1^k\xi_2^\ell
\le
\Lambda|\xi|^4 .
\label{eq:gewichtetlambda}
\end{align}

By the \textit{Schauder estimate} in Theorem \ref{gewichteteschauder} for the boundary problem \eqref{eq:gewichtetlinear} with constant $\Cr{ws}$ for all $w\in \M_T$ it follows that for $v=G_T w$
	\begin{align}
	\begin{aligned}
	\|v\|_{\Cc^{4+\alpha, 1+\alpha/4}_{m+\alpha}(Q_T)}\overset{\mathclap{\text{Thm. }\ref{gewichteteschauder}}}\le&\ \Cr{ws}(1) \left( \begin{aligned}
	\|&\mathcal R( \nabla w,D^2 w, D^3 w)\|_{\Cc^{\alpha,\alpha/4}_{m+\alpha-4}(Q_T)}+ \|\overline g_0\|_{\Cc^{4+\alpha,1+\alpha/4}_{m+\alpha}(\partial \Omega\times(0,T])}\\
	&+\|\overline g_1\|_{\Cc^{3+\alpha,\frac{3+\alpha}{4}}_{m-1+\alpha}(\partial \Omega\times(0,T])}+ \|u_0\|_{\Cc^{m+\alpha}(\overline\Omega)}
\end{aligned}\right), \\
	  \underset{\mathclap{\eqref{eq:gewichteteinsabsch}}}{\overset{\mathclap{\text{Lem. }\ref{gewichtethoedlerII}}}\le} &\  \Cl{overlineC}\big(\Omega, \|g_0\|_{\Cc^{4+\alpha}(\partial \Omega)},\|g_1\|_{\Cc^{3+\alpha}(\partial \Omega)},\|u_0\|_{\Cc^{m+\alpha}(\overline\Omega)}\big),
\end{aligned}\label{eq:gewichtetshud}
\end{align}
where, by Theorem~\ref{gewichteteschauder}, the results of Part~\cbm 1,
and \eqref{eq:gewichteteinsabsch}, the $T$-independent constant $\Cr{ws}(1)$
depends only on $\Cr{ast}$ and $\Omega$.

Further, let us
\emph{return to the iteration set.}
We now show that $v\in\M_T$ for $T>0$ sufficiently small.
First, we consider the difference $v-G_T\overline u_0$, which measures the
deviation from the reference solution
\begin{align*}
	\|v-G_T\overline u_0\|_{\Cc^{4+\alpha, 1+\alpha/4}_{m+\alpha}(Q_T)}\le&\ \|v\|_{\Cc^{4+\alpha, 1+\alpha/4}_{m+\alpha}(Q_T)}+\|G_T\overline u_0\|_{\Cc^{4+\alpha, 1+\alpha/4}_{m+\alpha}(Q_T)}	\le \Cr{overlineC} + \Cr{ast}\\
	 =:&\ \Cl{Ccirc}\big(\Omega, \|g_0\|_{\Cc^{4+\alpha}(\partial \Omega)},\|g_1\|_{\Cc^{3+\alpha}(\partial \Omega)},\|u_0\|_{\Cc^{m+\alpha}(\overline\Omega)}\big).
\end{align*}
So the difference is $\Cc^{4+\alpha, 1+\alpha/4}_{m+\alpha}(Q_T)$-bounded, because $G_T\overline u_0$ and $v$ are $\Cc^{4+\alpha, 1+\alpha/4}_{m+\alpha}(Q_T)$-bounded. Moreover, the derivatives of $v-G_T\overline u_0$ of order less than $m$ vanish at
$t=0$, since both functions satisfy the same initial condition:
\begin{align*}
\forall |\beta|\le m,\ \forall x\in\overline\Omega:\qquad
D_x^\beta v(x,0)-D_x^\beta G_T\overline u_0(x,0)=0.
\end{align*}
This \textit{vanishing property} is crucial for the application of
Lemma~\ref{gammaalpha} and is responsible for the gain of smallness in time
in the Hölder interpolation step below.

In the case $m=4$, we additionally have
\begin{align}
D_t v(x,0)-D_t G_T\overline u_0(x,0)=0,
\qquad x\in\overline\Omega,
\label{eq:grunaustar}
\end{align}
which follows directly from the Willmore flow equation.
 This is the only step where the auxiliary function
$G_T\overline u_0$ is required; for $m<4$, $\overline u_0$ would suffice.

We now pass to the lower Hölder regularity
$\Cc^{4+\gamma,1+\gamma/4}_{m+\gamma}(Q_T)$ and apply
Lemma~\ref{gammaalpha}.
Since all initial traces vanish, we obtain
\begin{align*}
\|v-G_T\overline u_0\|_{\Cc^{4+\gamma,1+\gamma/4}_{m+\gamma}(Q_T)}
&\le
\Cr{constgammaalpha}
\|v-G_T\overline u_0\|_{\Cc^{4+\alpha,1+\alpha/4}_{m+\alpha}(Q_T)}
\, T^{\frac{\alpha-\gamma}{4}} \\
&\le
\Cr{constgammaalpha}\Cr{Ccirc}\,T^{\frac{\alpha-\gamma}{4}}.
\end{align*}
Choosing $T>0$ sufficiently small such that
\[
T < (\Cr{constgammaalpha}\Cr{Ccirc})^{-\frac{4}{\alpha-\gamma}},
\]
we conclude that
\[
\|v-G_T\overline u_0\|_{\Cc^{4+\gamma,1+\gamma/4}_{m+\gamma}(Q_T)}\le 1,
\]
and hence $v\in\M_T$.
Therefore, $G_T\colon\M_T\to\M_T$.

\medskip\noindent\cb{4} \textbf{ $G$ is a contraction}\\
In this section, we prove that for $T>0$ sufficiently small,
the mapping $G_T\colon \M_T\to\M_T$ is a contraction.
Let $u,w\in\M_T$.
Since $G_Tu$ and $G_Tw$ satisfy the same initial and boundary conditions, we have
\begin{align*}
G_Tu(x,0)-G_Tw(x,0) &= u_0(x)-u_0(x)=0,
&& x\in\overline\Omega,\\
G_Tu(x,t)-G_Tw(x,t) &= g_0(x)-g_0(x)=0,
&& (x,t)\in\partial\Omega\times[0,T],\\
\partial_\nu(G_Tu-G_Tw)(x,t) &= g_1(x)-g_1(x)=0,
&& (x,t)\in\partial\Omega\times[0,T].
\end{align*}
Hence, the difference $v:=G_Tu-G_Tw$ solves the linear initial–boundary value
problem
\begin{align}
\left\{
\begin{aligned}
\partial_t v &=
- \tfrac12\big(L(\nabla u)+L(\nabla w)\big) D^4 v\\
&\quad - \tfrac12\big(L(\nabla u)-L(\nabla w)\big) D^4(G_Tu+G_Tw) \\
&\quad
- \mathcal R(\nabla u,D^2u,D^3u)
+ \mathcal R(\nabla w,D^2w,D^3w)
&& \text{in }\Omega\times(0,T],\\
v(\cdot,0) &= 0 && \text{in }\Omega,\\
v &= 0 && \text{on }\partial\Omega\times[0,T],\\
\partial_\nu v &= 0 && \text{on }\partial\Omega\times[0,T].
\end{aligned}
\right.
\label{eq:gewichtetmodi}
\tag{Z}
\end{align}

The operator 
\(
\tfrac12\big(L(\nabla u)+L(\nabla w)\big)
\in \Cc^{\alpha,\alpha/4}_{\max\{0,m+\alpha-4\}}(Q_T)
\)
is uniformly elliptic with the same time-independent constants
$\lambda$ and $\Lambda$ as in \eqref{eq:gewichtetlambda}.
Moreover, by applying the Schauder estimates to $G_Tu$ and $G_Tw$,
as solutions of \eqref{eq:gewichtetlinear},  it follows that
\begin{align}
\begin{aligned}
 \|D^4(G_Tu+ & G_Tw)\|_{\Cc^{\alpha,\alpha/4}_{m+\alpha-4}(Q_T)}\\
 \le &\	\|G_Tu\|_{\Cc^{4+\alpha, 1+\alpha/4}_{m+\alpha}(Q_T)}\ +\|G_Tw\|_{\Cc^{4+\alpha, 1+\alpha/4}_{m+\alpha}(Q_T)}\\
	\overset{
 }\le &\ 2 \Cr{overlineC}\big(\Omega, \|g_0\|_{\Cc^{4+\alpha}(\partial \Omega)},\|g_1\|_{\Cc^{3+\alpha}(\partial \Omega)},\|u_0\|_{\Cc^{m+\alpha}(\overline\Omega)}\big). 
\end{aligned}\label{eq:g+w}
\end{align}

As before, Theorem~\ref{gewichteteschauder} provides a weighted Schauder
estimate for $G_Tu-G_Tw$ as the solution of \eqref{eq:gewichtetmodi},
with a constant $\Cr{ws}(1)$ depending only on
$\|g_0\|_{\Cc^{4+\alpha}(\partial \Omega)}$,
$\|g_1\|_{\Cc^{3+\alpha}(\partial \Omega)}$,
$\|u_0\|_{\Cc^{m+\alpha}(\overline{\Omega})}$,
and $\Omega$.
\begin{align*}
	\|G_Tu- & G_Tw\|_{\Cc^{4+\alpha, 1+\alpha/4}_{m+\alpha}(Q_T)}\\
	\le &\ \Cr{ws}(1) \cdot\left\|\begin{aligned}
& - \tfrac{1}{2}\big(L(\nabla u)-L(\nabla w)\big) \big( D^4 (G_Tu+G_Tw)\big)\\
	&\ - \mathcal R\big(\nabla u, D^2 u, D^3 u\big)+\mathcal R\big(\nabla w, D^2 w, D^3 w\big)
\end{aligned}\right\|_{\Cc^{\alpha,\alpha/4}_{m+\alpha-4}(Q_T)}\\
\overset{\eqref{eq:simpleprod}}{\le}&\ \frac{\Cr{ws}(1)}{2}  \cdot\left\| \big(L(\nabla u)-L(\nabla w)\big)\right\|_{\Cc^{\alpha,\alpha/4}_{\max\{0,m+\alpha-4\}}(Q_T)}\\
&\quad \cdot \left\| \big( D^4 (G_Tu+G_Tw)\big)\right\|_{\Cc^{\alpha,\alpha/4}_{m+\alpha-4}(Q_T)}\\
	&\	 + \Cr{ws}(1)\left\|\Rr( \nabla u,D^2 u, D^3 u)-\Rr( \nabla w,D^2 w, D^3 w)\right \|_{\Cc^{\alpha,\alpha/4}_{m+\alpha-4}(Q_T)}.
\end{align*}

Using Lemma~\ref{gewichtethoedlerIII} and \eqref{eq:g+w} for $D^4(G_Tu+G_Tw)$, we obtain
\begin{align*}
\|G_Tu-G_Tw\|_{\Cc^{4+\alpha,1+\alpha/4}_{m+\alpha}(Q_T)}
\le
\Cl{Cdag}\,
\|u-w\|_{\Cc^{4+\gamma,1+\gamma/4}_{m+\gamma}(Q_T)},
\end{align*}
for all $u,w\in\M_T$, where $\Cr{Cdag}$ depends only on
$\Omega$, $\|g_0\|_{\Cc^{4+\alpha}(\partial \Omega)}$, $\|g_1\|_{\Cc^{3+\alpha}(\partial \Omega)}$, $\|u_0\|_{\Cc^{m+\alpha}(\overline\Omega)} $.

Since $v=G_Tu-G_Tw$ has vanishing initial and boundary data,
we may apply Lemma~\ref{gammaalpha} to obtain
\begin{align*}
	\|G_Tu\ -&\ G_Tw\|_{\Cc^{4+\gamma, 1+\gamma/4}_{m+\gamma} (Q_T)}\le { \Cr{constgammaalpha}\ \Cr{Cdag} T^{\frac{\alpha-\gamma}{4}}}\cdot\|u-w\|_{\Cc^{4+\gamma, 1+\gamma/4}_{m+\gamma} (Q_T)}.
\end{align*}

Choosing
\[
T < (\Cr{constgammaalpha}\Cr{Cdag})^{-\frac{4}{\alpha-\gamma}},
\]
we conclude that $G_T$ is a contraction on $\M_T$.

\medskip\noindent\cb{5} \textbf{Application of the Fixed Point Theorem}\\
For $T>0$ sufficiently small, the mapping
$G_T\colon \M_T\to\M_T$ is a contraction.
Hence, by the Banach fixed point theorem, there exists a unique fixed point
\[
v^\ast \in \M_T \subset
\Cc^{4+\gamma,1+\gamma/4}_{m+\gamma}(Q_T)
\]
such that
\[
v^\ast = G_T v^\ast .
\]
Since $G_T$ was defined by freezing the coefficients in the original
quasilinear equation, the fixed point identity $v^\ast=G_Tv^\ast$
implies that $v^\ast$  solves the original full nonlinear Willmore flow problem  \eqref{Willmore}.

Moreover, applying the weighted Schauder estimate
(Theorem~\ref{gewichteteschauder}) to $v^\ast$ yields the improved
regularity
\[
v^\ast \in
\Cc^{4+\alpha,1+\alpha/4}_{m+\alpha}(Q_T),
\]
together with the estimate
\begin{align*}
\|v^\ast\|_{\Cc^{4+\alpha,1+\alpha/4}_{m+\alpha}(Q_T)}
\le
\Cr{overlineC}\big(
\Omega,
\|g_0\|_{\Cc^{4+\alpha}(\partial\Omega)},
\|g_1\|_{\Cc^{3+\alpha}(\partial\Omega)},
\|u_0\|_{\Cc^{m+\alpha}(\overline\Omega)}
\big).
\end{align*}

\medskip\noindent\cb{6} \textbf{Uniqueness} \\
So far, uniqueness has only been established within the iteration set $\M_T$.
We now prove uniqueness in the larger space
$\Cc^{4+\alpha,1+\alpha/4}_{m+\alpha}(Q_T)$.
Our main idea is to consider the difference of two solutions and proceed as in the contraction step~\cbm 4.
Applying Lemma~\ref{gammaalpha} will introduce the factor
$T_0^{\frac{\alpha-\gamma}{4}}$,
which will allow us to conclude uniqueness for $T_0$ sufficiently small.

Let $T$ be the fixed point time chosen in Steps~\cbm 3 and \cbm 4.
Let $u\in\M_T$ be the solution constructed in Step~\cbm 5,
and let
\[
w \in \Cc^{4+\alpha,1+\alpha/4}_{m+\alpha}(Q_{T'})
\]
be another solution of \eqref{Willmore}.
Without loss of generality, we may assume $T' < T$.

Let $0<T_0<T'$ be a positive time, to be fixed sufficiently small below. \textit{First, we show uniqueness in the space}.
$\Cc^{4+\alpha,1+\alpha/4}_{m+\alpha}(Q_{T_0})$.
In this case, the difference $v:=u-w$ satisfies the following initial–boundary value problem on $Q_{T_0}$ with vanishing initial and boundary data:
\begin{align}
\left\{
\begin{aligned}
\partial_t v &=
- \tfrac12\big(L(\nabla u)+L(\nabla w)\big) D^4 v\\
&\quad - \tfrac12\big(L(\nabla u)-L(\nabla w)\big) D^4(u+w) \\
&\quad
- \mathcal R(\nabla u,D^2u,D^3u)
+ \mathcal R(\nabla w,D^2w,D^3w)
&& \text{in }\Omega\times(0,T_0],\\
v(\cdot,0) &= 0 && \text{in }\Omega,\\
v &= 0 && \text{on }\partial\Omega\times[0,T_0],\\
\partial_\nu v &= 0 && \text{on }\partial\Omega\times[0,T_0].
\end{aligned}
\right.
\label{eq:gewichteteind}
\tag{E${}_{w}$}
\end{align}
It follows that
\[
\tfrac12\big(L(\nabla u)+L(\nabla w)\big)
\in \Cc^{\alpha,\alpha/4}_{\max\{0,m+\alpha-4\}}(Q_{T_0})
\]
is a uniformly elliptic operator with time-independent ellipticity
constants
\[
\lambda^\ast :=
\frac{1}{2\bigl(1+\|u\|^2_{\Cc^{4+\alpha,1+\alpha/4}_{m+\alpha}(Q_{T'})}
+\|w\|^2_{\Cc^{4+\alpha,1+\alpha/4}_{m+\alpha}(Q_{T'})}\bigr)},
\qquad
\Lambda^\ast := 4 .
\]
Moreover, we obtain the estimates
\begin{align*} \|D^4(u+w)\|_{\Cc^{\alpha,\alpha/4}_{m+\alpha-4}(Q_{T_{0}})} \le&\ \|u\|_{\Cc^{4+\alpha, 1+\alpha/4}_{m+\alpha}(Q_{T'})}\ +\|w\|_{\Cc^{4+\alpha, 1+\alpha/4}_{m+\alpha}(Q_{T'})},\\ \|L(\nabla u)+L(\nabla w)\|_{\Cc^{\alpha,\alpha/4}_{\max\{0,m+\alpha-4\}}(Q_{T_{0}})} \le&\ \Cr{gewhoelderI}\left( \begin{aligned} &\ \Big(1+ \|u\|_{\Cc^{4+\alpha, 1+\alpha/4}_{m+\alpha} (Q_{T'})}\Big)^{k_H} \\ &\ + \Big(1+ \|w\|_{\Cc^{4+\alpha, 1+\alpha/4}_{m+\alpha} (Q_{T'})}\Big)^{k_H} \end{aligned}\right). \end{align*}
where we used Lemma~\ref{gewichtethoedlerII}.

Once again, Theorem~\ref{gewichteteschauder} yields a weighted Schauder
estimate for the problem \eqref{eq:gewichteteind},
with a constant $\Cr{ws}(1)$ depending only on $\Omega$ and on the
$\Cc^{4+\alpha,1+\alpha/4}_{m+\alpha}(Q_{T'})$-norms of $u$ and $w$.
Proceeding as in Step~\cbm 4, we obtain
Like in step \cbm 4:
	\begin{align*}
	\|u- \ &\  w\|_{\Cc^{4+\alpha, 1+\alpha/4}_{m+\alpha}(Q_{T_{0}})}\\ 
\le&\ \frac{\Cr{ws}(1)}{2}  \cdot\left\| L(\nabla u)-L(\nabla w)\right\|_{\Cc^{\alpha,\alpha/4}_{\max\{0,m+\alpha-4\}}( Q_{T_{0}})} \cdot \left\| \big( D^4 (u+w)\big)\right\|_{\Cc^{\alpha,\alpha/4}_{m+\alpha-4}(Q_{T_{0}})}\\
	&\	 + \Cr{ws}(1)\left\|\Rr( \nabla u,D^2 u, D^3 u)-\Rr( \nabla w,D^2 w, D^3 w)\right \|_{\Cc^{\alpha,\alpha/4}_{m+\alpha-4}(Q_{T_{0}})} \\
\le &\ {\Cr{Cdag}}\cdot\|u-w\|_{\Cc^{4+\gamma, 1+\gamma/4}_{m+\gamma} (Q_{T_{0}})},
\end{align*}
where the constant $\Cr{Cdag}$ depends only  on $\Omega$ and on the
$\Cc^{4+\alpha,1+\alpha/4}_{m+\alpha}(Q_{T'})$-norms of $u$ and $w$.
Since $u-w$ has vanishing initial and boundary data,
Lemma~\ref{gammaalpha} implies
\begin{align*}
\|u-w\|_{\Cc^{4+\gamma,1+\gamma/4}_{m+\gamma}(Q_{T_0})}
\le
\Cr{constgammaalpha}\Cr{Cdag}
\,T_0^{\frac{\alpha-\gamma}{4}}
\,
\|u-w\|_{\Cc^{4+\gamma,1+\gamma/4}_{m+\gamma}(Q_{T_0})}.
\end{align*}

Choosing
\[
T_0<
(\Cr{constgammaalpha}\Cr{Cdag})^{-\frac{4}{\alpha-\gamma}},
\]
we conclude \(
\|u -w\|_{\Cc^{4+\gamma, 1+\gamma/4}_{m+\gamma} (Q_{T_{0}})}=0\). Therefore, $u$ and $w$ coincide in
$\Cc^{4+\gamma,1+\gamma/4}_{m+\gamma}(Q_{T_0})$:
\[
u=w
\quad\text{on } Q_{T_0}.
\]

\textit{To extend the equality to times} $t\in (T_0,T']$,
we first consider the unweighted case $m=4$.
In this case, the time $T_0$ depends only on $\Omega$ and the bounds on $\Cc^{4+\alpha, 1+ \alpha/4}_{x,t} (Q_{T'})$-norm of $u$ and $w$.
Since for all $x\in\overline{\Omega}$ and all multi-indices $|\beta|\le 4$,
\[
D_x^\beta u(x,T_0)=D_x^\beta w(x,T_0),
\]
we may restart the uniqueness argument at time $T_0$,
using the same uniqueness time $T_0$.
Consequently, uniqueness holds on the interval
\[
[0,\min\{2T_0,T'\}].
\]
Repeating this procedure finitely many times,
we obtain $u=w$ on $[0,T']$.

By the definition of the weighted norms, in the cases $m=1,2,3$,
for every $T_0>0$ the solutions $u$ and $w$ belong to the unweighted
parabolic Hölder space on $\overline\Omega\times[T_0,T']$.
More precisely, we have
\begin{align*}
\|u\|_{\Cc^{4+\alpha,1+\alpha/4}_{x,t}(\overline\Omega\times[T_0,T'])}
&\le
C(T_0,T')\,
\|u\|_{\Cc^{4+\alpha,1+\alpha/4}_{m+\alpha}(Q_{T'})}, \\
\|w\|_{\Cc^{4+\alpha,1+\alpha/4}_{x,t}(\overline\Omega\times[T_0,T'])}
&\le
C(T_0,T')\,
\|w\|_{\Cc^{4+\alpha,1+\alpha/4}_{m+\alpha}(Q_{T'})}.
\end{align*}

Furthermore, the compatibility condition \eqref{eq:wkb}
is satisfied for all times in $[T_0,T']$.
Since $u$ and $w$ coincide at time $T_0$,
we may take
\[
u(\cdot,T_0)=w(\cdot,T_0)
\]
as new initial data for the corresponding unweighted
parabolic problem on $\overline\Omega\times[T_0,T']$.

The classical uniqueness result for the unweighted problem then implies
\[
u=w
\quad \text{on } \overline\Omega\times[T_0,T'].
\]
Consequently,
\[
u=w
\quad \text{in } \Cc^{4+\alpha,1+\alpha/4}_{m+\alpha}(Q_{T'}).
\]
Finally, extending this argument up to time $T$
yields uniqueness in
\(
\Cc^{4+\alpha,1+\alpha/4}_{m+\alpha}(Q_T).
\)
\end{proof}

\subsection{Lipschitz initial data}
\label{subsec:c1init}

We now turn to the short-time theory for merely Lipschitz initial data
\[
u_0\in \Cc^{0,1}(\overline\Omega).
\]
In contrast to the previous subsection, we no longer assume any Hölder
regularity of $\nabla u_0$.
The price for this weaker initial regularity is that the fixed-point
argument can only be closed under a smallness assumption on the data.

In the notation of the weighted linear theory from
Section~\ref{sec:paraspace}, this corresponds to the choice
\[
s=1,
\qquad
f=\Rr(\nabla u,D^2u,D^3u),
\qquad
(a_\beta)\cong L(\nabla u),
\]
together with time-independent clamped boundary data
\(
\varphi=g_0,
\
h=g_1.
\)

The essential difference to the $C^{1+\alpha}$-theory is the following.
In the space $\Cc^{4+\alpha,1+\alpha/4}_{1}(Q_T)$ one controls only the
Lipschitz norm of $u$ up to $t=0$, whereas higher derivatives may blow up
at the rate prescribed by the weights.
In particular, the norm no longer contains the spatial Hölder seminorm of
$\nabla u(\,\cdot\,,t)$ or the temporal Hölder seminorm of
$\nabla u(x,\,\cdot\,)$ at $t=0$.
For this reason, the interpolation argument based on lowering the Hölder
exponent from $\alpha$ to some $\gamma<\alpha$ is not available here.
Instead, the contraction will be obtained directly in the same weighted
space, and the smallness needed for the fixed-point argument will come
from the smallness of the data and, for time-independent boundary values,
from choosing the time interval sufficiently short.

Let $\Omega\subset\R^n$ be bounded with $\Cc^{4+\alpha}$-boundary and let
$T>0$.
For later reference we record the weighted norm in the case $s=1$.
The low-order anisotropic norm is
\begin{align*}
\|u\|_{\Cc^{1,1/4}_{x,t}(\overline Q_T)}
=
\sum_{|\beta|\le1}\sup_{(x,t)\in\overline Q_T}|D_x^\beta u(x,t)|
+
\sup_{x\in\overline\Omega}[u(x,\,\cdot\,)]_{\Cc^{1/4}([0,T])},
\end{align*}
and the full weighted norm is
\begin{align*}
\|u\|_{\Cc^{4+\alpha,1+\alpha/4}_{1}(Q_T)}
={}&
\sup_{0<t\le T} t^{\frac{3+\alpha}{4}}[u]^{4+\alpha}_{Q_t'}
+
\sum_{2\le 4k+|\beta|\le4}
\sup_{(x,t)\in Q_T}
t^{\frac{4k+|\beta|-1}{4}}|D_t^kD_x^\beta u(x,t)| \\
&\quad
+\|u\|_{\Cc^{1,1/4}_{x,t}(\overline Q_T)},
\end{align*}
where
\begin{align*}
[u]^{4+\alpha}_{Q_t'}
={}&
\sum_{4k+|\beta|=4}
\sup_{t'\in[t/2,t]}
\big[D_t^kD_x^\beta u(\,\cdot\,,t')\big]_{\Cc^\alpha(\overline\Omega)} \\
&\quad
+\sum_{1\le 4k+|\beta|\le4}
\sup_{x\in\overline\Omega}
\big[D_t^kD_x^\beta u(x,\,\cdot\,)\big]_
{\Cc^{\frac{4+\alpha-4k-|\beta|}{4}}([t/2,t])}.
\end{align*}

\emph{Remark.}
Compared with the space
$\Cc^{4+\alpha,1+\alpha/4}_{1+\alpha}(Q_T)$ from the previous subsection,
the norm $\Cc^{4+\alpha,1+\alpha/4}_{1}(Q_T)$ carries an additional factor
$t^{\alpha/4}$ in front of the highest-order quantities and no longer
controls the Hölder seminorms of $\nabla u$ up to the initial time.
This is precisely what makes the space compatible with
Lipschitz initial data.

As in the Hölder case, the technical product and composition estimates
needed for the nonlinear fixed-point argument are proved in the appendix.
We record below only the forms used later.

\begin{lemma}[Basic product estimates in the Lipschitz regime]
\label{productrule22}
Let $0<\alpha<1$ and let $T\le1$.
Then there exists a constant
\[
\Cl{thmproductrule22}=\Cr{thmproductrule22}(\alpha,\Omega)>0
\]
such that for all
$u,v,w\in \Cc^{4+\alpha,1+\alpha/4}_{1}(Q_T)$,
\begin{align*}
\|D^3w\,D^2u\|_{\Cc^{\alpha,\alpha/4}_{-3}(Q_T)}
&\le
\Cr{thmproductrule22}\,
\|D^3w\|_{\Cc^{1+\alpha,\frac{1+\alpha}{4}}_{-2}(Q_T)}
\|D^2u\|_{\Cc^{2+\alpha,\frac{2+\alpha}{4}}_{-1}(Q_T)},\\
\|D^2u\,D^2w\,D^2v\|_{\Cc^{\alpha,\alpha/4}_{-3}(Q_T)}
&\le
\Cr{thmproductrule22}\,
\|D^2u\|_{\Cc^{2+\alpha,\frac{2+\alpha}{4}}_{-1}(Q_T)}
\|D^2w\|_{\Cc^{2+\alpha,\frac{2+\alpha}{4}}_{-1}(Q_T)}
\|D^2v\|_{\Cc^{2+\alpha,\frac{2+\alpha}{4}}_{-1}(Q_T)},\\
\|\nabla u\|_{\Cc^{\alpha,\alpha/4}_{0}(Q_T)}
&\le
\Cr{thmproductrule22}\,
\|\nabla u\|_{\Cc^{3+\alpha,3+\alpha/4}_{0}(Q_T)}.
\end{align*}
\end{lemma}

\begin{proof}
See Appendix, Lemma~\ref{productrule22appendix}.
\end{proof}

The preceding product estimates imply the following bounds for the
coefficients of the principal part and for the lower-order remainder term.

\begin{lemma}[Hölder estimates I]
\label{1gewichtethoedlerII}
Let $0<\alpha<1$ and $T\le1$.
Then there exist constants
\[
\Cl{1gewhoelderI}=\Cr{1gewhoelderI}(\Omega,\alpha)>0,
\qquad
k_H\in\N,
\]
depending only on $\Omega$, $\alpha$, and the fixed algebraic structure
of $L$ and $\Rr$, such that for every
$u\in \Cc^{4+\alpha,1+\alpha/4}_{1}(Q_T)$,
\begin{align*}
\|\Rr(\nabla u,D^2u,D^3u)\|_{\Cc^{\alpha,\alpha/4}_{-3}(Q_T)}
&\le
\Cr{1gewhoelderI}
\Big(1+\|u\|_{\Cc^{4+\alpha,1+\alpha/4}_{1}(Q_T)}\Big)^{k_H}
\|u\|_{\Cc^{4+\alpha,1+\alpha/4}_{1}(Q_T)}^3,\\
\sum_{k+\ell=4}
\|L_{k\ell}(\nabla u)\|_{\Cc^{\alpha,\alpha/4}_{0}(Q_T)}
&\le
\Cr{1gewhoelderI}
\Big(1+\|u\|_{\Cc^{4+\alpha,1+\alpha/4}_{1}(Q_T)}^4\Big).
\end{align*}
\end{lemma}

\begin{proof}
The proof is analogous to that of
Appendix Lemma~\ref{gewichtethoedlerIIappendix}.
\end{proof}

For the contraction step we also need difference estimates.

\begin{lemma}[Hölder estimates II]
\label{1gewichtethoedlerIII}
Let $0<\alpha<1$ and $T\le1$.
Then there exist constants
\[
\Cl{1hoelderII}=\Cr{1hoelderII}(\Omega,\alpha)>0,
\qquad
k_H'\in\N,
\]
depending only on $\Omega$, $\alpha$, and the fixed algebraic structure
of $L$ and $\Rr$, such that for all
$u,w\in \Cc^{4+\alpha,1+\alpha/4}_{1}(Q_T)$,
\begin{align*}
	&\	\left\|\Rr( \nabla u,D^2 u, D^3 u)-\Rr( \nabla w,D^2 w, D^3 w)\right \|_{\Cc^{\alpha,\alpha/4}_{-3}(Q_T)} \\
&\ \qquad \qquad 
\le 
\Cr{1hoelderII} \Big(1+ \max\Big\{\|u\|_{\Cc^{4+\alpha, 1+\alpha/4}_{1} (Q_T)},\|w\|_{\Cc^{4+\alpha, 1+\alpha/4}_{1} (Q_T)}\Big\}\Big)^{k'_H} \\
& \ \qquad \qquad \qquad  
\cdot
\max\Big\{\|u\|_{\Cc^{4+\alpha, 1+\alpha/4}_{1} (Q_T)},\|w\|_{\Cc^{4+\alpha, 1+\alpha/4}_{1} (Q_T)}\Big\}^2
 \cdot
 \|u-w\|_{\Cc^{4+\alpha, 1+\alpha/4}_{1} (Q_T)}, \\
 &\ \quad 
\sum_{k+\ell=4}\left\|L_{k\ell}(\nabla u)-L_{k\ell}(\nabla w)\right\|_{\Cc^{\alpha,\alpha/4}_{0}( Q_T)}\\
&\ \qquad \qquad 
\le 
\Cr{1hoelderII} \Big(1+ \max\Big\{\|u\|_{\Cc^{4+\alpha, 1+\alpha/4}_{1} (Q_T)},\|w\|_{\Cc^{4+\alpha, 1+\alpha/4}_{1} (Q_T)}\Big\}\Big)^{k'_H} 
 \cdot
 \|u-w\|_{\Cc^{4+\alpha, 1+\alpha/4}_{1} (Q_T)}.
\end{align*}
\end{lemma}

\begin{proof}
The proof is similar to that of
Lemma~\ref{gewichtethoedlerIIappendix}; in addition one uses
\cite[Appendix, Lemma 96]{gulyak2024boundary}. to rewrite the difference of the relevant
polynomials.
\end{proof}



We now establish short-time existence for the graphical Willmore flow
with low-regularity Lipschitz initial data and time-dependent Dirichlet boundary
conditions. 

\begin{thm}[Short-Time Existence for time-dependent Dirichlet boundary data]
\label{shorttimeexistenceC}
Let $0<\alpha<1$ and $0<T\le 1$.
There exists a time-independent constant
\[
\Cl{c1short}=\Cr{c1short}(\Omega,\alpha)>0
\]
such that the following holds.
Assume
\[
u_0 \in \Cc^{0,1}(\overline\Omega), 
\quad
\varphi \in \Cc^{4+\alpha,1+\alpha/4}_{1}(\partial\Omega\times(0,T]),
\quad
h \in \Cc^{3+\alpha,\frac{3+\alpha}{4}}_{0}(\partial\Omega\times(0,T]),
\]
and that the compatibility conditions
\begin{align*}
\varphi(x,0) &= u_0(x), \quad
h(x,0) = {\partial_\nu u_0}(x),
\qquad x\in\partial\Omega,
\end{align*}
are satisfied.
If, in addition,
\begin{align*}
\| u_0\|_{\Cc^{0,1}(\overline\Omega)}
+
\|\varphi\|_{\Cc^{4+\alpha,1+\alpha/4}_{1}(\partial\Omega\times(0,T])}
+
\|h\|_{\Cc^{3+\alpha,\frac{3+\alpha}{4}}_{0}(\partial\Omega\times(0,T])}
< \Cr{c1short},
\end{align*}
then there exists a solution
\[
u \in \Cc^{4+\alpha,1+\alpha/4}_{1}(Q_T)
\]
of the initial-boundary value problem for the Willmore flow
with time-dependent Dirichlet boundary conditions:
\begin{align*}
\left\{
\begin{aligned}
\partial_t u
&=
- Q\Big\{
\Delta_{\Gamma(u)} H
+ 2H\Big(\tfrac14 H^2 - \Kk\Big)
\Big\}
&& \text{in } \Omega\times(0,T],\\
u(\cdot,0) &= u_0,
&& \text{in } \Omega,\\
u &= \varphi,
&& \text{on } \partial\Omega\times[0,T],\\
\partial_\nu u &= h,
&& \text{on } \partial\Omega\times[0,T].
\end{aligned}
\right.
\end{align*}
\end{thm}

\begin{proof}
We proceed as in the proof of Theorem~\ref{gewichtetexistkurz},
dividing the argument into several steps.


\medskip\noindent\cb{1} \textbf{Definition of the iteration map and set}

Proceeding similarly as in Step~\cbm 1 of the proof of
Theorem~\ref{gewichtetexistkurz}, we extend the boundary data
and obtain a function
\begin{align*}
\overline u_0 \in \Cc^{4+\alpha,1+\alpha/4}_{1}(Q_T) \quad 
\text{solving} \quad \left\{
\begin{aligned}
\partial_t v &= -\Delta^2 v 
&&\text{in } \Omega\times(0,T],\\
v(\cdot,0) &= u_0 
&&\text{in } \Omega,\\
v &=  \varphi
&&\text{on } \partial\Omega\times[0,T],\\
\partial_\nu v &= h
&&\text{on } \partial\Omega\times[0,T],
\end{aligned}
\right. 
\end{align*}
the corresponding biharmonic heat problem. This function will serve as an element to ensure that the iteration set is non-empty. Moreover, the weighted Schauder estimate of
Theorem~\ref{gewichteteschauder} yields
\begin{align*}
\|\overline u_0\|_{\Cc^{4+\alpha,1+\alpha/4}_{1}(Q_T)}
\le
\Cl{overlineuest}(\alpha, \Omega)\cdot \big(
\|\varphi\|_{\Cc^{4+\alpha,1+\alpha/4}_{1}(\partial\Omega\times(0,T])}
+
\|h\|_{\Cc^{3+\alpha,\frac{3+\alpha}{4}}_{0}(\partial\Omega\times(0,T])}
+
\|u_0\|_{\Cc^{0,1}(\overline\Omega)}
\big).
\end{align*}
To apply the Banach fixed point theorem, we define the iteration map
\[
G_T\colon
\Cc^{4+\alpha,1+\alpha/4}_{1}(Q_T)
\longrightarrow
\Cc^{4+\alpha,1+\alpha/4}_{1}(Q_T)
\]
by setting $v=G_T w$ to be the solution of
\begin{align*}
\left\{
\begin{aligned}
\partial_t v &=
- L(\nabla w)\, D^4 v
- \mathcal R(\nabla w, D^2 w, D^3 w)
&&\text{in } \Omega\times(0,T],\\
v(\cdot,0) &= u_0
&&\text{in } \Omega,\\
v &= \varphi
&&\text{on } \partial\Omega\times[0,T],\\
\partial_\nu v &= h
&&\text{on } \partial\Omega\times[0,T].
\end{aligned}
\right.
\end{align*}
Since $w\in \Cc^{4+\alpha,1+\alpha/4}_{1}(Q_T)$,
Lemma~\ref{1gewichtethoedlerII} implies
\[
L(\nabla w)\in \Cc^{\alpha,\alpha/4}_{0}(Q_T),
\qquad
\mathcal R(\nabla w,D^2w,D^3w)
\in \Cc^{\alpha,\alpha/4}_{-3}(Q_T).
\]
Uniform ellipticity follows from Lemma~\ref{ellipL}.
Hence, by Theorem~\ref{gewichteteexistenz}, there exists a unique solution
\[
v=G_T w \in \Cc^{4+\alpha,1+\alpha/4}_{1}(Q_T).
\]
Therefore, the mapping $G_T$ is well defined.

Since we work with $T\le 1$, Theorem~\ref{gewichteteschauder}
yields the estimate
	\begin{align}
	\begin{aligned}
	\|G_T w\|_{\Cc^{4+\alpha, 1+\alpha/4}_{1}(Q_T)}\le&\ \Cr{ws}(1) \left( \begin{aligned}
	\|&\mathcal R( \nabla w,D^2 w, D^3 w)\|_{\Cc^{\alpha,\alpha/4}_{-3}(Q_T)}+  \|u_0\|_{\Cc^{0,1}(\overline\Omega)} \\
    & +\|\varphi\|_{\Cc^{4+\alpha,1+\alpha/4}_{1}(\partial\Omega\times(0,T])} 
 +
\|h\|_{\Cc^{3+\alpha,\frac{3+\alpha}{4}}_{0}(\partial\Omega\times(0,T])}
\end{aligned}\right).
\end{aligned}\label{1gewichtetschauderv1} 
\end{align}
We now define a non-empty closed subset of
$\Cc^{4+\alpha,1+\alpha/4}_{1}(Q_T)$,
characterized by the boundary conditions and a uniform bound on the norm:
\begin{align}
\M_T :=
\left\{
w \in \Cc^{4+\alpha,1+\alpha/4}_{1}(Q_T)
\;\left|\;
\begin{aligned}
& \|w\|_{\Cc^{4+\alpha,1+\alpha/4}_{1}(Q_T)} \le \widehat C
,\\
& w(\cdot,0)=u_0, 
\ w|_{\partial\Omega}=g_0,
\ \partial_\nu w|_{\partial\Omega}=g_1
\end{aligned}
\right\}\right.. \label{eq:MsetC}
\end{align}
Here, $\widehat C>0$ is a constant to be specified below by conditions~\eqref{eq:1C1} and \eqref{eq:1C3}.
The set $\M_T$ is non-empty, since for $\Cr{c1short}$ sufficiently small,
it contains the function $\overline u_0$ constructed above.

\medskip\noindent\cb{2} \textbf{$G$ is a self-map}

Let $w\in \M_T$.
Then, by Lemma~\ref{1gewichtethoedlerII} we get
\begin{align*}
\|L_{\beta_1,\beta_2}(\nabla w)\|_{\Cc^{\alpha,\alpha/4}_{0}(Q_T)}
\le
\Cr{gewhoelderI}\big(1+\widehat C^4\big).
\end{align*}
Moreover, uniform ellipticity follows from
\(\|w\|_{\Cc^{4+\alpha,1+\alpha/4}_{1}(Q_T)} \le \widehat C\):
\begin{align*}
\frac{|\xi|^4}{(1+\widehat C^2)^2}
\le
\sum_{k+\ell=4}
L_{k\ell}(\nabla w)\,\xi_1^k\xi_2^\ell
\le
4|\xi|^4.
\end{align*}
Using \eqref{1gewichtetschauderv1} with $T\le1$
and Lemma~\ref{1gewichtethoedlerII},
we obtain
\begin{align*}
      \|G_Tw\|_{C^{4+\alpha,1+\alpha/4}_{1}(Q_T)} 
      \le&\ \Cr{ws}(1) \left( \begin{aligned}
	& \Cr{1gewhoelderI} (1+    \|w\|_{C^{4+\alpha,1+\alpha/4}_{1}(Q_T)})^{k_H}  \|w\|_{C^{4+\alpha,1+\alpha/4}_{1}(Q_T)}^3+ \|u_0\|_{\Cc^{0,1}(\overline\Omega)}
    \\
    & +\|\varphi\|_{\Cc^{4+\alpha,1+\alpha/4}_{1}(\partial\Omega\times(0,T])} 
 +
\|h\|_{\Cc^{3+\alpha,\frac{3+\alpha}{4}}_{0}(\partial\Omega\times(0,T])}
\end{aligned}\right)\\
     \le&\ \Cl{1Gws} \left( \begin{aligned}
	& \Cr{1gewhoelderI} (1+    \widehat C)^{k_H} \widehat C^3+ \|u_0\|_{\Cc^{0,1}(\overline\Omega)}
    \\
    & +\|\varphi\|_{\Cc^{4+\alpha,1+\alpha/4}_{1}(\partial\Omega\times(0,T])} 
 +
\|h\|_{\Cc^{3+\alpha,\frac{3+\alpha}{4}}_{0}(\partial\Omega\times(0,T])}
\end{aligned}\right).
 \end{align*}
where the constant $\Cr{1Gws}$ depends only on $\Omega$, $\alpha$, and an upper bound for $\widehat C$.
From now on, we choose 
$\widehat C$
 such that the nonlinear part contributes at most 
$\tfrac{1}{2}\widehat C$. So we assume that $\widehat C$ satisfies
\begin{align}
\Cr{1Gws}\Cr{1gewhoelderI}
(1+\widehat C)^{k_H}
\widehat C^3
\le
\frac{\widehat C}{2}.
\label{eq:1C1}
\tag{C1}
\end{align}
Furthermore, we impose the smallness condition on the parabolic boundary values
\begin{align}
\|u_0\|_{\Cc^{0,1}(\overline\Omega)}
+\|\varphi\|_{\Cc^{4+\alpha,1+\alpha/4}_{1}(\partial\Omega\times(0,T])} 
 +
\|h\|_{\Cc^{3+\alpha,\frac{3+\alpha}{4}}_{0}(\partial\Omega\times(0,T])}
\le
\frac{\widehat C}{2\Cr{1Gws}}.
\label{eq:1C2}
\tag{C2}
\end{align}
Later, after imposing an additional condition on $\widehat C$
(cf.~\eqref{eq:1C3}), we will define
\(
\Cr{c1short} := \frac{\widehat C}{2\,\Cr{1Gws}},
\)
which is consistent with the smallness assumption in the theorem. Under conditions \eqref{eq:1C1} and \eqref{eq:1C2},
we obtain
\[
\|G_T w\|_{\Cc^{4+\alpha,1+\alpha/4}_{1}(Q_T)}
\le
\widehat C.
\]
Hence, $G_T$ maps $\M_T$ into itself.

\medskip\noindent\cb{3} \textbf{$G$ is a contraction}

Let $u,w\in\M_1$.
Then, arguing as in the contraction step of
Theorem~\ref{gewichtetexistkurz},
the difference $G_Tu-G_Tw$ solves the linear problem
\eqref{eq:gewichtetmodi}.
The operator
\[
\tfrac12\big(L(\nabla u)+L(\nabla w)\big)
\in \Cc^{\alpha,\alpha/4}_{0}(Q_T)
\]
is uniformly elliptic with the same time-independent constants
$\lambda$ and $\Lambda$ as before  for $L(\nabla u)$ and $L(\nabla w)$.
Moreover, since $G_Tu,G_Tw\in\M_T$, we obtain
\begin{align}
\begin{aligned}
 \|D^4(G_Tu+\ G_Tw)\|_{\Cc^{\alpha,\alpha/4}_{-3}(Q_T)}
 \le&\	\|G_Tu\|_{\Cc^{4+\alpha, 1+\alpha/4}_{0}(Q_T)}\ +\|G_Tw\|_{\Cc^{4+\alpha, 1+\alpha/4}_{0}(Q_T)}\\
	\overset{
 }\le &\ 2 \hat C. 
\end{aligned}\label{eq:1g+w}
\end{align}

Applying the weighted Schauder estimate to \eqref{eq:gewichtetmodi}
yields the estimate
\begin{align*}
	\|G_Tu- &\ G_Tw\|_{\Cc^{4+\alpha, 1+\alpha/4}_{1}(Q_T)}\\
	\le&\ \Cr{ws}(1) \cdot\left\|\begin{aligned}
& - \frac{1}{2}\big(L(\nabla u)-L(\nabla w)\big) \big( D^4 (G_Tu+G_Tw)\big)\\
	&\ - \mathcal R\big(\nabla u, D^2 u, D^3 u\big)+\mathcal R\big(\nabla w, D^2 w, D^3 w\big)
\end{aligned}\right\|_{\Cc^{\alpha,\alpha/4}_{-3}(Q_T)}\\
\overset{\mathclap{\eqref{eq:simpleprod}}}{\le}&\ \frac{\Cr{ws}(1)}{2}  \cdot\left\| \big(L(\nabla u)-L(\nabla w)\big)\right\|_{\Cc^{\alpha,\alpha/4}_{0}(Q_T)} \cdot \left\| \big( D^4 (G_Tu+G_Tw)\big)\right\|_{\Cc^{\alpha,\alpha/4}_{-3}(Q_T)}\\
	&\	 + \Cr{ws}(1)\left\|\Rr( \nabla u,D^2 u, D^3 u)-\Rr( \nabla w,D^2 w, D^3 w)\right \|_{\Cc^{\alpha,\alpha/4}_{-3}(Q_T)}
\end{align*}
with a constant $\Cr{ws}(1)$ depending only on $\Omega$, $\alpha$,
and an upper bound for $\widehat C$.

By Lemma~\ref{1gewichtethoedlerIII} and the fact that
$u,w\in \M_T$, so that
\(u,w \in M_T\)
, so that
\[\|u\|_{\Cc^{4+\alpha, 1+\alpha/4}_{1} (Q_T)},\|w\|_{\Cc^{4+\alpha, 1+\alpha/4}_{1} (Q_T)}\le \widehat C\]
we obtain the estimates
\begin{align*}
\big\|\Rr( \nabla u,D^2 u, D^3 u)& -\Rr( \nabla w,D^2 w, D^3 w)\big \|_{\Cc^{\alpha,\alpha/4}_{-3}(Q_T)} \\
&
\le 
\Cr{1hoelderII} (1+ \widehat C)^{k'_H} 
\widehat C^2
 \cdot
 \|u-w\|_{\Cc^{4+\alpha, 1+\alpha/4}_{1} (Q_T)}, 
 \\
\sum_{k+\ell=4}\big\|L_{k\ell}(\nabla u)&-L_{k\ell}(\nabla w)\big\|_{\Cc^{\alpha,\alpha/4}_{0}( Q_T)}\\
&\ 
\le 
\Cr{1hoelderII} (1+ \widehat C)^{k'_H} 
 \cdot
 \|u-w\|_{\Cc^{4+\alpha, 1+\alpha/4}_{1} (Q_T)}.
\end{align*}
Following, Using  \eqref{eq:1g+w}, we obtain
\begin{align*}
\|G_Tu-G_Tw\|_{\Cc^{4+\alpha,1+\alpha/4}_{1}(Q_T)}
\le &
\Cr{ws}(1)\Cr{1hoelderII}
(1+\widehat C)^{k'_H}
\widehat C
\cdot
\|u-w\|_{\Cc^{4+\alpha,1+\alpha/4}_{1}(Q_T)}\\
&
+
\Cr{ws}(1)\Cr{1hoelderII}
(1+\widehat C)^{k'_H}
\widehat C^2
\cdot
\|u-w\|_{\Cc^{4+\alpha,1+\alpha/4}_{1}(Q_T)}\\
\le&
\Cr{ws}(1)\Cr{1hoelderII}
(1+\widehat C)^{k'_H}
(\widehat C+\widehat C^2)
\cdot
\|u-w\|_{\Cc^{4+\alpha,1+\alpha/4}_{1}(Q_T)}.
\end{align*}
Therefore, we impose the additional condition
\begin{align}
\Cr{ws}(1)\Cr{1hoelderII}
(1+\widehat C)^{k'_H}
(\widehat C+\widehat C^2)
\le \tfrac12.
\label{eq:1C3}\tag{C3}
\end{align}
Under conditions \eqref{eq:1C1}–\eqref{eq:1C3}, we obtain
\[
\|G_Tu-G_Tw\|_{\Cc^{4+\alpha,1+\alpha/4}_{1}(Q_T)}
\le
\tfrac12
\|u-w\|_{\Cc^{4+\alpha,1+\alpha/4}_{1}(Q_T)}.
\]
Hence, $G_T$ is a contraction on $\M_T$.

\medskip\noindent\cb{4} \textbf{Application of the Fixed Point Theorem}

Assume that \eqref{eq:1C1}–\eqref{eq:1C3} and
\[
\|u_0\|_{\Cc^{0,1}(\overline\Omega)}
+
\|\varphi\|_{\Cc^{4+\alpha,1+\alpha/4}_{1}(\partial\Omega\times(0,T])}
+
\|h\|_{\Cc^{3+\alpha,\frac{3+\alpha}{4}}_{0}(\partial\Omega\times(0,T])}
\le
\frac{\widehat C}{2\,\Cr{1Gws}}
=:\Cr{c1short},
\]
cf.~\eqref{eq:1C2}.
Under this smallness condition,
the mapping
\[
G_T:\M_T\to\M_T
\]
is a contraction. 
Hence, by the Banach fixed point theorem,
there exists a unique fixed point
\[
v^\ast\in\M_T
\subset
\Cc^{4+\alpha,1+\alpha/4}_{1}(Q_T)
\quad
\text{
such that}
\quad
v^\ast = G_T v^\ast.
\]

By construction of the iteration map,
the identity $v^\ast = G_T v^\ast$
implies that $v^\ast$ solves the original
Willmore flow problem with time-dependent Dirichlet boundary conditions
in the space $\M_T$. We emphasize that uniqueness in the space
$\Cc^{4+\alpha,1+\alpha/4}_{1}(Q_T)$
is not established here, since time-interpolation techniques
are not available in this setting.
\end{proof}

The below corollary shows that the Willmore flow can be initiated
solely under a smallness condition on the Lipschitz norm of the
initial graph $u_0$.
No smallness is required for the boundary data $g_0$ and $g_1$ up to the compatibility conditions. In general, large boundary norms merely reduce the maximal existence time $T$.
\begin{corollary}[Short-Time Existence for time-independent Dirichlet boundary data]
\label{shorttimeexistenceCfixed}
Let $0<\alpha<1$.
Assume
\[
u_0 \in \Cc^{0,1}(\overline\Omega), 
\qquad
g_0 \in \Cc^{4+\alpha}(\partial\Omega),
\qquad
g_1 \in \Cc^{3+\alpha}(\partial\Omega),
\]
and that the compatibility conditions
\[
g_0(x)=u_0(x),
\qquad
g_1(x)=\partial_\nu u_0(x),
\qquad x\in\partial\Omega,
\]
are satisfied.
If, in addition,
\[
\|u_0\|_{\Cc^{0,1}(\overline\Omega)}<\tfrac 13\Cr{c1short},
\]
with 
\(
\Cr{c1short}=\Cr{c1short}(\Omega,\alpha)>0
\) from Theorem \ref{shorttimeexistenceC}
 then there exists a time
\[
T=T\!\left(\Omega,\alpha,
\|g_0\|_{\Cc^{4+\alpha}(\partial\Omega)},
\|g_1\|_{\Cc^{3+\alpha}(\partial\Omega)}\right)\in(0,1]
\]
and a solution
\[
u \in \Cc^{4+\alpha,1+\alpha/4}_{1}(Q_T)
\]
of the Willmore flow problem \eqref{Willmore} with time-independent boundary data.
\end{corollary}
\begin{proof}
We argue as in Step~\cbm{1} (Boundary Values Discussion)
in the proof of Theorem~\ref{gewichtetexistkurz}.
Since the boundary data are time–independent, we extend them trivially to
$\partial\Omega\times(0,1]$ by setting
\[
\overline g_0(x,t):=g_0(x), 
\qquad 
\overline g_1(x,t):=g_1(x),
\quad x\in\partial\Omega .
\]

Hence all time derivatives of $\overline g_0$ and $\overline g_1$ vanish.
In particular, the temporal Hölder seminorms are zero and we obtain
\begin{align*}
\|\overline g_0\|_{\Cc^{4+\alpha,1+\alpha/4}_{1}(\partial\Omega\times(0,T])}
&=
T^{\frac{3+\alpha}{4}}
\sum_{|\beta|=4}[D^\beta_x g_0]_{\Cc^\alpha(\partial\Omega)}
+
\sum_{1\le |\beta|\le4}
T^{\frac{|\beta|-1}{4}}
\|D^\beta_x g_0\|_{C^0(\partial\Omega)}
+
\|g_0\|_{C^0(\partial\Omega)},\\
\|\overline g_1\|_{\Cc^{3+\alpha,(3+\alpha)/4}_{0}(\partial\Omega\times(0,T])}
&=
T^{\frac{3+\alpha}{4}}
\sum_{|\beta|=3}[D^\beta_x g_1]_{\Cc^\alpha(\partial\Omega)}
+
\sum_{0\le |\beta|\le3}
T^{\frac{|\beta|}{4}}
\|D^\beta_x g_1\|_{C^0(\partial\Omega)} .
\end{align*}

Using the compatibility conditions and the assumption $T<1$, we deduce
\begin{align*}
\|\overline g_0\|_{\Cc^{4+\alpha,1+\alpha/4}_{1}(\partial\Omega\times(0,T])}
&\le
T^{1/4}
\left(
\sum_{|\beta|=4}[D^\beta_x g_0]_{\Cc^\alpha(\partial\Omega)}
+
\sum_{2\le |\beta|\le4}\|D^\beta_x g_0\|_{C^0(\partial\Omega)}
\right)
+
\|u_0\|_{\Cc^{0,1}(\overline\Omega)},\\
\|\overline g_1\|_{\Cc^{3+\alpha,(3+\alpha)/4}_{0}(\partial\Omega\times(0,T])}
&\le
T^{1/4}
\left(
\sum_{|\beta|=3}[D^\beta_x g_1]_{\Cc^\alpha(\partial\Omega)}
+
\sum_{1\le |\beta|\le3}\|D^\beta_x g_1\|_{C^0(\partial\Omega)}
\right)
+
\|u_0\|_{\Cc^{0,1}(\overline\Omega)} .
\end{align*}
Consequently, choosing $\varphi=g_0$, $h=g_1$, and $T>0$ sufficiently small,
we obtain
\[
\|u_0\|_{\Cc^{0,1}(\overline\Omega)}
+
\|\varphi\|_{\Cc^{4+\alpha,1+\alpha/4}_{1}(\partial\Omega\times(0,T])}
+
\|h\|_{\Cc^{3+\alpha,(3+\alpha)/4}_{0}(\partial\Omega\times(0,T])}
<
\Cr{c1short}.
\]
Hence the assumptions of Theorem~\ref{shorttimeexistenceC} are satisfied,
which yields the desired short–time solution.
\end{proof}

\section{Global Existence} \label{sec:globexist}

In this chapter, we want to investigate global existence of the graphical Willmore flow solutions. Especially, we will need a lower bound on short existence time in order to be able to extend the local solution by this fixed time. In this way, we will prevent blow-ups. In what follows we always assume the compatibility condition \eqref{eq:wkb} for $u_0,g_0$ and $g_1$.

\begin{lemma}\label{Lemma:schranken}
Let $\alpha\in(0,1)$.
Then there exist constants
\[
\Cl{delta}=\Cr{delta}(\alpha,\Omega)>0,
\qquad
\Cl{11}=\Cr{11}(\alpha,\Omega)>0,
\]
such that the following holds.
If
\begin{align*}
\|u_0\|_{\Cc^{0,1}(\overline\Omega)}
+
\|g_0\|_{\Cc^{4+\alpha}(\partial\Omega)}
+
\|g_1\|_{\Cc^{3+\alpha}(\partial\Omega)}
\le \Cr{delta},
\end{align*}
then there exists a solution
\[
u\in \Cc^{4+\alpha,1+\alpha/4}_{1}(Q_{1})
\]
of the Willmore flow problem \eqref{Willmore} with $T=1$ satisfying
\begin{align}
\|u\|_{\Cc^{4+\alpha,1+\alpha/4}_{1}(Q_{1})}
\le
\Cr{11}
\Big(
\|u_0\|_{\Cc^{0,1}(\overline\Omega)}
+
\|g_0\|_{\Cc^{4+\alpha}(\partial\Omega)}
+
\|g_1\|_{\Cc^{3+\alpha}(\partial\Omega)}
\Big).
\label{eq:klein}
\end{align}
\end{lemma}
\begin{proof}
\emph{Idea of the proof.}
Starting from the short-time existence result for time-independent
boundary data, we derive a weighted Schauder estimate for the solution.
The key observation is that the nonlinear terms grow only polynomially
in the weighted Hölder norm. Using this structure of the nonlinear terms and the bound on the Schauder
constants, we show that the nonlinear contribution can be absorbed
for sufficiently small data.
This yields the desired a priori estimate.

Assume first that $\Cr{delta}\le \Cr{c1short}(\alpha,\Omega)$.
Arguing as in Step~\cbm{1} (Boundary Values Discussion)
in the proof of Theorem~\ref{gewichtetexistkurz},
we extend the time-independent Dirichlet boundary data trivially by
\[
\overline g_0(x,t):=g_0(x),\qquad
\overline g_1(x,t):=g_1(x).
\]
Since the boundary data do not depend on time, the weighted norms satisfy
\[
\|\overline g_0\|_{\Cc^{4+\alpha,1+\alpha/4}_{1}(\partial\Omega\times(0,1])}
=
\|g_0\|_{\Cc^{4+\alpha}(\partial\Omega)},
\qquad
\|\overline g_1\|_{\Cc^{3+\alpha,(3+\alpha)/4}_{0}(\partial\Omega\times(0,1])}
=
\|g_1\|_{\Cc^{3+\alpha}(\partial\Omega)} .
\]
Hence Theorem~\ref{shorttimeexistenceCfixed} yields a solution
\[
u\in \Cc^{4+\alpha,1+\alpha/4}_{1}(Q_{1})
\]
of \eqref{Willmore} for \(T=1\)
\begin{align}
\left\{
\begin{aligned}
\partial_t u
&=
- L(\nabla u)\,D^4u
- \mathcal R(\nabla u,D^2u,D^3u),
&& \text{in } \Omega\times(0,1],\\
u(\cdot,0) &= u_0,
&& \text{in } \Omega,\\
u &= g_0,
&& \text{on } \partial\Omega\times[0,1],\\
\partial_\nu u &= g_1,
&& \text{on } \partial\Omega\times[0,1].
\end{aligned}
\right.
\tag{G} \label{eq:newg}
\end{align}
Moreover, we have the estimate
\begin{align*}
\|u\|_{\Cc^{4+\alpha,1+\alpha/4}_{1}(Q_{1})}
&\le
\Cr{overlineC}\bigl(
\Omega,\,
\|g_0\|_{\Cc^{4+\alpha}(\partial\Omega)},\,
\|g_1\|_{\Cc^{3+\alpha}(\partial\Omega)},\,
\|u_0\|_{\Cc^{0,1}(\overline\Omega)}
\bigr)
\\
&\le \Cl{12},
\end{align*}
where $\Cr{12}=\Cr{12}(\alpha,\Omega)$, since $\Cr{delta}\le \Cr{c1short}$ and the data are
assumed to be $\Cr{delta}$-small.
Since $u \in \Cc^{4+\alpha,1+\alpha/4}_{1}(Q_{1})$,
Lemma~\ref{1gewichtethoedlerII} implies
\[
\mathcal R(\nabla u,D^2u,D^3u)\in \Cc^{\alpha,\alpha/4}_{-3}(Q_1).
\]
Together with Lemma~\ref{ellipL}, this allows us to apply
Theorem~\ref{gewichteteschauder} and obtain
\begin{align*}
\|u\|_{\Cc^{4+\alpha,1+\alpha/4}_{1}(Q_1)}
\le
\Cr{ws}(1)
\left(
	\begin{aligned}
	    &\ \|\mathcal R(\nabla u,D^2u,D^3u)\|_{\Cc^{\alpha,\alpha/4}_{-3}(Q_1)}
+\|g_0\|_{\Cc^{4+\alpha}(\partial \Omega)} \\
&\ + \|g_1\|_{\Cc^{3+\alpha}(\partial \Omega)} + \|u_0\|_{\Cc^{0,1}(\overline\Omega)}
	\end{aligned}
\right).
\end{align*}

Here $\Cr{ws}(1)$ only depends on
the operator in \eqref{eq:newg} and on $\Omega$,
and on an upper bound for $\|u\|_{\Cc^{4+\alpha,1+\alpha/4}_{1}(Q_{1})}$.
Indeed, by Lemma~\ref{1gewichtethoedlerII} we can control the ellipticity
constant and the coefficients in terms of this norm:
\begin{align*}
\frac{1}{\bigl(1+\|u\|_{\Cc^{4+\alpha,1+\alpha/4}_{1}(Q_{1})}^2\bigr)^2}
\le 
\lambda,
\qquad
\|L(\nabla u)\|_{\Cc^{\alpha,\alpha/4}_{0}(Q_1)}
\le
\Cr{gewhoelderI}\bigl(1+\|u\|_{\Cc^{4+\alpha,1+\alpha/4}_{1}(Q_{1})}^4\bigr).
\end{align*}
In particular, arguing as in the proof of Theorem~\ref{shorttimeexistenceC},
we may use Lemma~\ref{1gewichtethoedlerII} to estimate
\[
\|\mathcal R(\nabla u,D^2u,D^3u)\|_{\Cc^{\alpha,\alpha/4}_{-3}(Q_1)}
\le
\Cr{gewhoelderI}
\bigl(1+\|u\|_{\Cc^{4+\alpha,1+\alpha/4}_{1}(Q_1)}\bigr)^{k_H}
\|u\|_{\Cc^{4+\alpha,1+\alpha/4}_{1}(Q_1)}^{3}.
\]
Consequently,
\begin{align*}
	\|u \|_{\Cc^{4+\alpha, 1+\alpha/4}_1  (Q_{T_1} 	 )} 
	& \le \Cr{ws}(1) 
	\left(
	\begin{aligned}
	     	\Cr{gewhoelderI} &
\Big(1+\|u\|_{\Cc^{4+\alpha, 1+\alpha/4}_1 (Q_{1})} \Big)^{k_H}
 \|u\|_{\Cc^{4+\alpha, 1+\alpha/4}_1 (Q_{1})}^3 \\
&
+\|g_0\|_{\Cc^{4+\alpha}(\partial \Omega)} + \|g_1\|_{\Cc^{3+\alpha}(\partial \Omega)} + \|u_0\|_{\Cc^{0,1}(\overline\Omega)}
	\end{aligned}
\right)\\
	&
	\le \Cr{SchauderT0}
	\left(
	\begin{aligned}
	     	 &
\Big(1+\|u\|_{\Cc^{4+\alpha, 1+\alpha/4}_1 (Q_{1})} \Big)^{k_H}
\|u\|_{\Cc^{4+\alpha, 1+\alpha/4}_1 (Q_{1})}^3  \\
&
+\|g_0\|_{\Cc^{4+\alpha}(\partial \Omega)} + \|g_1\|_{\Cc^{3+\alpha}(\partial \Omega)} + \|u_0\|_{\Cc^{0,1}(\overline\Omega)}
	\end{aligned}
\right) .
\end{align*}
Since 
\(
\|u\|_{\Cc^{4+\alpha,1+\alpha/4}_{1}(Q_{1})}
\le
\Cr{12}(\alpha,\Omega),
\)
the Schauder constants \(
\Cl{SchauderT0}=\Cr{SchauderT0}(\alpha,\Omega)
\) remain uniformly bounded. 

Furthermore, there exists a constant
\(
\Cl{25}=\Cr{25}(\alpha,\Omega)
\)
such that,
\begin{align}
\|u\|_{\Cc^{4+\alpha,1+\alpha/4}_{1}(Q_1)}
\le \Cr{25}
\ \Rightarrow\
\Cr{SchauderT0}
\bigl(1+\|u\|_{\Cc^{4+\alpha,1+\alpha/4}_{1}(Q_1)}\bigr)^{k_H}
\|u\|_{\Cc^{4+\alpha,1+\alpha/4}_{1}(Q_1)}^2
\le
\frac{1}{2},
\label{eq:cstern} \tag{C1}
\end{align}
where the constant $\Cr{SchauderT0}$ is a Schauder constant and depends only on $\alpha$.

Next, we reconsider the proof of the short-time existence
Theorem~\ref{shorttimeexistenceC} and modify the iteration set
$\M_T$ defined in \eqref{eq:MsetC} by
\begin{align}
	\M_1:=\left\{
w	 \in \Cc^{4+\alpha,1+\alpha/4}_1(Q_1) 
\left|  \begin{aligned} &  \|w\|_{\Cc^{4+\alpha, 1+\alpha/4}_{1} (Q_1)} \le  \tfrac 12\Cr{25} ,\\
& w(\cdot,0)=u_0, 
\ w|_{\partial\Omega}=g_0,
\ \partial_\nu w|_{\partial\Omega}=g_1
	\end{aligned} \right. \right\}. \label{eq:M1def}
\end{align}

Arguing as in Theorem~\ref{shorttimeexistenceC}, we obtain a fixed
point $u$ of the mapping $G_T:\M_1\to\M_1$, provided that $\Cr{25}$ and therefore  \[\Cr{delta}\le \min(\Cr{25}/\Cr{overlineuest},\Cr{25})\]
are chosen sufficiently small.
In particular, we require $\Cr{25}$ to be smaller than the bound
appearing in \eqref{eq:cstern}, namely
\begin{align}
    \Cr{25}\le \widehat C,
    \label{eq:const2} \tag{C2}
\end{align}
where $\widehat C=\widehat C(\alpha,\Omega)$ is the constant introduced
in the proof of Theorem~\ref{shorttimeexistenceC}.
Moreover, the fixed point
\[
u\in\Cc^{4+\alpha,1+\alpha/4}_1(Q_{1})
\]
 solves \eqref{eq:newg} with $T=1$.
From the definition of $\M_T$ we obtain
\[
\|u\|_{\Cc^{4+\alpha,1+\alpha/4}_1(Q_{1})}
\le
\tfrac{1}{2}\Cr{25}.
\]
We can therefore apply the above statement \eqref{eq:cstern}. Namely, by using the $\Cc^{4+\alpha, 1+\alpha/4}_1 (Q_{1})$-Schauder estimate for the solution $u$ of  \eqref{eq:newg} in Theorem~\ref{shorttimeexistenceC}   we then get
\begin{align*}
	\|u &\|_{\Cc^{4+\alpha, 1+\alpha/4}_1  (Q_{1} 	 )} \\
	&\ \overset{
	}{\le} \Cr{SchauderT0} 
	\left(
	\begin{aligned}
	     	 &
\Big(1+\|u\|_{\Cc^{4+\alpha, 1+\alpha/4}_1 (Q_{1})} \Big)^{k_H}
\|u\|_{\Cc^{4+\alpha, 1+\alpha/4}_1 (Q_{1})}^3  \\
&
+\|g_0\|_{\Cc^{4+\alpha}(\partial \Omega)} + \|g_1\|_{\Cc^{3+\alpha}(\partial \Omega)} + \|u_0\|_{\Cc^{0,1}(\overline\Omega)}
	\end{aligned}
\right)\\
&\ \overset{\mathclap{\eqref{eq:cstern}}}{\le}  \Cr{SchauderT0}
	\left(
	\begin{aligned}
	     \frac{1}{2 \Cr{SchauderT0}} \|u\|_{\Cc^{4+\alpha, 1+\alpha/4}_1 ( Q_{1})} 
+\|g_0\|_{\Cc^{4+\alpha}(\partial \Omega)} + \|g_1\|_{\Cc^{3+\alpha}(\partial \Omega)} + \|u_0\|_{\Cc^{0,1}(\overline\Omega)}
	\end{aligned}
\right).
\end{align*}
Collecting terms,  we conclude
\begin{align*}
	\|u \|_{\Cc^{4+\alpha, 1+\alpha/4}_1 ( Q_{1} 	 )} \le&\ 
	2\Cr{SchauderT0}
	\left(
\|g_0\|_{\Cc^{4+\alpha}(\partial \Omega)} + \|g_1\|_{\Cc^{3+\alpha}(\partial \Omega)} + \|u_0\|_{\Cc^{0,1}(\overline\Omega)}
\right) \\
	\le & \Cr{11}(\alpha, \Omega)
	\left(
\|g_0\|_{\Cc^{4+\alpha}(\partial \Omega)} + \|g_1\|_{\Cc^{3+\alpha}(\partial \Omega)} + \|u_0\|_{\Cc^{0,1}(\overline\Omega)}
\right)
,
\end{align*}
which proves the lemma.
\end{proof}


\subsection{Main Global Existence Theorem} \label{subsec:globalexist}

\begin{thm}[Global Existence]\label{thm:globalc1ac4a}
Let $\alpha\in(0,1)$ and let $\Omega\subset\R^2$ 
be a bounded domain with $\Cc^{4+\alpha}$-boundary.
Then there exists a constant
\[
\Cl{13}=\Cr{13}(\alpha,\Omega)>0
\]
such that the following holds.
If
\begin{align*}
\|u_0\|_{\Cc^{0,1}(\overline\Omega)}
+
\|g_0\|_{\Cc^{4+\alpha}(\partial\Omega)}
+
\|g_1\|_{\Cc^{3+\alpha}(\partial\Omega)}
<
\Cr{13},
\end{align*}
then the corresponding solution of the Willmore flow problem
\eqref{Willmore} exists globally in time, i.e.,
for every $T>0$ there exists a solution
\[
u \in \Cc^{4+\alpha,1+\alpha/4}_{1}(Q_T) \quad \text{ with } \|u(\,\cdot\,,t)\|_{\Cc^{0,1}(\overline\Omega)}\le \tfrac 13 \Cr{delta}
\text{ for all } t\in[0,\infty).
\]
\end{thm}

\begin{proof} 
\emph{Idea of the proof.}
We argue by a continuation principle.
First, Lemma~\ref{Lemma:schranken} provides a short-time solution on $[0,1]$
together with a small Lipschitz bound.
Assuming for contradiction the existence of a finite maximal time $T_m$,
we introduce the supremum time $T_b\in[1,T_m)$ on which the Lipschitz norm
remains below $\Cr{delta}/2$.
On each time slab of length $1$ contained in $[1,T_b]$ we may restart
Lemma~\ref{Lemma:schranken}, yielding uniform weighted Schauder bounds and,
in particular, a uniform $C^3$-bound on $[1,T_b]$.
An interpolation inequality then shows that the Lipschitz norm becomes small
once $\|u(t)\|_{L^2(\Omega)}$ is sufficiently small.
Finally, smallness of $\|u(t)\|_{L^2(\Omega)}$ on $[1,T_b]$ is obtained
from the monotonicity of the Willmore energy and Theorem~\ref{l2bound}\,\cbm b,
provided $\Cr{13}$ is chosen sufficiently small.
This yields a contradiction to the definition of $T_b$ and implies $T_m=\infty$.

\medskip\noindent\cb{1}\ \textbf{Short-time existence and Lipschitz control up to $1$.}

We first assume that
\[
\Cr{13}\le \frac{\Cr{delta}}{2},
\]
where $\Cr{delta}=\Cr{delta}(\alpha,\Omega)$
is the constant from Lemma~\ref{Lemma:schranken}.
By Lemma~\ref{Lemma:schranken} the Willmore flow problem \eqref{Willmore}
admits a unique solution
\[
u\in \Cc^{4+\alpha,1+\alpha/4}_{1}(Q_{1}).
\]
Observe that the parabolic Hölder norm controls the spatial Lipschitz norm.
Indeed,
\begin{align}
\|u\|_{\Cc^{4+\alpha,1+\alpha/4}_{1}(Q_{1})}
\ge
\sum_{|\beta|\le 1}
\sup_{(x,t)\in \overline\Omega\times(0,1]}
|D_x^\beta u(x,t)|
=
\sup_{t\in[0,1]}
\|u(t)\|_{\Cc^{0,1}(\overline\Omega)}.
\label{eq:c2c4bound}
\end{align}

Next, we further assume
\[
\Cr{13}\le \frac{\Cr{delta}}{2\,\Cr{11}},
\]
where $\Cr{11}=\Cr{11}(\alpha,\Omega)$
is the constant from Lemma~\ref{Lemma:schranken}.
Note that $\Cr{11}$ depends only on $\alpha$ and $\Omega$.
Then Lemma~\ref{Lemma:schranken} implies
\begin{align}
    \begin{aligned}
\sup_{t\in[0,1]}	\|u(\,\cdot\,, t)\|_{C^{0,1}(\overline\Omega)}\overset{\eqref{eq:c2c4bound}}{\le}&\ \|u\|_{\Cc^{4+\alpha,1+\frac{\alpha}{4}}_{1}( Q_{1})} \\
	\le &\  \Cr{11}\left(   \|u_0\|_{\Cc^{0,1}(\overline\Omega)}
   + \|g_0\|_{\Cc^{4+\alpha}(\partial \Omega)}
   + \|g_1\|_{\Cc^{3+\alpha}(\partial \Omega)}\right)\\
   \le &\ \frac{\Cr{delta}}2. 
   \end{aligned}\label{eq:supt0}
\end{align}
Since the $\Cc^{0,1}(\overline\Omega)$-norm of $u(\,\cdot\,,t)$
remains below $\Cr{delta}$ on $[0,1]$,
the smallness assumption of Lemma~\ref{Lemma:schranken}
is still valid at time $1$.
Therefore, we may restart the argument at $t=1$
and extend the solution to a larger time interval.

\medskip\noindent\cb{2}\ \textbf{Contradiction setup: maximal existence time $T_m$ and the supremum time $T_b$.}

Suppose, contrary to the claim, that there exists a finite maximal
existence time $T_m>0$.
More precisely, assume that for every $T\in (1,T_m)$
the following boundary value problem admits a unique solution
\[
v\in \Cc^{4+\alpha,1+\alpha/4}_{x,t}
(\overline\Omega\times[1,T]),
\]
where
\begin{align}
\left\{
\begin{aligned}
\partial_t v
&=
- L(\nabla v) D^4 v
- \mathcal R(\nabla v,D^2v,D^3v),
&& \text{in } \Omega\times[1,T],\\
v(x,1) &= u(x,1),
&& x\in\overline\Omega,\\
v(x,t) &= g_0(x),
&& (x,t)\in\partial\Omega\times[1,T],\\
\partial_\nu v(x,t) &= g_1(x),
&& (x,t)\in\partial\Omega\times[1,T],
\end{aligned}
\right.
\end{align}
but that the solution cannot be extended beyond $T_m$.
By $\Cc^{4+\alpha,1+\alpha/4}_{x,t}
(\overline\Omega\times[1,T])$-uniqueness, we may define
\[
u(x,t):=v(x,t),
\qquad (x,t)\in \overline\Omega\times[1,T_m).
\]

Next, we define the supremum time
\[
T_b:=\sup\Big\{
\tau\in[1,T_m)\;:\;
\|u(\,\cdot\,,t)\|_{\Cc^{0,1}(\overline\Omega)}\le \tfrac 13 \Cr{delta}
\text{ for all } t\in[0,\tau]
\Big\}.
\]
By \eqref{eq:supt0}, we have $T_b\ge 1$.
Combining the assumption
\(
\Cr{13}\le \tfrac 12\Cr{delta}
\)
with the definition of $T_b$, we obtain  
\begin{align}
\forall t \in [0, T_b ]\colon\quad  \|u(\,\cdot\,,t)\|_{C^{0,1}(\overline \Omega)}  + \|g_0\|_{\Cc^{4+\alpha}(\partial \Omega)}
   + \|g_1\|_{\Cc^{3+\alpha}(\partial \Omega)}\le&\   \Cr{delta} /3 + \Cr{delta}/2 < \Cr{delta} . \label{eq:defannahmec1alphacor}
\end{align}

\medskip\noindent\cb{3}\ \textbf{Restart on slabs and uniform higher-order bounds on $[1,T_b]$.}

Now fix $t\in[1,T_b]$.
Since \eqref{eq:defannahmec1alphacor} holds at time $t-1$,
the hypotheses of Lemma~\ref{Lemma:schranken}
are satisfied with initial time $t-1$. 
Since we want to apply the weighted estimate \eqref{eq:klein}
from Lemma~\ref{Lemma:schranken} on the time interval $(t-1,t]$,
we must ensure that the
$u\in\Cc^{4+\alpha,1+\alpha/4}_{x,t}(\overline\Omega\times[t-1,t])$
solution of the Willmore flow also belongs to the set $\M_1$
defined in \eqref{eq:M1def}.
To this end, choose $\varepsilon>0$ sufficiently small.
Then
\begin{align*}
\|u\|_{\Cc^{4+\alpha,1+\alpha/4}_{1}
(\overline\Omega\times(t-1,t-1+\varepsilon])}
&\le
\sup_{t'\in[0,T_b-1]}
\|u(\,\cdot\,,t')\|_{\Cc^{0,1}(\overline\Omega)}
+
\varepsilon^{1/4}
\|u\|_{\Cc^{4+\alpha,1+\alpha/4}_{x,t}
(\overline\Omega\times[t-1,t-1+\varepsilon])} \\
&\le
\frac{\Cr{delta}}{2}
\le
\frac{\Cr{25}}{2}.
\end{align*}

Hence the restriction of $u$ to
$\overline\Omega\times(t-1,t-1+\varepsilon]$
lies in $\M_\varepsilon$. 
By uniqueness in $\M_\varepsilon$ from Theorem~\ref{shorttimeexistenceC}, where we replace $1$ by $\varepsilon$,
the solution in
$\Cc^{4+\alpha,1+\alpha/4}_{x,t}
(\overline\Omega\times[t-1,t-1+\varepsilon])$
coincides with the small
$\Cc^{4+\alpha,1+\alpha/4}_{1}$–solution.
Finally, uniqueness on the remaining interval
$[t-1+\varepsilon,t]$
follows from the unweighted existence theorem
Theorem~\ref{gewichtetexistkurz}.

Next, we obtain
\begin{align*}
    	\|u\|_{\Cc^{4+\alpha, 1+\alpha/4}_{1}( \overline\Omega\times (t-1, t]) } &\le  \Cr{11} \left(   \|u(\,\cdot\,, t-1 )\|_{\Cc^{0,1}(\overline\Omega)}
   + \|g_0\|_{\Cc^{4+\alpha}(\partial \Omega)}
   + \|g_1\|_{\Cc^{3+\alpha}(\partial \Omega)}\right) \\
&\overset{\eqref{eq:defannahmec1alphacor}}\le 
\Cr{11} \Cr{delta}.
\end{align*}

By the definition of the weighted Hölder norm,
for every $t\in[1,T_b]$ we have
\begin{align*}
\|u\|_{\Cc^{4+\alpha,1+\alpha/4}_{1}
(\overline\Omega\times(t-1,t])}
&\ge
\sum_{|\beta|\le 1}
\sup_{(x,t')\in\overline\Omega\times(t-1,t]}
\big|D_x^\beta u(x,t')\big|
\\
&\quad
+
\sup_{(x,t')\in\overline\Omega\times(t-1,t]}
(t'-t+1)^{1/2}
\big|D_x^3 u(x,t')\big|
\\
&\quad
+
\sup_{(x,t')\in\overline\Omega\times(t-1,t]}
(t'-t+1)^{1/4}
\big|D_x^2 u(x,t')\big|
\\
&\ge 
    \sup_{x\in\overline\Omega}1^{\frac{2}{4} }\big| D_x^3  u(x, t)\big|
    +	
    \sup_{x\in\overline\Omega}1^{\frac{1}{4} }\big| D_x^2  u(x,t)\big| 
	+ 
    \sum_{|\beta|\le 1}\sup_{x\in \overline \Omega}\big|  D^\beta_x u(x,t)\big|.
\end{align*}
Then this estimate implies that for all $t\in[1,T_b]$ we have \textit{a uniform $C^3$-bound }
\begin{align*}
\|u(\,\cdot\,,t)\|_{\Cc^{3}(\overline\Omega)}
&\le
\Cr{11}\Cr{delta}
=:
\Cl{14},
\end{align*}
where $\Cr{14}=\Cr{14}(\alpha,\Omega)>0$
depends only on $\alpha$ and $\Omega$.

\medskip\noindent\cb{4}\ \textbf{Interpolation: Lipschitz control from an $L^2$-bound.}

By the  Theorem~\ref{thm:interpolation}, there exists $\beta\in(0,1)$
such that for all $t\in[1,T_b]$,
\begin{align*}
 \|u( \,\cdot\,,t)\|_{C^{0,1}(\overline\Omega)}
 \le &\ \Cl{15}(\alpha,\Omega) \|u( \,\cdot\,,t)\|_{C^{2}(\overline\Omega)}
 \\
 \le&\  \Cl{16}(\alpha,\Omega) \|u(\,\cdot\,,t)\|_{L^2(\Omega)}^\beta\cdot \|u(\,\cdot\,,t)\|_{C^{2+\alpha}(\overline\Omega)}^{1-\beta}  \\
 \le &\ \Cr{16}\|u(\,\cdot\,,t)\|_{L^2(\Omega)}^\beta\big( \Cr{15}(\alpha,\Omega) \|u(\,\cdot\,,t)\|_{C^{3}(\overline\Omega)}\big)^{1-\beta} \\
  \le &\ \Cr{16}\|u(\,\cdot\,,t)\|_{L^2(\Omega)}^\beta\big( \Cr{15}\Cr{14}\big)^{1-\beta}.
\end{align*}
Here the constants $\Cr{14},\Cr{15},\Cr{16}$
depend only on $\alpha$ and $\Omega$,
and in particular are independent of $\Cr{13}$. 

Suppose now that $\|u(\,\cdot\,,t)\|_{L^2(\Omega)}$
remains sufficiently small for all $t\in[1,T_b]$,
so that
\[
\|u(\,\cdot\,,t)\|_{\Cc^{0,1}(\overline\Omega)}
<
\tfrac 14{\Cr{delta}}
\quad\text{for all } t\in[1,T_b].
\]
Then this contradicts the definition of $T_b$
Therefore, we must have $T_b=T_m$.
In particular, the $\Cc^{0,1}$-norm of $u$
remains bounded at time $T_m$.
Consequently, the local existence result applies once more,
and the solution can be extended beyond $T_m$,
contradicting the maximality of $T_m$.
Hence $T_m=\infty$, and the solution exists globally in time.

\medskip\noindent\cb{5}\ \textbf{Smallness of $\|u(\,\cdot\,,t)\|_{L^2}$ from energy monotonicity.}

So our goal is to choose $\Cr{13}$ sufficiently small such that
$\|u(\,\cdot\,,t)\|_{L^2(\Omega)}$ becomes small.
First, evaluating the definition of the weighted norm at time $t=1$
(cf.~\eqref{eq:c2c4bound}), we obtain
\begin{align*}
\|u(\,\cdot\,,1)\|_{\Cc^{2}(\overline\Omega)}
&\le
\sum_{|\beta|\le 1}\sup_{x\in\overline\Omega}|D_x^\beta u(x,1)|
+
\sup_{x\in\overline\Omega}|D_x^2u(x,1)|
\\
&\le
\|u\|_{\Cc^{4+\alpha,1+\alpha/4}_{1}(Q_{1})}
+
1^{-1/4}\,
\|u\|_{\Cc^{4+\alpha,1+\alpha/4}_{1}(Q_{1})}
\\
&=
2
\|u\|_{\Cc^{4+\alpha,1+\alpha/4}_{1}(Q_{1})}.
\end{align*}
In view of Lemma~\ref{Lemma:schranken} we further have
\[
\|u\|_{\Cc^{4+\alpha,1+\alpha/4}_{1}(Q_{1})}
\le
\Cr{11}\Big(
\|u_0\|_{\Cc^{0,1}(\overline\Omega)}
+
\|g_0\|_{\Cc^{4+\alpha}(\partial\Omega)}
+
\|g_1\|_{\Cc^{3+\alpha}(\partial\Omega)}
\Big),
\]
and hence
\begin{align}
\|u(\,\cdot\,,1)\|_{\Cc^{2}(\overline\Omega)}
\le
2\Cr{11}\Big(
\|u_0\|_{\Cc^{0,1}(\overline\Omega)}
+
\|g_0\|_{\Cc^{4+\alpha}(\partial\Omega)}
+
\|g_1\|_{\Cc^{3+\alpha}(\partial\Omega)}
\Big)
\le
2\Cr{11}\Cr{13},
\label{eq:C2atT1}
\end{align}
Moreover, the Willmore energy is finite at time $1$ and non-increasing
along the flow. In particular, for all $t\ge 1$,
\[
\W\big(u(\,\cdot\,,t)\big)\le \W\big(u(\,\cdot\,,1)\big)\le \C(\Omega)\|u(\,\cdot\,,1)\|_{C^2(\overline\Omega)}^3.
\]

\medskip\noindent\cb{6}\ \textbf{Conclusion: $T_m=\infty$.}

Combining \eqref{eq:C2atT1} with the monotonicity of the Willmore energy,
we can make $\W(u(1))$ arbitrarily small by choosing $\Cr{13}$ sufficiently small.
Therefore, Theorem~\ref{l2bound}\,\cbm b implies that
\[
\sup_{t\in[1,T_b]}\|u(\,\cdot\,,t)\|_{L^2(\Omega)}
\]
is as small as desired (depending only on $\alpha$ and $\Omega$).
In particular, we may ensure that the right-hand side in the interpolation
estimate is below $\Cr{delta}/4$.

\end{proof}


\subsection{Subconvergence/Convergence to a Critical Point}
\label{subsec:subconv}
The next theorem shows that every globally defined solution whose
$C^{0,1}$-norm remains uniformly small subconverges to a critical point of
the Willmore energy.
\begin{thm}[Subconvergence]\label{thm:subc}
Let $\alpha\in(0,1)$ and let 
\[
g_0\in C^{4+\alpha}(\partial\Omega),\qquad
g_1\in C^{3+\alpha}(\partial\Omega),\qquad
u_0\in C^{0,1}(\overline{\Omega}).
\]
Let $(u(\,\cdot\,,t))_{t\ge0}$ be a global solution of the Willmore flow
\eqref{Willmore} with initial value $u(0)=u_0$ and $u(\,\cdot\,,t)\in C^{4+\alpha}(\overline{\Omega})$ for all $t>0$,  constructed in
Theorem~\ref{thm:globalc1ac4a}.

Then there exist a sequence $(t_k)_{k\in\N}\subset\R_+$ with
\[
t_k\to\infty
\quad\text{as }k\to\infty
\]
and a critical point $u_\infty$ of the Willmore energy such that for
every $\beta\in(0,\alpha)$
\[
u(\,\cdot\,, t_k)\to u_\infty
\qquad\text{in } C^{4+\beta}(\overline\Omega).
\]
\end{thm}
\begin{proof} 
First, $u$ is the global solution constructed in
Theorem~\ref{thm:globalc1ac4a}, which yields
\[
\|u(t)\|_{C^{0,1}(\overline\Omega)}\le \frac{\Cr{delta}}{2}
\qquad\text{for all }t>0 .
\]
Hence the assumptions of Lemma~\ref{Lemma:schranken} are satisfied
with initial time $t$.
Applying it on the interval $(t,t+1]$ for some
 yields
\[
\|u\|_{C^{4+\alpha,1+\alpha/4}_{1}
(\overline\Omega\times(t,t+1])}
\le \Cl{C}
\]
where $\Cr{C}$ depends on $\alpha, \Omega$ and $\|g_0\|_{\Cc^{4+\alpha}(\partial\Omega)}
$ with $
\|g_1\|_{\Cc^{3+\alpha}(\partial\Omega)}$.
Consequently, we obtain
\[
u(\,\cdot\,,,t)\in C^{4+\alpha}(\overline\Omega)
\qquad\text{for all }t>0,
\]
and moreover it follows
\[
\|u(\,\cdot\,,,t)\|_{C^{4+\alpha}(\overline\Omega)}\le \Cr C
\qquad\text{for all }t\ge1 .
\]

Furthermore, since the Dirichlet boundary data are independent of time,
the Willmore energy is non-increasing along the flow. Indeed,
for every $t_0>0$ we have
\begin{align*}
\frac{d}{dt}\W(u)\big|_{t=t_0}
&=
-\frac12
\int_\Omega
\left|
\Delta_{\Gamma(u)}H
+2H\!\left(\frac12H^2-\Kk\right)
\right|^2
Q\,dx \\
&=
-\frac12
\int_\Omega |\partial_t u|^2 Q\,dx
\le 0 .
\end{align*}
Hence the Willmore energy decreases monotonically and for all
$t_1,t_2>0$ we obtain
\[
\W(u(\,\cdot\,,t_2))-\W(u(\,\cdot\,,t_1))
=
-\frac12
\int_{t_1}^{t_2}\!\!\int_\Omega
|\partial_t u|^2 Q\,dx\,dt .
\]

Letting $t_1\to0$ and $t_2\to\infty$, and using $\W\ge0$ and $Q\ge1$,
we obtain
\[
\|\partial_t u\|_{L^2(\Omega\times\R_+)}
\le
\sqrt{2\W(u_0)} .
\]

Consequently,
\[
t\mapsto\|\partial_t u(\,\cdot\,,t)\|_{L^2(\Omega)}^2
\in L^1(\R_+),
\]
and therefore there exists a sequence $(t_k)_{k\in\N}$ with
$\R_+\ni t_k\to\infty$ such that
\begin{align}
    \| \partial_t u (\,\cdot\,, t_k) \|^2_{L^2(\Omega)} 
    \underset{k\to\infty}{\longrightarrow} 
    0. \label{eq:l2dtu}
\end{align}

We may assume that $t_1>1$. 
The uniform $C^{4+\alpha}(\overline\Omega)$ bound obtained above implies that
the sequence $\{u(\,\cdot\,,t_k)\}_{k\in\N}$ is bounded in $C^{4+\alpha}(\overline\Omega)$.
Hence, by compactness, there exists a subsequence
$\{t_{k_\ell}\}_{\ell\in\N}$ and a function
\[
u_\infty \in C^{4+\beta}(\overline\Omega)
\qquad \text{for every } \beta\in(0,\alpha),
\]
such that
\[
u(\,\cdot\,,t_{k_\ell}) \to u_\infty
\qquad \text{in } C^{4+\beta}(\overline\Omega).
\]

Since $u$ satisfies the Willmore flow equation
the above convergence together with the uniform
$C^{4+\alpha}$ bounds and \eqref{eq:l2dtu} yields
\[
\Delta_{\Gamma(u(\cdot,t_k))}H
+2H\Big(\tfrac12 H^2-\Kk\Big)(u(\cdot,t_k))
\underset{k\to\infty}{\longrightarrow}  0
\qquad\text{in } C^0(\overline\Omega).
\]

Passing to the limit along the subsequence we conclude that
$u_\infty$ satisfies the stationary Willmore equation, hence,  $u_\infty$ is a critical point of the Willmore energy.
\end{proof}

In the subsequent theorem, we establish that as $t \rightarrow \infty$, there is convergence towards a unique critical point. This critical point is derived from the elliptic solution to the Willmore equation, as detailed in \cite[Section 5, Theorem 38]{gulyak2024boundary}.  At that point,  it was necessary to impose a smallness condition on the boundary data, which we also assume for the temporal limit for the Willmore flow solution.

\begin{thm}[Global Convergence]\label{thm:globconv}
Let $\alpha\in(0,1)$.
Then there exists a constant
\[
\Cl{deltaun}=\Cr{deltaun}(\Omega,\alpha)>0
\]
such that the following holds.
If 
\[
g_0\in C^{4+\alpha}(\partial\Omega),\qquad
g_1\in C^{3+\alpha}(\partial\Omega),\qquad
u_0\in C^{0,1}(\overline{\Omega}).
\]
with 
\[
\|u_0\|_{\Cc^{0,1}(\overline\Omega)}
+
\|g_0\|_{\Cc^{4+\alpha}(\partial\Omega)}
+
\|g_1\|_{\Cc^{3+\alpha}(\partial\Omega)}
\le \Cr{deltaun},
\]
then there exists a unique critical point
\[
u_\infty\in \Cc^{4+\alpha}(\overline\Omega)
\]
of the Willmore energy such that the corresponding global solution of
\eqref{Willmore} satisfies  for every $\beta\in(0,\alpha)$ 
\begin{align*}
    \|u(\,\cdot\,,t)-u_\infty\|_{C^{4+\beta}(\overline\Omega)}
    \le \Cl{new34} e^{-\Cr{new34} \cdot (t-1)} \quad \text{ for all $t\ge 1$},
\end{align*}
where $\Cr{new34}$ depends only on $\alpha, \beta$ and $\Omega$.
\end{thm}

\begin{proof}

From elliptic theory we know that, for $\Cr{deltaun}$ sufficiently small,
there exists a solution (see defintion of $\M^K_\delta$ in \cite[Theorem 21]{gulyak2024boundary})
\[
u_\infty \in C^{4+\alpha}(\overline{\Omega})\quad \text{ with }\quad \|u_\infty\|_{C^{2+\alpha}(\overline{\Omega})}\le \C \Cr{deltaun}  
\]
of the stationary Willmore equation with the prescribed boundary data.
We denote this solution by $u_\infty$, anticipating that it will be the
limit of the Willmore flow.

In this proof we first assume that
\begin{align}
\Cr{deltaun}\le \Cr{13}.
\tag{C1}\label{eq:condition1}
\end{align}
Then Theorem~\ref{thm:globalc1ac4a} provides a global solution $u$ satisfying
\[
\|u(\cdot,t)\|_{\Cc^{0,1}(\overline\Omega)}\le \tfrac13\Cr{delta}
\qquad \text{for all } t\ge0 .
\]
Hence Lemma~\ref{Lemma:schranken}, in particular estimate
\eqref{eq:klein}, yields
\begin{align*}
\|u\|_{\Cc^{4+\alpha,1+\alpha/4}_{1}
(\overline\Omega\times(t,t+1])}
\le
\Cr{11}\,\Cr{deltaun}.
\end{align*}
Consequently,
\begin{align}
\|u(\cdot,t)\|_{C^{4+\alpha}(\overline\Omega)}
\le
\Cr{11}\,\Cr{deltaun}
\qquad\text{for all } t\ge1 .
\label{eq:letzter}
\end{align}

Next, set \(w:=u-u_\infty\). Then \(w\) satisfies
\begin{align*}
\left\{
\begin{aligned}
\partial_t w
&\ = -\Delta^2w + f_0[u]-f_0[u_\infty] && &\\
&\ \quad +\nabla_i \big( f^i_1[u]- f^i_1[u_\infty] \big) + D^2_{ij} \big( f^{ij}_2[u] - f_2^{ij}[u_\infty] \big),
&& \text{in } \Omega\times(0,\infty),\\
w(\cdot,0) &= u_0-u_\infty,
&& \text{in } \Omega,\\
w &= 0,
&& \text{on } \partial\Omega\times[0,\infty),\\
\partial_\nu w &= 0,
&& \text{on } \partial\Omega\times[0,\infty).
\end{aligned}
\right.
\end{align*}

Multiplying the equation for $w=u-u_\infty$ by $w$ and integrating over
$\Omega$, an integration by parts together with the Poincaré inequality
(for $w=\partial_\nu w=0$ on $\partial\Omega$) and the Cauchy–Schwarz
inequality yields, for all $t\ge0$,
\begin{align*}
\frac12 \partial_t\|w\|_{L^2(\Omega)}^2
+\|D^2 w\|_{L^2(\Omega)}^2
&\le
\C\Big(
\|f_0[u]-f_0[u_\infty]\|_{L^2(\Omega)}^2
+\|f_1[u]-f_1[u_\infty]\|_{L^2(\Omega)}^2
\\
&\qquad+\|f_2[u]-f_2[u_\infty]\|_{L^2(\Omega)}^2
\Big).
\end{align*}
Using the structure of the nonlinearities,
we obtain
\begin{align*}
\|f_i[u]-f_i[u_\infty]\|_{L^2(\Omega)}
\le &\
\C
\sup_{t\ge1}\max\{\|u(\,\cdot\,,t\|_{C^2(\overline\Omega)},\|u_\infty\|_{C^2(\overline\Omega)}\}^2
\\
&\ \qquad \cdot\|D^2(u-u_\infty)\|_{L^2(\Omega)},
\end{align*}
for $i=0,1,2$. Hence
\begin{align*}
\frac12 \partial_t\|w\|_{L^2(\Omega)}^2
+\|D^2 w\|_{L^2(\Omega)}^2
\le
\Cl{333}
\Cr{deltaun}^2\|D^2 w\|_{L^2(\Omega)}^2 .
\end{align*}

Denoting by $\Cr{333}$ the corresponding constant depending only on
$\Omega$, we impose the smallness condition
\begin{align}
\Cr{333}\,\Cr{deltaun}^2 \le \frac12 ,
\tag{C2}\label{eq:condition2}
\end{align}
which will ensure that the right-hand side can be absorbed by the
left-hand side.



Therefore, we obtain for all $t\ge 1$
\begin{align*}
   \partial_t \|u(\,\cdot\,,t)-u_\infty\|^2_{L^2(\Omega)}+ & \ \Cl {9898}(\Omega)\|u(\,\cdot\,,t)-u_\infty\|^2_{L^2(\Omega)}
   \\
   \le& \
   \partial_t \|u(\,\cdot\,,t)-u_\infty\|^2_{L^2(\Omega)}+ \|D^2(u(\,\cdot\,,t)-u_\infty)\|^2_{L^2(\Omega)}
   \\
   \le&\
   0.
\end{align*}
By Grönwall's lemma, it follows for all $t\ge 1$
\begin{align*}
    \|u(\,\cdot\,,t)-u_\infty\|^2_{L^2(\Omega)}\le e^{-\Cr {9898}\cdot (t-1)} \|u(\,\cdot\,,1)-u_\infty\|^2_{L^2(\Omega)} \le \Cl{new11} e^{-\Cr {new11}\cdot (t-1)}
\end{align*}
where $\Cr{new11}$ depends only on $\alpha$ and $\Omega$.

Then one uses an interpolation inequality of the form
\[
\|w\|_{C^{4+\beta}(\overline{\Omega})}
\le
\Cl{zzz}\,\|w\|_{L^2(\Omega)}^{\theta}
\|w\|_{C^{4+\alpha}(\overline{\Omega})}^{1-\theta},
\]
for some $\theta=\theta(\alpha,\beta,\Omega)\in(0,1)$ and get together with the uniform bound \eqref{eq:letzter} all $t\ge 1$
\begin{align*}
    \|u(\,\cdot\,,t)-u_\infty\|_{C^{4+\beta}(\overline\Omega)}
    \le &\
    \Cr{zzz} \Cr{new11}^\theta e^{-\theta\Cr {new11}\cdot (t-1)} (\Cr{11} \cdot  \Cr{deltaun})^{1-\theta}\\
    \le &\ \Cr{new34} e^{-\Cr{new34} \cdot (t-1)}
\end{align*}
where $\Cr{new34}$ depends only on $\alpha, \beta$ and $\Omega$.
\end{proof}


\begin{corollary}[Global existence and convergence for homogeneous boundary data]
\label{cor:nullbd}
Let $\alpha\in(0,1)$ and let $\Omega\subset\R^2$ be a bounded domain with
$\Cc^{4+\alpha}$-boundary.
Assume homogeneous boundary data
\[
g_0\equiv 0,
\qquad
g_1\equiv 0 \ \text{on }\partial\Omega.
\]
Then there exists a constant
\[
\Cl{nulldelta}=\Cr{nulldelta}(\alpha,\Omega)>0
\]
such that the following holds: if
\[
\|u_0\|_{\Cc^{0,1}(\overline\Omega)}<\Cr{nulldelta},
\]
then the solution $u$ of the Willmore flow problem \eqref{Willmore}
exists globally in time, i.e. for all $t>0$, and
\[
u(t)\to u_\infty \equiv 0
\quad \text{as } t\to\infty .
\]
\end{corollary}

{\bf Acknowledgments.}
A part of this research was carried out during the author’s PhD studies under the supervision of Hans-Christoph Grunau.
The author is deeply indebted to him for his guidance, careful reading, corrections, and many helpful discussions.
The author is also grateful to Klaus Deckelnick for valuable suggestions.
Finally, the author acknowledges Matteo Novaga and Matthias Röger for their careful review of the author’s PhD thesis and for their valuable comments and suggestions.

\section{Appendix}
\begin{lemma}[Structure of the remainder term]\label{thm:structure-remainder-term}
Let $u\in C^4(\Omega)$ and assume that the remainder term $\Rr(\nabla u,D^2u,D^3u)$ is given by
\[
\Rr(\nabla u,D^2u,D^3u)
=
-f_0[u]-\nabla_i f_1^i[u]-D^2_{ij}f_2^{ij}[u]+\Delta^2u-L(\nabla u)D^4u.
\]
Then $\Rr$ admits the representation
\begin{align}
\Rr(\nabla u,D^2u,D^3u)
=&\ 
D^3u\star D^2u \star \sum_{k=1}^4 Q^{-2k} P_{2k-1}(\nabla u)
\notag\\
&\quad
+ D^2u\star D^2u\star D^2u \star \sum_{k=0}^4 Q^{-2(k+1)} P_{2k}(\nabla u).
\label{eq:R-structure} \tag{$\mathcal{R}$}
\end{align}
In particular, the remainder term consists only of terms which are either
linear in $D^3u$ and quadratic in $D^2u$, or cubic in $D^2u$, with
coefficients that are smooth rational functions of $\nabla u$ of the form
$Q^{-2m}P_\ell(\nabla u)$.
\end{lemma}

\begin{proof}
Starting from
\begin{align*}
\Rr(\nabla u,D^2u,D^3u)
=&\ 
-f_0[u]-\nabla_i f_1^i[u]-D^2_{ij}f_2^{ij}[u]+\Delta^2u-L(\nabla u)D^4u,
\end{align*}
we expand the divergence terms according to the product rule. Since
\[
D(Q^{-2\ell})
=
D^2u\star\nabla u\star Q^{-2(\ell+1)},
\] this yields
\begin{align*}
\Rr(\nabla u,D^2u,D^3u)
=&\ 
D^2 u \star D^2 u\star D^2 u\star \sum_{k=1}^4 Q^{-2k} P_{2k-2} (\nabla u) \\
&\ 
+ D\!\left( D^2 u \star D^2 u\star  \sum_{k=1}^4 Q^{-2k} P_{2k-1} (\nabla u) \right)\\
&\ 
+ D^2\!\left( D^2 u \star  \sum_{k=1}^2 Q^{-2k} P_{2k} (\nabla u)\right)
+\Delta^2 u- L(\nabla u) D^4 u .
\end{align*}
Expanding the first- and second-order derivatives once more gives
\begin{align*}
\Rr(\nabla u,D^2u,D^3u)
=&\ 
D^2 u \star D^2 u\star D^2 u\star \sum_{k=1}^4 Q^{-2k} P_{2k-2} (\nabla u) \\
&\ 
+ D^3 u \star D^2 u\star  \sum_{k=1}^4 Q^{-2k} P_{2k-1} (\nabla u) \\
&\ 
+ D^2 u \star D^2 u\star D^2 u \star  \sum_{k=1}^4 Q^{-2(k+1)} P_{2k} (\nabla u)\\
&\ 
+ D\!\left( D^3 u \star  \sum_{k=1}^2 Q^{-2k} P_{2k} (\nabla u)
+ D^2 u \star D^2u\star  \sum_{k=1}^3 Q^{-2k} P_{2k-1} (\nabla u)\right) \\
&\ 
+\Delta^2 u- L(\nabla u) D^4 u .
\end{align*}
After distributing the remaining derivative and collecting terms of the same
type, we obtain
\begin{align*}
\Rr(\nabla u,D^2u,D^3u)
=&\ 
D^2 u \star D^2 u\star D^2 u\star \sum_{k=1}^5 Q^{-2k} P_{2k-2} (\nabla u)  \\
&\ 
+ D^3 u \star D^2 u\star  \sum_{k=1}^4 Q^{-2k} P_{2k-1} (\nabla u)\\
&\ 
+ D^3 u \star D^2 u\star  \sum_{k=1}^2 Q^{-2(k+1)} P_{2k+1} (\nabla u) \\
&\ 
+ D^3 u \star D^2u\star  \sum_{k=1}^3 Q^{-2k} P_{2k-1} (\nabla u)\\
&\ 
+ D^2u\star D^2 u \star D^2u\star  \sum_{k=1}^3 Q^{-2(k+1)} P_{2k} (\nabla u).
\end{align*}
All contributions are therefore of one of the two forms
\[
D^3u\star D^2u \star Q^{-2k}P_{2k-1}(\nabla u)
\qquad\text{or}\qquad
D^2u\star D^2u\star D^2u \star Q^{-2(k+1)}P_{2k}(\nabla u).
\]
Absorbing overlapping index ranges into the sums proves \eqref{eq:R-structure}.
\end{proof}
\begin{thm}[Diameter estimate]\label{thm:imdiamest}
Let $f:\Sigma \to \mathbb{R}^3$ be a compact two–dimensional surface
with boundary, and denote its image by $\M := f(\Sigma)\subset\mathbb{R}^3$.
Then the following estimate holds:
\begin{align}
\operatorname{diam}(\M)
\le
\frac{16}{\pi}
\left(
\int_{\M} |H| \, d\mathcal H^2
+
\frac{\pi}{2}\,
\mathcal H^1(\partial \M)
\right).
\end{align}
\end{thm}

\begin{proof}
The estimate follows from \cite[Theorem~1.1]{miura2022diameter}
applied with $n=3$. 
Using the bound $C_T(3)\ge \pi/16$, which is established in
\cite[(A.1)]{miura2022diameter}, yields the stated inequality.
\end{proof}


    \begin{thm}[Hölder Interpolation Result] \label{interpolappendix}
For $\Omega\subset \R^n$, $0\le a\le b$ and $\Omega$ bounded with $C^b$ boundary, $0<\lambda<1$ there is a constant $\Cl{int}=\Cr{int}(\Omega, a,b)$ such that
\begin{align}
    \|u\|_{C^a(\overline{\Omega})}\le  \Cr{int} \|u\|_{C^b(\overline{\Omega})}, \qquad 
    \| u\|_{C^{\lambda a +(1-\lambda)b}(\overline{\Omega})} \le \Cr{int} 
    \|u\|_{C^{a}(\overline{\Omega})}^\lambda \cdot  \|u\|_{C^{b}(\overline{\Omega})}^{1-\lambda}.
\end{align}
with $C^k(\overline{\Omega})=C^{k-1,1}(\overline{\Omega})$ in this Theorem. 
\end{thm}
\begin{proof}
We apply \cite[Theorem A.5]{hormander1976boundary}, which proves the
interpolation inequality for functions defined on convex sets.
Since $\Omega$ need not be convex, we extend $u$ to a convex domain.

Let $b=k+\alpha$, and $R>0$ with $\overline{\Omega}\subset B_R(0)$. By the proof of \cite[Lemma 6.37
]{gilbarg2001elliptic} there exists $\Cl{ext}=\Cr{ext}(\Omega, k,\alpha)$  and an extension operator $T\colon\Cc^{k+\alpha}(\overline{\Omega})\to \Cc^{k+\alpha}_0(B_R(0))$ such that 
\begin{align*}
    \forall\  0\le j+\beta \le k+\alpha: 
    \quad \|T u\|_{\Cc^{j+\beta}_0(B_R(0))} \le \Cr{ext} \| u\|_{\Cc^{j+\beta}(\overline{\Omega})}.
\end{align*}
Then by \cite[Theorem A.5 p. 50]{hormander1976boundary} applied in $B_R(0)$ with a constant $\Cl{hor}=\Cr{hor}(a,b,R)$ it follows
\begin{align*}
     \| u\|_{C^{\lambda a +(1-\lambda)b}(\overline{\Omega})} \le &\  
     \| T u\|_{C^{\lambda a +(1-\lambda)b}_0(B_R(0))}
    \le \Cr{hor} 
    \|Tu\|_{C^{a}_0(B_R(0))}^\lambda \cdot  \|Tu\|_{C^{b}_0(B_R(0))}^{1-\lambda} \\
    \le &\ \Cr{hor} \big(\Cr{ext}
    \|u\|_{C^{a}(\overline{\Omega})}\big)^\lambda \cdot  \big(\Cr{ext}\|u\|_{C^{b}(\overline{\Omega})}\big)^{1-\lambda} \\
    \le &\ \Cr{hor}\Cr{ext}    \|u\|_{C^{a}(\overline{\Omega})}^\lambda \cdot  \|u\|_{C^{b}(\overline{\Omega})}^{1-\lambda}.
\end{align*}
We finish the proof by setting $\Cr{int}=\Cr{hor}\Cr{ext}$.
\end{proof}

For the study of the long-time existence of the Willmore flow solutions, we also need a Hölder interpolation involving $L^2(\Omega)$-norm, over which we will have more control than over $\Cc^0(\overline{\Omega})$-norm. The idea is the same as in \cite[after 5.16] {dall2016lojasiewicz} \index{Hölder!  Lebesgue interpolation} \index{interpolation!  Hölder-$L^p$}

\begin{thm}[Hölder-$L^2$ interpolation]
\label{thm:interpolation}
Let $m\in\N_0$, let $0<\alpha<1$, and let
$\Omega\subset\R^2$ be a bounded domain with
$\partial\Omega\in C^{m+\alpha}$.
Then there exist constants
\[
C=C(m,\alpha,\Omega)>0,
\qquad
\theta=\theta(m,\alpha)\in(0,1),
\]
such that for every
$u\in C^{m+\alpha}(\overline\Omega)$,
\[
\|u\|_{C^m(\overline\Omega)}
\le
C\,
\|u\|_{C^{m+\alpha}(\overline\Omega)}^{1-\theta}
\|u\|_{L^2(\Omega)}^\theta .
\]
\end{thm}
\begin{proof}
We use the interpolation results for Besov spaces by  \cite[Thm 6.4.5 (3)]{bergh1976interpolation} with $\theta=2(b-a)/(4+2b)$:
\begin{align*}
    \big(B^{s_0}_{p_0,q_0}, B^{s_1}_{p_1,q_1}\big)_{\theta,p^\ast} = B^{s^\ast}_{p^\ast ,q^\ast}, \qquad (s_0\ne s_1, p^\ast=q^\ast, 1\le p_0,p_1,q_0,q_1\le \infty)
\end{align*}
with additional restrictions
\begin{align*}
    s^\ast=(1-\theta)s_0+\theta s_1 ,\qquad 
    \frac{1}{p^\ast}= \frac{1-\theta}{p_0 }+ \frac{\theta}{p_1}, \qquad
    \frac{1}{q^\ast}= \frac{1-\theta}{q_0 }+ \frac{\theta}{q_1}.
\end{align*}
With $q_0 = p_0=2$ by \cite[Theorem 4.6.1 (b) p. 328 ]{triebel1978interpolation} we get $B^s_{2,2}=H^s_{2}$ the generalized Sobolev Space. By \cite[Subsection 1.5.1 Theorem (ii) p.29]{triebel1992theory} for $s>0$, then $B^s_{\infty,\infty}=\mathcal C ^s$ ( Hölder–Zygmund space \cite[Subsection 1.2.2]{triebel1992theory}) and by \cite[Subsection 1.2.2 Theorem (ii)]{triebel1992theory} if $0<s\ne$ integer, then $C^s=\mathcal C^s$
\begin{align}
\big(B^{0}_{2,2}, B^{s_1}_{\infty,\infty }\big)_{\theta,p^\ast} = B^{s^\ast}_{p^\ast, q^\ast}.
\end{align}
To fulfill the above restrictions we set $p^\ast= q^\ast= 2/(1-\theta)$ and $s^\ast = \theta s_0$, thus
\begin{align}
\big( L^2(\Omega), C^{s_1}(\overline{\Omega})\big)_{\theta,p^\ast} = B^{\theta s_1}_{\frac{2}{1-\theta}, \frac{2}{1-\theta}}.
\end{align}
By \cite[Theorem 4.6.1. (b) p.327]{triebel1978interpolation}: For $1<p<\infty, 1\le r\le\infty, t\ge 0$ and $s> t + n/p$ it follows $B^s(\Omega)\subset  C^t(\overline\Omega)$, and by the proof of (e) and (d) of 2.8.1 on page 205 in \cite{triebel1978interpolation} the embedding is continuous. In our case: $\theta s_1> t + (1-\theta)$ so that $s_1 /(1-1/\theta)> t$, so we can find $\theta\in (0,1)$ for every $t<s_1$. Next, we use the interpolation inequality \cite[Thm 1.3.3. p.25]{triebel1978interpolation}
\begin{align*}
    \|u\| _{ C^{t}(\overline{\Omega}	)} \le C \|u\|^{1-\theta}_{L^2(\Omega)} \|u\|^{\theta}_{C^{s_1}(\overline{\Omega})}.
\end{align*}
We used  $\partial\Omega\in C^{m+\alpha}$ for technical reasons related to smoothing. 
\end{proof}

In following lemmas, we repeatedly use the following elementary facts.

\smallskip
\noindent
\emph{(i) Reduction of Hölder exponents in time.}
If $0<\mu<\nu\le 1$ and $f\in C^\nu([0,T])$, then
\begin{align}
 [f]_{C^\mu([0,T])}\le T^{\nu-\mu}[f]_{C^\nu([0,T])}.
\label{eq:i}   
\end{align}

\smallskip
\noindent
\emph{(ii) Vanishing initial trace.}
If $f(0)=0$ and $f\in C^\nu([0,T])$ for some $\nu\in(0,1]$, then
\begin{align}
\sup_{t\in[0,T]}|f(t)|\le T^\nu [f]_{C^\nu([0,T])}.
\label{eq:ii}   
\end{align}

\smallskip
\noindent
\emph{(iii) Interpolation of spatial Hölder seminorms.}
For $0<\gamma<\alpha<1$ and $v\in C^\alpha(\overline\Omega)$,
\begin{align}
[v]_{C^\gamma(\overline\Omega)}
\le
2^{1-\gamma/\alpha}
[v]_{C^\alpha(\overline\Omega)}^{\gamma/\alpha}
\|v\|_{C^0(\overline\Omega)}^{1-\gamma/\alpha}.
\label{eq:iii}   
\end{align}
\begin{lemma}\label{gammaalphaappendix}
Let $m\in\{1,2,3,4\}$ and let
$u\in \Cc^{4+\alpha,1+\alpha/4}_{m+\alpha}(Q_T)$ satisfy
\[
D_t^kD_x^\beta u(x,0)=0
\qquad
\text{for all }x\in\overline\Omega
\text{ and }4k+|\beta|\le m.
\]
Then there exists a constant
\[
\Cr{constgammaalpha}=\Cr{constgammaalpha}(\alpha,\gamma)>0
\]
such that for every $0<\gamma<\alpha$ and every $T\le1$,
\[
\|u\|_{\Cc^{4+\gamma,1+\gamma/4}_{m+\gamma}(Q_T)}
\le
\Cr{constgammaalpha}\,
T^{\frac{\alpha-\gamma}{4}}
\|u\|_{\Cc^{4+\alpha,1+\alpha/4}_{m+\alpha}(Q_T)}.
\]
\end{lemma}
\begin{proof}
Set
\[
M:=\|u\|_{\Cc^{4+\alpha,1+\alpha/4}_{m+\alpha}(Q_T)}.
\]
We show that every term occurring in the
$\Cc^{4+\gamma,1+\gamma/4}_{m+\gamma}(Q_T)$-norm
is bounded by $C(\alpha,\gamma)\,T^{\frac{\alpha-\gamma}{4}}M$.

Then we split the $\Cc^{4+\gamma,1+\frac{\gamma}{4}}_{m+\gamma}(Q_T)$–norm
into several parts:
\begin{align*}
\|u\|_{\Cc^{4+\gamma,1+\frac{\gamma}{4}}_{m+\gamma}(Q_T)}
= &\
\underbrace{\sum_{m<4k+|\beta|\le4}
\sup_{(x,t)\in\overline\Omega\times(0,T]}
t^{\frac{4k+|\beta|-m-\gamma}{4}}
\big|D_t^k D_x^\beta u(x,t)\big|}_{\cbm 1}
\\
&\ + \|u\|_{\Cc^{m+\gamma,\frac{m+\gamma}{4}}_{x,t}(\overline Q_T)}
+ \sup_{t\in(0,T]} t^{\frac{4-m}{4}} [u]_{Q_t'}^{4+\gamma}.
\end{align*}

We further decompose the second and the third term as
\begin{align*}
\|u\|_{\Cc^{m+\gamma,\frac{m+\gamma}{4}}_{x,t}(\overline Q_T)}
= &\
\underbrace{
\sum_{m-3\le 4k+|\beta|\le m}
\sup_{x\in\overline\Omega}
\Big[
D_t^k D_x^\beta u(x,\cdot)
\Big]_{\Cc^{\frac{m+\gamma-4k-|\beta|}{4}}([0,T])}
}_{\cbm 2}
\\
&\ +
\underbrace{
\sum_{4k+|\beta|\le m}
\sup_{(x,t)\in\overline Q_T}
\big|D_t^k D_x^\beta u(x,t)\big|
}_{\cbm 3}
\\
&\ +
\underbrace{
\sum_{4k+|\beta|=m}
\sup_{t\in[0,T]}
\Big[
D_t^k D_x^\beta u(\cdot,t)
\Big]_{\Cc^\gamma(\overline\Omega)}
}_{\cbm 4}.
\\
[u]_{Q_t'}^{4+\gamma}
= &\
\underbrace{
\sum_{1\le 4k+|\beta|\le4}
\sup_{x\in\overline\Omega}
\Big[
D_t^k D_x^\beta u(x,\cdot)
\Big]_{\Cc^{\frac{4-4k-|\beta|+\gamma}{4}}([t/2,t])}
}_{\cbm 5}
\\
&\ +
\underbrace{
\sum_{4k+|\beta|=4}
\sup_{t\in[t/2,t]}
\Big[
D_t^k D_x^\beta u(\cdot,t)
\Big]_{\Cc^\gamma(\overline\Omega)}
}_{\cbm 6}.
\end{align*}

We now estimate the different parts of the norm.

\medskip
\noindent
\cb 1 \emph{Weighted pointwise terms of parabolic order \(>m\).}
For every $m<4k+|\beta|\le4$ we have
\begin{align*}
\sup_{(x,t)\in Q_T}
t^{\frac{4k+|\beta|-m-\gamma}{4}}
|D_t^kD_x^\beta u(x,t)|
&=
\sup_{(x,t)\in Q_T}
t^{\frac{\alpha-\gamma}{4}}
t^{\frac{4k+|\beta|-m-\alpha}{4}}
|D_t^kD_x^\beta u(x,t)| \\
&\le
T^{\frac{\alpha-\gamma}{4}}\,M.
\end{align*}

\medskip
\noindent
\cb 2 \emph{Temporal Hölder seminorms in the low-order part.}
Let $m-3\le 4k+|\beta|\le m$. Then
\[
\frac{m+\gamma-4k-|\beta|}{4}
<
\frac{m+\alpha-4k-|\beta|}{4},
\]
hence by in \eqref{eq:i} in \textit{(i)},
\begin{align*}
\sup_{x\in\overline\Omega}
\big[D_t^kD_x^\beta u(x,\cdot)\big]_
{C^{\frac{m+\gamma-4k-|\beta|}{4}}([0,T])}
\le
T^{\frac{\alpha-\gamma}{4}}
\sup_{x\in\overline\Omega}
\big[D_t^kD_x^\beta u(x,\cdot)\big]_
{C^{\frac{m+\alpha-4k-|\beta|}{4}}([0,T])}
\le
T^{\frac{\alpha-\gamma}{4}}M.
\end{align*}

\medskip
\noindent
\cb 3 \emph{Pointwise terms of parabolic order \(\le m\).}
We first consider the case $m-3\le 4k+|\beta|\le m$.
Since
\[
D_t^kD_x^\beta u(x,0)=0
\qquad\text{for all }x\in\overline\Omega,
\]
we may apply (ii) to the function
$t\mapsto D_t^kD_x^\beta u(x,t)$ and obtain
\begin{align*}
\sup_{(x,t)\in\overline Q_T}|D_t^kD_x^\beta u(x,t)|
&\le
T^{\frac{m+\alpha-4k-|\beta|}{4}}
\sup_{x\in\overline\Omega}
\big[D_t^kD_x^\beta u(x,\cdot)\big]_
{C^{\frac{m+\alpha-4k-|\beta|}{4}}([0,T])}
\\
&\le
T^{\frac{\alpha-\gamma}{4}}M,
\end{align*}
because \(m+\alpha-4k-|\beta|\ge \alpha>\alpha-\gamma\) and \(T\le1\).

It remains to treat the case \(m=4\) and \(4k+|\beta|=0\), i.e. the
term \(u\) itself.
Since \(u(x,0)=0\), we use \(u_t\in C^0(\overline Q_T)\) and obtain
\[
|u(x,t)|
=
\left|\int_0^t u_t(x,\tau)\,d\tau\right|
\le
T\,\|u_t\|_{C^0(\overline Q_T)}
\le
T^{\frac{\alpha-\gamma}{4}}M,
\]
again because \(T\le1\).

\medskip
\noindent
\cb 4 \emph{Spatial Hölder seminorms of parabolic order \(m\).}
Let \(4k+|\beta|=m\).
By Step~3 and the vanishing initial trace,
\[
\|D_t^kD_x^\beta u(\,\cdot\,,t)\|_{C^0(\overline\Omega)}
\le
M\,T^{\alpha/4}.
\]
Therefore, by \eqref{eq:iii} in \emph{(iii)},
\begin{align*}
[D_t^kD_x^\beta u(\,\cdot\,,t)]_{C^\gamma(\overline\Omega)}
&\le
2^{1-\gamma/\alpha}
[D_t^kD_x^\beta u(\,\cdot\,,t)]_{C^\alpha(\overline\Omega)}^{\gamma/\alpha}
\|D_t^kD_x^\beta u(\,\cdot\,,t)\|_{C^0(\overline\Omega)}^{1-\gamma/\alpha}
\\
&\le
2^{1-\gamma/\alpha}
M\,T^{\frac{\alpha-\gamma}{4}}.
\end{align*}
Taking the supremum over \(t\in[0,T]\) gives the desired bound for this
part of the norm.

\medskip
\noindent
\cb 5 \emph{Weighted temporal Hölder seminorms in \([u]^{4+\gamma}_{Q_t'}\).}
Let \(1\le 4k+|\beta|\le4\).
Applying \eqref{eq:i} in \emph{(i)} on the interval \([t/2,t]\) yields
\begin{align*}
t^{\frac{4-m}{4}}
\big[D_t^kD_x^\beta u(x,\cdot)\big]_
{C^{\frac{4+\gamma-4k-|\beta|}{4}}([t/2,t])}
&\le
C(\alpha,\gamma)\,
t^{\frac{4-m}{4}}
t^{\frac{\alpha-\gamma}{4}}
\big[D_t^kD_x^\beta u(x,\cdot)\big]_
{C^{\frac{4+\alpha-4k-|\beta|}{4}}([t/2,t])}
\\
&\qquad\le
C(\alpha,\gamma)\,T^{\frac{\alpha-\gamma}{4}}M.
\end{align*}

\medskip
\noindent
\cb 6 \emph{Weighted spatial Hölder seminorms of parabolic order \(4\).}
Let \(4k+|\beta|=4\), and 
\(
t'\in[t/2,t].
\)
From the weighted norm of order \(4+\alpha\) and \cbm 1 we have
\[
t^{\frac{4-m}{4}}[D_t^kD_x^\beta u(\,\cdot\,,t')]_{C^\alpha(\overline\Omega)}\le M, \quad \|D_t^kD_x^\beta u(\,\cdot\,,t')\|_{C^0(\overline\Omega)}
\le
M\,(t')^{\frac{m+\alpha-4}{4}}
\le
C\,M\,T^{\frac{m+\alpha-4}{4}},
\]

Applying \eqref{eq:iii} in \emph{(iii)} to \(D_t^kD_x^\beta u(\,\cdot\,,t')\) gives
\begin{align*}
t^{\frac{4-m}{4}}[D_t^kD_x^\beta u(\,\cdot\,,t')]_{C^\gamma(\overline\Omega)}
&\le
C(\alpha,\gamma)\,
t^{\frac{4-m}{4}}
[D_t^kD_x^\beta u(\,\cdot\,,t')]_{C^\alpha(\overline\Omega)}^{\gamma/\alpha}
\|D_t^kD_x^\beta u(\,\cdot\,,t')\|_{C^0(\overline\Omega)}^{1-\gamma/\alpha}
\\
&\le
C(\alpha,\gamma)\,
M\,T^{\frac{\alpha-\gamma}{4}}.
\end{align*}
Taking the supremum over \(t'\in[t/2,t]\) and then over \(t\in(0,T]\)
gives the required estimate.

\medskip
Collecting the bounds and absorbing the finite number of
terms into the constant proves the lemma.
\end{proof}


\begin{lemma}\label{productruleappendix}
Let $m=1,2,3,4$, let $0<\gamma\le\alpha<1$, and assume
$\gamma\ge\alpha/2$.
If
\[
u,v,w\in \Cc^{4+\gamma,1+\gamma/4}_{m+\gamma}(Q_T)
\qquad\text{and}\qquad T\le1,
\]
then there exists a constant
\[
\Cr{prodrule}=\Cr{prodrule}(\alpha,\gamma,\Omega)>0
\]
depending only on $\alpha,\gamma,\Omega$ and on the fixed algebraic
structure of $\Rr$ and $L$ such that
\begin{align}
&\|\nabla u\|_{\Cc^{\alpha,\alpha/4}_{\max\{0,m+\alpha-4\}}(Q_T)}
\le
\Cr{prodrule}
\|\nabla u\|_{\Cc^{3+\gamma,\frac{3+\gamma}{4}}_{m+\gamma-1}(Q_T)},
\tag{\ref{eq:1}}\label{eq:1-app}
\\
&\ \|D^3wD^2 u\|_{C^{\alpha, \alpha/4}_{m+\alpha-4}(Q_T)}  
     \le 
     \Cr{prodrule} \|D^3w\|_{C^{1+\gamma, \frac{1+\gamma}{4}}_{m+\gamma-3}(Q_T)} 
     \cdot \|D^2u\|_{C^{2+\gamma, \frac{2+\gamma}{4}}_{m+\gamma-2}(Q_T)}  \tag{\ref{eq:32}} \\ 
 &\begin{aligned}
            \|D^2 u D^2 &\ w D^2v\|_{C^{\alpha, \alpha/4}_{m+\alpha-4}(Q_T)} \\
           &\ \le 
          \Cr{prodrule} \|D^2u\|_{C^{2+\gamma, \frac{2+\gamma}{4}}_{m+\gamma-2}(Q_T)} \cdot \|D^2w\|_{C^{2+\gamma, \frac{2+\gamma}{4}}_{m+\gamma-2}(Q_T)} 
          \cdot \|D^2v\|_{C^{2+\gamma, \frac{2+\gamma}{4}}_{m+\gamma-2}(Q_T)}  .
        \end{aligned} \tag{\ref{eq:222}}
\end{align}
\end{lemma}
\begin{proof}
For $m=4$ the statement reduces to the corresponding product estimates in
unweighted anisotropic Hölder spaces; see \cite{gulyak2017willmore}.
Hence it remains to consider the cases $m=1,2,3$.

\medskip\noindent\cb{1}\ \emph{Decomposition of the weighted norm.}
The norm
\[\|h\|_{\Cc^{\alpha,\alpha/4}_{m+\alpha-4}(Q_T)}=I_0+I_x+I_t\]
consists of three parts :
\begin{align*}
I_0^m(h)
&:=
\sup_{(x,t)\in Q_T}
t^{\frac{4-m-\alpha}{4}}|h(x,t)|,\\[0.2em]
I_x^m(h)
&:=
\sup_{0<t\le T}
t^{\frac{4-m}{4}}
\sup_{t'\in[t/2,t]}
[h(\cdot,t')]_{\Cc^\alpha(\overline\Omega)},\\[0.2em]
I_t^m(h)
&:=
\sup_{0<t\le T}
t^{\frac{4-m}{4}}
\sup_{x\in\overline\Omega}
[h(x,\cdot)]_{\Cc^{\alpha/4}([t/2,t])}.
\end{align*}
We estimate these three contributions separately.

\medskip\noindent\cb{2}\ \emph{Pointwise part  \(I_0^m(D^3w\,D^2u)\).}
Since \(\gamma\ge\alpha/2\) and \(T\le1\), the factor
\(T^{\frac{2\gamma-\alpha}{4}}\le1\) we have
\begin{align*}
I_0^1(D^3w\,D^2u)
&=
\sup_{(x,t)\in Q_T}
t^{\frac{3-\alpha}{4}}|D^3w(x,t)|\,|D^2u(x,t)| \\
&\le
\sup_{(x,t)\in Q_T}
t^{\frac{2-\gamma}{4}}|D^3w(x,t)|
\,
t^{\frac{1-\gamma}{4}}|D^2u(x,t)|
\,
T^{\frac{2\gamma-\alpha}{4}}
\\
&
\le
\|D^3w\|_{\Cc^{1+\gamma,(1+\gamma)/4}_{1+\gamma-3}(Q_T)} \|D^2u\|_{\Cc^{2+\gamma,(2+\gamma)/4}_{1+\gamma-2}(Q_T)},
\\
I_0^2(D^3w\,D^2u)
&=
\sup_{(x,t)\in Q_T}
t^{\frac{2-\alpha}{4}}|D^3w(x,t)|\,|D^2u(x,t)| \\
&\le
\sup_{(x,t)\in Q_T}
t^{\frac{1-\gamma}{4}}|D^3w(x,t)|
\,|D^2u(x,t)|
\,T^{\frac{1+\gamma-\alpha}{4}}
\\
&\le
\|D^3w\|_{\Cc^{1+\gamma,(1+\gamma)/4}_{2+\gamma-3}(Q_T)} \|D^2u\|_{\Cc^{2+\gamma,(2+\gamma)/4}_{2+\gamma-2}(Q_T)},
\\
I_0^3(D^3w\,D^2u)
& \le
\sup_{Q_T}|D^3w|\,\sup_{Q_T}|D^2u|
\\
&\le
\|D^3w\|_{\Cc^{1+\gamma,(1+\gamma)/4}_{3+\gamma-3}(Q_T)} \|D^2u\|_{\Cc^{2+\gamma,(2+\gamma)/4}_{3+\gamma-2}(Q_T)},
\end{align*}

\newpage
\medskip\noindent\cb{3}\ \emph{Spatial Hölder seminorm  \(I_x^m(D^3w\,D^2u)\).}
For fixed \(t'\in[t/2,t]\), the elementary product rule gives
First, in preparation for the later estimates, we omit the weights and conclude for $t'\in [t/2,t]$
\begin{align*}
    \big[D^3w\cdot D^2u(\,\cdot\,,t')\big]_{C^\alpha(\overline\Omega)}\! = & \
    \sup_{x\in\overline\Omega} \big|D^3w(x,t')\big|\cdot \big[ D^2u(\,\cdot\,,t')\big]_{C^\alpha(\overline\Omega)} \! 
    \\
    &\ +\Big[D^3w(\,\cdot\,,t')\Big]_{C^\alpha(\overline\Omega)} \cdot \sup_{x\in\overline\Omega} \big|D^2u(x,t')\big| \\
    \le & \  
    \sup_{x\in\overline\Omega} \big|D^3w(x,t')\big|\cdot \sup_{x\in\overline\Omega} \big|D^3u(x,t')\big|^{\alpha} \cdot
    \sup_{x\in\overline\Omega} \big|D^2u(x,t')\big|^{1-\alpha}\\
     & \ +\sup_{x\in\overline\Omega} \big|D^4w(x,t')\big|^{\alpha} \cdot \sup_{x\in\overline\Omega} \big|D^3u(x,t')\big|^{1-\alpha}\cdot\sup_{x\in\overline\Omega} \big|D^2u(x,t')\big|
\end{align*}
By using \(\gamma\ge\alpha/2\) and \(T\le1\), we get for $m=1,2,3$
\begin{align*}
  t^{\frac{4-1}{4}} \big[D^3w\, & D^2u(\,\cdot\,,t')\big]_{C^\alpha(\overline\Omega)}\\
  \le &\ T^{\frac{2\gamma-\alpha} 4} t^{\frac{3-2\gamma+\alpha}4}\big[D^3w\cdot D^2u(\,\cdot\,,t')\big]_{C^\alpha(\overline\Omega)}  \\
   \le &\  \sup_{x\in\overline\Omega} \big|t^{\frac{2-\gamma}4}D^3w(x,t')\big|\cdot \sup_{x\in\overline\Omega} \big|t^{\frac{2-\gamma}4}D^3u(x,t')\big|^{\alpha} \cdot
    \sup_{x\in\overline\Omega} \big|t^{\frac{1-\gamma}4}D^2u(x,t')\big|^{1-\alpha}   \\
    &\ + \sup_{x\in\overline\Omega} \big| t^{\frac{3-\gamma}4}D^4w(x,t')\big|^{\alpha} \cdot \sup_{x\in\overline\Omega} \big|t^{\frac{2-\gamma}4}D^3w(x,t')\big|^{1-\alpha}\cdot\sup_{x\in\overline\Omega} \big|t^{\frac{1-\gamma}4} D^2u(x,t')\big|, 
    \\
   t^{\frac{4-2}{4}} \big[D^3w\, & D^2u(\,\cdot\,,t')\big]_{C^\alpha(\overline\Omega)} \\
   \le &\ T^{\frac{\gamma(1+\alpha)}{4}}\sup_{x\in\overline\Omega} \big|t^{\frac{1-\gamma}4}D^3w(x,t')\big|\cdot \big|t^{\frac{1-\gamma}4}D^3u(x,t')\big|^{\alpha} \cdot
    \sup_{x\in\overline\Omega} \big|D^2u(x,t')\big|^{1-\alpha}   \\
    &\ + T^{\frac{1+\gamma-\alpha}4} \sup_{x\in\overline\Omega} \big| t^{\frac{2-\gamma}4}D^4w(x,t')\big|^{\alpha} \cdot \sup_{x\in\overline\Omega} \big|t^{\frac{1-\gamma}4}D^3w(x,t')\big|^{1-\alpha}\cdot\sup_{x\in\overline\Omega} \big| D^2u(x,t')\big|,
    \\
     t^{\frac{4-3}{4}} \big[D^3w\, & D^2u(\,\cdot\, ,t')\big]_{C^\alpha(\overline\Omega)} \\
   \le &\ T^{\frac{1}{4}} \sup_{x\in\overline\Omega} \big|D^3w(x,t')\big|\cdot \big|D^3u(x,t')\big|^{\alpha} \cdot
    \sup_{x\in\overline\Omega} \big|D^2u(x,t')\big|^{1-\alpha}   \\
    &\ + T^{\frac{1-(1-\gamma)\alpha}4} \sup_{x\in\overline\Omega} \big| t^{\frac{1-\gamma}4}D^4w(x,t')\big|^{\alpha} \cdot \sup_{x\in\overline\Omega} \big|D^3w(x,t')\big|^{1-\alpha}\cdot\sup_{x\in\overline\Omega} \big| D^2u(x,t')\big|.
\end{align*} 
After inserting $\sup_{0<t\le T}\sup_{t'\in[t/2,t]}$ and the weighted norms, we obtain
\[
I_x^m(D^3w\,D^2u)\le C(\alpha,\gamma)\,\|D^3w\|_{\Cc^{1+\gamma,(1+\gamma)/4}_{3+\gamma-3}(Q_T)} \|D^2u\|_{\Cc^{2+\gamma,(2+\gamma)/4}_{3+\gamma-2}(Q_T)}.
\]

\medskip\noindent\cb{4}\ \emph{Temporal Hölder seminorm  \(I_t^m(D^3w\,D^2u)\).}
For fixed \(x\in\overline\Omega\), we use \eqref{eq:i} from \emph{(i)} and write
\begin{align*}
  \big[ D^3w \,  D^2u(x, .\,)\big]_{C^{\frac\alpha 4}([t/2,t])}
 \le &\ 
 t^{\frac{2+\gamma-\alpha}{4}} \sup_{t'\in[t/2,t]}\big| D^3w(x, t')\big|
     \cdot   \big[D^2u(x, .\,)\big]_{C^{\frac{2+\gamma}{4}}([t/2,t])} \\
    &\ + t ^{\frac{1+\gamma-\alpha}{4}} \sup_{t'\in[t/2,t]}\big|  D^2u(x, t')\big|
     \cdot   \big[ D^3w(x, .\,)\big]_{C^{\frac{1+\gamma}{4}}([t/2,t])} .
\end{align*}
Therefore, we obtain for $m=1,2,3$
\begin{align*}
  t^{\frac{4-1}{4}} \big[ D^3w\, &D^2u(x,\, .\,)\big]_{C^{\frac\alpha4}([t/2,t])} \\
        \le &\  T^{\frac{2\gamma-\alpha}{4}}  \sup_{t'\in[t/2,t]}\big|t^{\frac{2-\gamma}4} D^3w(x, t')\big|
     \cdot t^{\frac{3}4} \big[D^2u(x, .\,)\big]_{C^{\frac{2+\gamma}{4}}([t/2,t])} \\
       &\ + T^{\frac{2\gamma-\alpha}{4}}\sup_{t'\in[t/2,t]}\big|t^{\frac{1-\gamma}4} D^2u(x, t')\big|
     \cdot  t^{\frac{3}4}\big[ D^3w(x, .\,)\big]_{C^{\frac{1+\gamma}{4}}([t/2,t])} ,
\\
  t^{\frac{4-2}{4}} \big[ D^3w\, & D^2u(x, .\,)\big]_{C^{\frac\alpha4}([t/2,t])}
  \\
   \le &\ T^{\frac{1+2\gamma-\alpha}{4}}  \sup_{t'\in[t/2,t]}\big|t^{\frac{1-\gamma}4} D^3w(x, t')\big|
     \cdot t^{\frac{2}4} \big[D^2u(x, .\,)\big]_{C^{\frac{2+\gamma}{4}}([t/2,t])} \\
       &\ + T^{\frac{1+\gamma-\alpha}{4}}\sup_{t'\in[t/2,t]}\big| D^2u(x, t')\big|
     \cdot  t^{\frac{2}4}\big[ D^3w(x, .\,)\big]_{C^{\frac{1+\gamma}{4}}([t/2,t])},
     \\
       t^{\frac{4-3}{4}} \big[ D^3w\, & D^2u(x, \,.\,)\big]_{C^{\frac\alpha 4}([t/2,t])} \\
        \le &\ T^{\frac{2+\gamma-\alpha}{4}}  \sup_{t'\in[t/2,t]}\big| D^3w(x, t')\big|
     \cdot t^{\frac{1}4} \big[D^2u(x, .\,)\big]_{C^{\frac{2+\gamma}{4}}([t/2,t])} \\
       &\ + T^{\frac{1+\gamma-\alpha}{4}}\sup_{t'\in[t/2,t]}\big| D^2u(x, t')\big|
     \cdot  t^{\frac{1}4}\big[ D^3w(x, .\,)\big]_{C^{\frac{1+\gamma}{4}}([t/2,t])} .
\end{align*}
Consequently,
\[
I_t^m(D^3w\,D^2u)\le C(\alpha,\gamma)\,\|D^3w\|_{\Cc^{1+\gamma,(1+\gamma)/4}_{3+\gamma-3}(Q_T)} \|D^2u\|_{\Cc^{2+\gamma,(2+\gamma)/4}_{3+\gamma-2}(Q_T)}.
\]
Combining Steps \cbm{2}-\cbm{4} yields
\[
\|D^3w\,D^2u\|_{\Cc^{\alpha,\alpha/4}_{m+\alpha-4}(Q_T)}
\le
\Cr{prodrule}\,AB,
\]
which proves \eqref{eq:32}.

\medskip\noindent\cb{5}\ \emph{Estimate of \(D^2u\,D^2w\,D^2v\).}
The proof of \eqref{eq:222} is completely analogous and in fact
slightly simpler, since only second derivatives occur.
Again, one decomposes the weighted norm into the three parts
\(I_0^m\), \(I_x^m\), and \(I_t^m\), applies the product rule, and uses the
weighted bounds for \(D^2u\), \(D^2w\), and \(D^2v\).
The additional powers of \(T\) that appear are nonnegative for
\(m=1,2,3\), and since \(T\le1\) they can be discarded.
Hence
\begin{align*}
    \|D^2u\,& D^2w\,D^2v\|_{\Cc^{\alpha,\alpha/4}_{m+\alpha-4}(Q_T)}
\\
&\le
\Cr{prodrule}\,\|D^2u\|_{\Cc^{2+\gamma,(2+\gamma)/4}_{m+\gamma-2}(Q_T)}\,\|D^2w\|_{\Cc^{2+\gamma,(2+\gamma)/4}_{m+\gamma-2}(Q_T)}\,\|D^2v\|_{\Cc^{2+\gamma,(2+\gamma)/4}_{m+\gamma-2}(Q_T)}.
\end{align*}

\medskip\noindent\cb{6}\ \emph{Estimate of \(\nabla u\).}
It remains to prove \eqref{eq:1-app}. We decompose 
\begin{align*}
    \|\nabla u\| _{C^{\alpha, \alpha/4}_{0}(Q_T)} 
    = &\ \sup_{t<T}
    t^{\alpha/4}\left( \sup_{t'\in[t/2,t]}\Big[\nabla u (\,,\,,t')\Big]_{C^\alpha(\overline\Omega)} 
    + \sup_{x\in\overline\Omega}\Big[ \nabla u(x,\,\cdot\,)\Big]_{C^{\alpha/4}\big([t/2,t]\big)}\right)\\
    &\ +\sup_{(x,t)\in\overline\Omega\times(0,T]} \big|\nabla u(x,t)\big|.
\end{align*}
where as in step \cbm 1 the norm consists of three parts
\begin{align*}
I_0(\nabla u)
&:=
\sup_{(x,t)\in Q_T}
\big|\nabla u(x,t)\big|,\\[0.2em]
I_x(\nabla u)
&:=
\sup_{0<t\le T}
t^{\frac{\alpha}{4}}
\sup_{t'\in[t/2,t]}
\big[\nabla u(\cdot,t')\big]_{\Cc^\alpha(\overline\Omega)},\\[0.2em]
I_t(\nabla u)
&:=
\sup_{0<t\le T}
t^{\frac{\alpha}{4}}
\sup_{x\in\overline\Omega}
[\nabla u(x,\cdot)]_{\Cc^{\alpha/4}([t/2,t])}.
\end{align*}
The \emph{pointwise term} $I_x(\nabla u)$ is immediate from the definition of the norm:
\[
\sup_{Q_T}|\nabla u|
\le
\|\nabla u\|_{\Cc^{3+\gamma,\frac{3+\gamma}{4}}_{m+\gamma-1}(Q_T)}.
\]
For the \emph{temporal Hölder seminorm }$I_t(\nabla u)$ we lower the exponent from
\((3+\gamma)/4\) to \(\alpha/4\) by \eqref{eq:i}, obtaining
\begin{align*}
t^{\alpha/4}
\big[\nabla u(x,\cdot)\big]_{\Cc^{\alpha/4}([t/2,t])}
&\le
t^{\alpha/4}
t^{\frac{3+\gamma-\alpha}{4}}
\big[\nabla u(x,\cdot)\big]_{\Cc^{\frac{3+\gamma}{4}}([t/2,t])} \\
&\le
T^{\frac{m-1+\gamma}{4}}
t^{\frac{4-m}{4}}
\big[\nabla u(x,\cdot)\big]_{\Cc^{\frac{3+\gamma}{4}}([t/2,t])}.
\end{align*}
Since \(T\le1\), this is bounded by the right-hand side of
\eqref{eq:1-app}.
For the \emph{spatial Hölder seminorm} $I_x(\nabla u)$ we use
\[
\big[\nabla u(\cdot,t')\big]_{\Cc^\alpha(\overline\Omega)}
\le
C\big\|D^2u(\cdot,t')\big\|_{\Cc^0(\overline\Omega)}^{\alpha}
\big\|\nabla u(\cdot,t')\big\|_{\Cc^0(\overline\Omega)}^{1-\alpha}.
\]
If \(m=2,3\), then \(\|D^2u\|_{\C^0(\overline\Omega)}\) is already controlled without
weight, and thus
\[
t^{\alpha/4}
\big[\nabla u(\cdot,t')\big]_{\Cc^\alpha(\overline\Omega)}
\le
C\,T^{\alpha/4}
\|u\|_{\Cc^{4+\gamma,1+\gamma/4}_{m+\gamma}(Q_T)}.
\]
If \(m=1\), then
\[
t^{\alpha/4}\big\|D^2u(\cdot,t')\big\|_{C^0(\overline\Omega)}^{\alpha}
=
t^{\frac{\alpha\gamma}{4}}
\big(t^{\frac{1-\gamma}{4}}\big\|D^2u(\cdot,t')\big\|_{C^0(\overline\Omega)}\big)^\alpha,
\]
and hence again
\[
t^{\alpha/4}
\big[\nabla u(\cdot,t')\big]_{\Cc^\alpha(\overline\Omega)}
\le
C(\alpha,\gamma)
\|u\|_{\Cc^{4+\gamma,1+\gamma/4}_{m+\gamma}(Q_T)}.
\]
This proves \eqref{eq:1-app} and completes the proof.
\end{proof}

\begin{lemma}[Hölder estimates I]
\label{gewichtethoedlerIIappendix}
Let $m=1,2,3,4$, let $0<\gamma\le\alpha<1$, assume $\gamma\ge \alpha/2$,
and let $T\le1$.
Then there exist constants
\[
\Cr{gewhoelderI}=\Cr{gewhoelderI}(\Omega,\alpha,\gamma)>0,
\qquad
k_H\in\N,
\]
depending only on $\Omega$, $\alpha$, $\gamma$, and the fixed algebraic
structure of $\Rr$ and $L$, such that for every
$u\in \Cc^{4+\gamma,1+\gamma/4}_{m+\gamma}(Q_T)$,
\begin{align*}
\|\Rr(\nabla u,D^2u,D^3u)\|_{\Cc^{\alpha,\alpha/4}_{m+\alpha-4}(Q_T)}
&\le
\Cr{gewhoelderI}
\Big(1+\|\nabla u\|_{\Cc^{3+\gamma,\frac{3+\gamma}{4}}_{m+\gamma-1}(Q_T)}\Big)^{k_H}
\|\nabla u\|_{\Cc^{3+\gamma,\frac{3+\gamma}{4}}_{m+\gamma-1}(Q_T)}^{3},\\
\sum_{k+\ell=4}
\|L_{k\ell}(\nabla u)\|_{\Cc^{\alpha,\alpha/4}_{\max\{0,m+\alpha-4\}}(Q_T)}
&\le
\Cr{gewhoelderI}
\Big(1+\|\nabla u\|_{\Cc^{3+\gamma,\frac{3+\gamma}{4}}_{m+\gamma-1}(Q_T)}^{4}\Big).
\end{align*}
\end{lemma}

\begin{proof}
First, we set
\[
X:=\|\nabla u\|_{\Cc^{3+\gamma,\frac{3+\gamma}{4}}_{m+\gamma-1}(Q_T)}.
\]
By \eqref{eq:einbett}, we then have the estimate
\[
\|D^2u\|_{\Cc^{2+\gamma,\frac{2+\gamma}{4}}_{m+\gamma-2}(Q_T)}
+
\|D^3u\|_{\Cc^{1+\gamma,\frac{1+\gamma}{4}}_{m+\gamma-3}(Q_T)}
\le C X.
\]

\medskip\noindent\cb{1}\ \emph{Structure of the remainder term.}
By \eqref{eq:R-structure}, every monomial in $\Rr$ is of one of the
following two types:
\[
D^3u\star D^2u \star Q^{-2k}P_{2k-1}(\nabla u),
\qquad 1\le k\le4,
\]
or
\[
D^2u\star D^2u\star D^2u \star Q^{-2(k+1)}P_{2k}(\nabla u),
\qquad 0\le k\le4.
\]
Hence, by the product estimate \eqref{eq:simpleprod} for weighted norms,
\begin{align*}
\|\Rr(\nabla u,&\, D^2u,D^3u)\|_{\Cc^{\alpha,\alpha/4}_{m+\alpha-4}(Q_T)}
\\
&\le
C\|D^3u\star D^2u\|_{\Cc^{\alpha,\alpha/4}_{m+\alpha-4}(Q_T)}
\sum_{k=1}^4
\|Q^{-2k}P_{2k-1}(\nabla u)\|_{\Cc^{\alpha,\alpha/4}_{\max\{0,m+\alpha-4\}}(Q_T)}
\\
&\quad
+
C\|D^2u\star D^2u\star D^2u\|_{\Cc^{\alpha,\alpha/4}_{m+\alpha-4}(Q_T)}
\sum_{k=0}^4
\|Q^{-2(k+1)}P_{2k}(\nabla u)\|_{\Cc^{\alpha,\alpha/4}_{\max\{0,m+\alpha-4\}}(Q_T)} .
\end{align*}

\medskip\noindent\cb{2}\ \emph{Derivative factors.}
By Lemma~\ref{productruleappendix}, equations \eqref{eq:32} and \eqref{eq:222},
together with the previous bound on $D^2u$ and $D^3u$, we obtain
\begin{align*}
\|D^3u\star D^2u\|_{\Cc^{\alpha,\alpha/4}_{m+\alpha-4}(Q_T)}
&\le C X^2,\\
\|D^2u\star D^2u\star D^2u\|_{\Cc^{\alpha,\alpha/4}_{m+\alpha-4}(Q_T)}
&\le C X^3.
\end{align*}

\medskip
\noindent\cb{3}\ \emph{Gradient-dependent coefficients.}
We next estimate the factors \(Q^{-2k}P_b(\nabla u)\).
Since \(P_b(\nabla u)\) is a polynomial of degree \(b\) in \(\nabla u\),
repeated use of \eqref{eq:simpleprod} and \eqref{eq:1} gives
\[
\|P_b(\nabla u)\|_{\Cc^{\alpha,\alpha/4}_{\max\{0,m+\alpha-4\}}(Q_T)}
\le C\cdot X^b.
\]
Moreover, the map
\[
z\mapsto (1+|z|^2)^{-k}
\]
is smooth on \(\R^2\), hence the standard composition estimate in Hölder
spaces yields
\[
\|Q^{-2k}\|_{\Cc^{\alpha,\alpha/4}_{\max\{0,m+\alpha-4\}}(Q_T)}
\le C(1+X)^{2k}.
\]
Combining the two estimates and applying \eqref{eq:simpleprod} once more,
we infer
\[
\|Q^{-2k}P_b(\nabla u)\|_{\Cc^{\alpha,\alpha/4}_{\max\{0,m+\alpha-4\}}(Q_T)}
\le C(1+X)^{2k}X^b.
\]

\medskip\noindent\cb{4}\ \emph{Estimate for \(\Rr\).}
Inserting the bounds from Steps \cbm{2} and \cbm{3} into the expansion from
Step \cbm{1}, we obtain
\begin{align*}
\|\Rr(\nabla u,D^2u,D^3u)\|_{\Cc^{\alpha,\alpha/4}_{m+\alpha-4}(Q_T)}
&\le
C X^{2}
\sum_{k=1}^4 (1+X)^{2k}X^{2k-1}
+
C X^{3}
\sum_{k=0}^4 (1+X)^{2k+1}X^{2k}.
\end{align*}
Hence, the right-hand side is bounded by
\[
C(1+X)^{k_H}X^3
\]
for some \(k_H\in\N\).
This proves the first estimate.

\medskip
\noindent\cb{5}\ \emph{Estimate for the coefficients \(L_{k\ell}(\nabla u)\).}
By the explicit representation \eqref{eq:L}, each coefficient \(L_{k\ell}\)
is a finite linear combination of terms of the form
\[
Q^{-4}P_b(\nabla u),
\qquad 0\le b\le4.
\]
Therefore, the estimate from Step \cbm{3} gives
\[
\|L_{k\ell}(\nabla u)\|_{\Cc^{\alpha,\alpha/4}_{\max\{0,m+\alpha-4\}}(Q_T)}
\le C(1+X)^4.
\]
Summing over \(k+\ell=4\) yields
\[
\sum_{k+\ell=4}
\|L_{k\ell}(\nabla u)\|_{\Cc^{\alpha,\alpha/4}_{\max\{0,m+\alpha-4\}}(Q_T)}
\le
\Cr{gewhoelderI}(1+X^4),
\]
which is the desired second estimate.
\end{proof}


\begin{lemma}\label{productrule22appendix}
Let $0<\alpha<1$ and let $T\le1$.
If
\[
u,v,w\in \Cc^{4+\alpha,1+\alpha/4}_{1}(Q_T),
\]
then there exists a constant
\[
\Cr{thmproductrule22}=\Cr{thmproductrule22}(\alpha,\Omega)>0
\]
such that
\begin{align*}
\|D^3w\,D^2u\|_{\Cc^{\alpha,\alpha/4}_{-3}(Q_T)}
&\le
\Cr{thmproductrule22}\,
\|D^3w\|_{\Cc^{1+\alpha,\frac{1+\alpha}{4}}_{-2}(Q_T)}
\|D^2u\|_{\Cc^{2+\alpha,\frac{2+\alpha}{4}}_{-1}(Q_T)},\\[0.3em]
\|D^2u\,D^2w\,D^2v\|_{\Cc^{\alpha,\alpha/4}_{-3}(Q_T)}
&\le
\Cr{thmproductrule22}\,
\|D^2u\|_{\Cc^{2+\alpha,\frac{2+\alpha}{4}}_{-1}(Q_T)}
\|D^2w\|_{\Cc^{2+\alpha,\frac{2+\alpha}{4}}_{-1}(Q_T)}
\|D^2v\|_{\Cc^{2+\alpha,\frac{2+\alpha}{4}}_{-1}(Q_T)},\\[0.3em]
\|\nabla u\|_{\Cc^{\alpha,\alpha/4}_{0}(Q_T)}
&\le
\Cr{thmproductrule22}\,
\|\nabla u\|_{\Cc^{3+\alpha,\frac{3+\alpha}{4}}_{0}(Q_T)}.
\end{align*}
\end{lemma}

\begin{proof}
\noindent\cb{1}\ \emph{Decomposition of the weighted norm.}
For a function \(h\in \Cc^{\alpha,\alpha/4}_{-3}(Q_T)\), the weighted norm is
\[
\|h\|_{\Cc^{\alpha,\alpha/4}_{-3}(Q_T)}
=
I_0(h)+I_x(h)+I_t(h),
\]
where we estimate three following contributions separately
\begin{align*}
I_0(h)
&:=
\sup_{(x,t)\in Q_T} t^{3/4}\big|h(x,t)\big|,\\[0.2em]
I_x(h)
&:=
\sup_{0<t\le T}
t^{\frac{3+\alpha}{4}}
\sup_{t'\in[t/2,t]}
\big[h(\cdot,t')\big]_{\Cc^\alpha(\overline\Omega)},\\[0.2em]
I_t(h)
&:=
\sup_{0<t\le T}
t^{\frac{3+\alpha}{4}}
\sup_{x\in\overline\Omega}
\big[h(x,\cdot)\big]_{\Cc^{\alpha/4}([t/2,t])}.
\end{align*}

\medskip\noindent\cb{2}\ \emph{Estimate of \(D^3w\,D^2u\).}
\emph{Pointwise term.}
By the definition of the weighted norms,
\begin{align*}
I_0(D^3w\,D^2u)
&=
\sup_{Q_T}
t^{3/4}|D^3w|\,|D^2u| \\
&\le
\sup_{Q_T} t^{2/4}|D^3w|
\cdot
\sup_{Q_T} t^{1/4}|D^2u|
\le \|D^3w\|_{\Cc^{1+\alpha,(1+\alpha)/4}_{-2}(Q_T)}
\|D^2u\|_{\Cc^{2+\alpha,\frac{2+\alpha}{4}}_{-1}(Q_T)}.
\end{align*}
\noindent
\emph{Spatial Hölder seminorm.}
For fixed \(t'\in[t/2,t]\), the product rule gives
\begin{align*}
&t^{\frac{3+\alpha}{4}}\big[D^3w\,D^2u(\cdot,t')\big]_{\Cc^\alpha(\overline\Omega)}
\\
&\quad \le
t^{\frac{3+\alpha}{4}} 
\|D^3w(\cdot,t')\|_{C^0(\overline\Omega)}
[D^2u(\cdot,t')]_{\Cc^\alpha(\overline\Omega)}
\\
&\qquad
+ t^{\frac{3+\alpha}{4}}
[D^3w(\cdot,t')]_{\Cc^\alpha(\overline\Omega)}
\|D^2u(\cdot,t')\|_{C^0(\overline\Omega)}
\\
&\quad\le
C\,
\big(t^{2/4}\|D^3w(\cdot,t')\|_{C^0(\overline\Omega)}\big)
\big(t^{2/4}\|D^3u(\cdot,t')\|_{C^0(\overline\Omega)}\big)^\alpha
\big(t^{1/4}\|D^2u(\cdot,t')\|_{C^0(\overline\Omega)}\big)^{1-\alpha}
\\
&\qquad
+
C\,
\big(t^{3/4}\|D^4w(\cdot,t')\|_{C^0(\overline\Omega)}\big)^\alpha
\big(t^{2/4}\|D^3w(\cdot,t')\|_{C^0(\overline\Omega)}\big)^{1-\alpha}
\big(t^{1/4}\|D^2u(\cdot,t')\|_{C^0(\overline\Omega)}\big)
\\
&\quad\le C\,\|D^3w\|_{\Cc^{1+\alpha,(1+\alpha)/4}_{-2}(Q_T)}
\|D^2u\|_{\Cc^{2+\alpha,\frac{2+\alpha}{4}}_{-1}(Q_T)}.
\end{align*}
Taking the supremum in \(t'\) and \(t\) gives
\(
I_x(D^3w\,D^2u)\le C\,\|D^3w\|_{\Cc^{1+\alpha,(1+\alpha)/4}_{-2}(Q_T)}
\|D^2u\|_{\Cc^{2+\alpha,\frac{2+\alpha}{4}}_{-1}(Q_T)}.
\)

\smallskip
\noindent
\emph{Temporal Hölder seminorm.}
For fixed \(x\in\overline\Omega\),
lowering the temporal Hölder exponents,
we infer
\begin{align*}
&t^{\frac{3+\alpha}{4}}[D^3w\,D^2u(x,\cdot)]_{\Cc^{\alpha/4}([t/2,t])}
\\
&\quad\le
t^{\frac{3+\alpha}{4}}\sup_{t'\in[t/2,t]}|D^3w(x,t')|\,
[D^2u(x,\cdot)]_{\Cc^{\alpha/4}([t/2,t])}
\\
&\qquad
+
t^{\frac{3+\alpha}{4}}\sup_{t'\in[t/2,t]}|D^2u(x,t')|\,
[D^3w(x,\cdot)]_{\Cc^{\alpha/4}([t/2,t])}
\\
&\quad\le
t^{\frac{3+\alpha}{4}}\sup_{t'\in[t/2,t]}|D^3w(x,t')|\,
t^{1/2}[D^2u(x,\cdot)]_{\Cc^{(2+\alpha)/4}([t/2,t])}
\\
&\qquad
+
t^{\frac{3+\alpha}{4}}\sup_{t'\in[t/2,t]}|D^2u(x,t')|\,
t^{1/4}[D^3w(x,\cdot)]_{\Cc^{(1+\alpha)/4}([t/2,t])}
\\
&\quad\le
\big(t^{2/4}\sup_{t'\in[t/2,t]}|D^3w(x,t')|\big)
\big(t^{\frac{3+\alpha}{4}}[D^2u(x,\cdot)]_{\Cc^{(2+\alpha)/4}([t/2,t])}\big)
\\
&\qquad
+
\big(t^{1/4}\sup_{t'\in[t/2,t]}|D^2u(x,t')|\big)
\big(t^{\frac{3+\alpha}{4}}[D^3w(x,\cdot)]_{\Cc^{(1+\alpha)/4}([t/2,t])}\big)
\\
&\quad\le
C\,\|D^3w\|_{\Cc^{1+\alpha,(1+\alpha)/4}_{-2}(Q_T)}
\|D^2u\|_{\Cc^{2+\alpha,\frac{2+\alpha}{4}}_{-1}(Q_T)}.
\end{align*}
Hence \(I_t(D^3w\, D^2u)\le C\,\|D^3w\|_{\Cc^{1+\alpha,(1+\alpha)/4}_{-2}(Q_T)}
\|D^2u\|_{\Cc^{2+\alpha,\frac{2+\alpha}{4}}_{-1}(Q_T)}\).
Combining the three parts, we obtain
\[
\|D^3w\,D^2u\|_{\Cc^{\alpha,\alpha/4}_{-3}(Q_T)}
\le
\Cr{thmproductrule22}\,\|D^3w\|_{\Cc^{1+\alpha,(1+\alpha)/4}_{-2}(Q_T)}
\|D^2u\|_{\Cc^{2+\alpha,\frac{2+\alpha}{4}}_{-1}(Q_T)}.
\]

\medskip
\noindent
\cb{3}\ \emph{Estimate of \(D^2u\,D^2w\,D^2v\).}
The proof is completely analogous.
One again decomposes the norm into \(I_0\), \(I_x\), and \(I_t\), applies
the product rule, and estimates the spatial Hölder seminorms by
\[
[D^2u(\cdot,t')]_{\Cc^\alpha}
\le
C\,\|D^3u(\cdot,t')\|_{C^0}^{\alpha}
\|D^2u(\cdot,t')\|_{C^0}^{1-\alpha},
\]
and similarly for \(D^2w\) and \(D^2v\).
The weights match exactly, and therefore
\[
\|D^2u\,D^2w\,D^2v\|_{\Cc^{\alpha,\alpha/4}_{-3}(Q_T)}
\le
\Cr{thmproductrule22}\,\|D^3w\|_{\Cc^{1+\alpha,(1+\alpha)/4}_{-2}(Q_T)}
\|D^2u\|_{\Cc^{2+\alpha,\frac{2+\alpha}{4}}_{-1}(Q_T)}
\|D^2v\|_{\Cc^{2+\alpha,\frac{2+\alpha}{4}}_{-1}(Q_T)}.
\]

\medskip
\noindent
\cb{4}\ \emph{Estimate of \(\nabla u\).}
We write
\begin{align*}
J_0&:=\sup_{Q_T}|\nabla u|,\\
\|\nabla u\|_{\Cc^{\alpha,\alpha/4}_{0}(Q_T)}
=
J_0+J_x+J_t \quad \text{with} \quad J_x&:=\sup_{0<t\le T}t^{\alpha/4}\sup_{t'\in[t/2,t]}
[\nabla u(\cdot,t')]_{\Cc^\alpha(\overline\Omega)},\\
J_t&:=\sup_{0<t\le T}t^{\alpha/4}\sup_{x\in\overline\Omega}
[\nabla u(x,\cdot)]_{\Cc^{\alpha/4}([t/2,t])}.
\end{align*}
The \emph{pointwise term} $J_0$ is immediate estimated.
For the \emph{temporal Hölder seminorm} we lower the exponent from
\((3+\alpha)/4\) to \(\alpha/4\):
\begin{align*}
t^{\alpha/4}[\nabla u(x,\cdot)]_{\Cc^{\alpha/4}([t/2,t])}
&\le
t^{\alpha/4}\,t^{3/4}
[\nabla u(x,\cdot)]_{\Cc^{(3+\alpha)/4}([t/2,t])}\\
&\le
t^{\frac{3+\alpha}{4}}
[\nabla u(x,\cdot)]_{\Cc^{(3+\alpha)/4}([t/2,t])},
\end{align*}
hence
\(
J_t\le
\|\nabla u\|_{\Cc^{3+\alpha,\frac{3+\alpha}{4}}_{0}(Q_T)}.
\)
For the \textit{spatial Hölder seminorm} $J_x$ we obtain
\begin{align*}
t^{\alpha/4}[\nabla u(\cdot,t')]_{\Cc^\alpha(\overline\Omega)}
&\le
C\,
\big(t^{1/4}\|D^2u(\cdot,t')\|_{C^0}\big)^\alpha
\|\nabla u(\cdot,t')\|_{C^0}^{1-\alpha}
\\
&\le
C\,
\|\nabla u\|_{\Cc^{3+\alpha,\frac{3+\alpha}{4}}_{0}(Q_T)}.
\end{align*}
Thus \(J_x\) is bounded by the same quantity.
Combining \(J_0\), \(J_x\), and \(J_t\) proves the claim.
\end{proof}




\bibliographystyle{alphaurl}
\bibliography{lib}


\end{document}